\documentclass[11pt]{article}
\usepackage{bbm}
\usepackage{mathrsfs}
\usepackage{amsfonts}
\usepackage{amssymb}
\usepackage{amssymb,amsmath,amsthm, amsfonts}

\def\sqr#1#2{{\vcenter{\vbox{\hrule height.#2pt
              \hbox{\vrule width.#2pt height#1pt \kern#1pt \vrule width.#2pt}
          \hrule height.#2pt}}}}
\def\sqr#1#2{{\vcenter{\vbox{\hrule height.#2pt
              \hbox{\vrule width.#2pt height#1pt \kern#1pt \vrule width.#2pt}
              \hrule height.#2pt}}}}
\def\3n{\negthinspace \negthinspace \negthinspace }
\def\2n{\negthinspace \negthinspace }
\def\1n{\negthinspace }

\def\={\buildrel \triangle \over =}

\def\sup{\mathop{\rm sup}}

\def\inf{\hbox{\rm inf$\,$}}

\def\|{\Big |}
\def\({\Big (}
\def\){\Big )}
\def\[{\Big[}
\def\]{\Big]}
\def\be{\begin{equation}}
\def\bel{\begin{equation}\label}
\def\ee{\end{equation}}
\def\bt{\begin{theorem}}
\def\bcd{\begin{condition}}
\def\ecd{\end{condition}}
\def\et{\end{theorem}}
\def\bc{\begin{corollary}}
\def\ec{\end{corollary}}
\def\bde{\begin{definition}}
\def\ede{\end{definition}}
\def\bl{\begin{lemma}}
\def\el{\end{lemma}}
\def\bp{\begin{proposition}}
\def\ep{\end{proposition}}
\def\br{\begin{remark}}
\def\er{\end{remark}}
\def\ba{\begin{array}}
\def\ea{\end{array}}
\def\ed{\end{document}}

\def\square#1{\vbox{\hrule\hbox{\vrule height#1%
     \kern#1\vrule}\hrule}}
\def\rectangle#1#2{\vbox{\hrule\hbox{\vrule height#1%
     \kern#2\vrule}\hrule}}

\font\tenbb=msbm10 \font\sevenbb=msbm7 \font\fivebb=msbm5

\newfam\bbfam
\scriptscriptfont\bbfam=\fivebb \textfont\bbfam=\tenbb
\scriptfont\bbfam=\sevenbb

\newtheorem{lemma}{Lemma}[section]
\newtheorem{remark}{Remark}[section]
\newtheorem{example}{Example}[section]
\newtheorem{theorem}{Theorem}[section]
\newtheorem{corollary}{Corollary}[section]

\newtheorem{definition}{Definition}[section]
\newtheorem{proposition}{Proposition}[section]
\newtheorem{condition}{Condition}[section]

\makeatletter
   
   \@addtoreset{equation}{section}
\makeatother

\begin{document}

\title{Mean-field stochastic differential equations and associated PDEs}
\author{Rainer Buckdahn$^{1, 3}$,\, \,  Juan Li$^{2}$\footnote{The author has
been supported by the NSF of P.R.China (Nos. 11171187, 11222110), Shandong Province (Nos. BS2011SF010, JQ201202), 111 Project (No. B12023).},\, \,  Shige Peng$^{3}$,\, \, Catherine Rainer$^{1}$ \\
{\small $^1$Laboratoire de Math\'{e}matiques LMBA, CNRS-UMR 6205, Universit\'{e} de
Bretagne Occidentale,}\\
 {\small 6, avenue Victor-le-Gorgeu, CS 93837, 29238 Brest cedex 3, France.}\\
{\small $^2$School of Mathematics and Statistics, Shandong University (Weihai), Weihai 264209, P. R. China.;}\\
{\small $^3$School of Mathematics, Shandong University, Jinan 250100, P. R. China.}\\
{\small{\it E-mails: rainer.buckdahn@univ-brest.fr, juanli@sdu.edu.cn, peng@sdu.edu.cn, catherine.rainer@univ-brest.fr.}}
\date{June 7, 2014} }

\maketitle \noindent{\bf Abstract} In this paper we consider a mean-field stochastic differential equation, also called Mc Kean-Vlasov equation, with initial data $(t,x)\in[0,T]\times R^d,$ which coefficients depend on both the solution $X^{t,x}_s$ but also its law. By considering square integrable random variables $\xi$ as initial condition for this equation, we can easily show the flow property of the solution $X^{t,\xi}_s$ of this new equation. Associating it with a process $X^{t,x,P_\xi}_s$ which coincides with $X^{t,\xi}_s$, when one substitutes $\xi$ for $x$, but which has the advantage to depend only on the law $P_\xi$ of $\xi$, we characterise the function $V(t,x,P_\xi)=E[\Phi(X^{t,x,P_\xi}_T,P_{X^{t,\xi}_T})]$ under appropriate regularity conditions on the coefficients of the stochastic differential equation as the unique classical solution of a non local PDE of mean-field type, involving the first and second order derivatives of $V$ with respect to its space variable and the probability law. The proof bases heavily on a preliminary study of the first and second order derivatives of the solution of the mean-field stochastic differential equation with respect to the probability law and a corresponding It\^{o} formula. In our approach we use the notion of derivative with respect to a square integrable probability measure introduced in \cite{PL} and we extend it in a direct way to second order derivatives.

\bigskip

 \noindent{{\bf Mathematics Subject Classification:} primary: 60H10; secondary: 60K35.}\\
{{\bf Keywords:}\small \  Mean-field stochastic differential equation; McKean-Vlasov equation; value function; PDE of mean-field type.}

\section{\large{Introduction}}

Given a complete probability space $(\Omega,{\cal F},P)$ endowed with a Brownian motion $B=(B_t)_{t\in[0,T]}$ and its filtration $\mathbb{F}=({\cal F}_{t\in[0,T]}$ augmented by all $P$-null sets and a sufficiently rich sub-$\sigma$-algebra, we consider the mean-field stochastic differential equation (SDE), also known under the name McKean-Vlasov SDE,

\smallskip

\centerline{$\displaystyle dX_s^{t,x}=\sigma(X_s^{t,x},P_{X_s^{t,x}})dB_s+b(X_s^{t,x},P_{X_s^{t,x}})ds,\, s\in[t,T],\ X^{t,x}_t=x\in R^d.$}

\smallskip

\noindent It is well-known that under an appropriate Lipschitz assumption on the coefficients this equation possesses for all $(t,x)\in[t,T]\times R^d$ a unique solution $X_s^{t,x},\ s\in[0,T].$ For the classical SDE which coefficients $\sigma(x,\mu)=\sigma(x),\ b(x,\mu)=b(x)$ depend only on $x\in R^d$ but not on the probability measure $\mu$, it is well known that the solution $X_s^{t,x},\ 0\le t\le s\le T, x\in R^d$, defines a flow and, if the coefficients are regular enough, the unique classical solution of the partial differential equation (PDE)

\smallskip

$\displaystyle \partial_t V(t,x)+\frac12\mbox{tr}\left(\sigma\sigma^*(x)D^2_xV(t,x)\right)+b(x)D_xV(t,x)=0,\ (t,x)\in [0,T]\times R^d,$

$\displaystyle V(T,x)=\Phi(x),\ x\in R^d,$

\smallskip

\noindent is $V(t,x)=E[\Phi(X_T^{t,x})],\, (t,x)\in[0,T]\times R^d.$ But how about the above SDE which coefficients depend on $(x,\mu)\in R^d\times{\cal P}_2(R^d)$, where ${\cal P}_2(R^d)$ denotes the space of the square integrable probability measures over $R^d$? Of course,
for an SDE with coefficients depending on $(x,\mu)$ the solution $X^{t,x}_s,\ 0\le s\le t\le T,\ x\in R^d$, does obviously not define a flow. But we see easily that, if we replace the deterministic initial condition $X^{t,x}_t=x\in R^d$ by a square integrable random variable $X^{t,\xi}_t=\xi\in L^2({\cal F}_t;R^d)(:=L^2(\Omega,{\cal F}_t,P;R^d))$ and consider the SDE

\smallskip

\centerline{$\displaystyle dX_s^{t,\xi}=\sigma(X_s^{t,\xi},P_{X_s^{t,\xi}})dB_s+b(X_s^{t,\xi},P_{X_s^{t,\xi}})ds,\, s\in[t,T],\ X^{t,\xi}_t=\xi\in R^d$}

\noindent (where,obviously, in general, $X^{t,\xi}\not={X^{t,x}}_{|x=\xi}$), then we have the flow property: For all $0\le t\le s\le T,\ \xi\in L^2(\Omega,{\cal F}_t,P;R^d)$, $X_r^{s,\eta}=X_r^{t,\xi},\ r\in[s,T],$ for $\eta=X_s^{t,\xi}.$ This flow proporty should give rise to a PDE with a solution $V(t,\xi)=E[\Phi(X^{t,\xi}_T,P_{X^{t,\xi}_T})],$ but the fact that $\xi$ has to belong to $L^2({\cal F}_t;R^d)$ has the consequence that $V(t,\xi)$ is defined over a Hilbert space depending on $t$, which makes such PDE difficult to handle. As alternative we associate with the above SDE for $X^{t,\xi}$ the SDE

\smallskip

\centerline{$\displaystyle dX_s^{t,x,\xi}=\sigma(X_s^{t,x,\xi},P_{X_s^{t,\xi}})dB_s+b(X_s^{t,x,\xi},P_{X_s^{t,\xi}})ds,\, s\in[t,T],\ X^{t,x,\xi}_t=x\in R^d.$}

\smallskip

\noindent It turns out (cf. Lemma 3.1) that $X_s^{t,x,P_\xi}=X_s^{t,x,\xi},\ s\in[t,T],$ depends on $\xi\in L^2({\cal F}_t;R^d)$ only through its law $P_\xi,$ $X_s^{t,\xi}={X_s^{t,x,P_\xi}}_{|x=\xi}$, and $\left(X_s^{t,x,P_\xi},X^{t,\xi}_{s}\right),\ 0\le t\le s\le T,\ \xi\in  L^2({\cal F}_t;R^d),$ has the flow property.

The objective of our manuscript is to study under appropriate regularity assumptions on the coefficients the second order PDE which is associated with this stochastic flow, i.e., the PDE which unique classical solution is given by the function

\smallskip

\centerline{$\displaystyle V(t,x,P_\xi)=E\left[\left(\Phi(X^{t,x,P_\xi}_T,P_{X^{t,\xi}_T}\right)\right],\, (t,x)\in[0,T]\times R^d,\
\xi\in L^2({\cal F}_t;R^d).$}

\smallskip

\noindent The function $V$ is defined over $[0,T]\times R^d\times{\cal P}_2(R^d)$, and so the study of the first and second order derivatives with respect to the probability measure will play a crucial role. In our work we have based ourselves on the notion of derivative of a function $f:{\cal P}_2(R^d)\rightarrow R$ with respect to the probability measure $\mu$, which was studied by P.-L.Lions in his course at {\it Coll\`{e}ge de France} \cite{PL}. The derivative of $f$ with respect to $\mu$ is a function $\partial_\mu f:{\cal P}_2(R^d)\times R^d\rightarrow R^d$ (cf. Section 2. Preliminaries). The main result of our work says that, if the coefficients $b$ and $\sigma$ are twice differentiable in $(x,\mu)$ with bounded Lipschitz derivatives of first and second order, then the function $V(t,x,P_\xi)$ defined above is the unique classical solution of the following non local PDE of mean-field type (cf. Theorem 5.2):

\smallskip

\noindent\hskip 2mm$\displaystyle 0=\partial_tV(t,x,P_\xi)+\sum_{i=1}^d\partial_{x_i}
V(t,x,P_\xi)b_{i}(x,P_\xi)\displaystyle+\frac12\sum_{i,j,k=1}^d\partial_{x_i x_j}^2V(t,x,P_\xi)(\sigma_{i,k}
\sigma_{j,k})(x,P_\xi)$

\vskip -4mm\noindent\hskip 8mm$\displaystyle +E\big[\sum_{i=1}^d(\partial_\mu V)_i(t,x,P_\xi,\xi) b_{i}(\xi,P_\xi)+ \frac12\sum_{i,j,k=1}^d\partial_{y_i}(\partial_\mu V)_j(t,x,P_\xi,\xi)(\sigma_{i,k}\sigma_{j,k})(\xi,P_\xi)\big],$

\noindent\hskip 2mm$\displaystyle V(T,x,P_\xi)=\Phi(x,P_\xi),$

\smallskip

\noindent with $(t,x,P_\xi)$ running $[0,T]\times {R}^d\times {\cal P}_2(R^d)$. We see, in particular, that, in contrast to the classical case, the derivative $\partial_\mu V(t,x,P_\xi,y)$ and, as second order derivative, the derivative of $\partial_\mu V(t,x,P_\xi,y)$ with respect to $y$ are involved.

Mean-field SDEs, also known as McKean-Vlasov equations, were discussed the first time by Kac \cite{Kac1}, \cite{Kac2} in the frame of his study of the Boltzman equation for the particle density in diluted monatomic gases as well as in that of the stochastic toy model for the Vlasov kinetic equation of plasma. A by now classical method of solving mean-field SDEs by approximation consists in the use of
so called $N$-particle systems with weak interaction, formed by $N$ equations driven by independent Brownian motions.
The convergence of this system to the mean-field SDE is called in the literature propagation of chaos for the McKean-Vlasov equation.

The pioneering works by Kac, and in the aftermath by other authors, have attracted a lot of researchers interested in the study of the chaos propagation and the limit equations in different frame works; for an impression we refer the reader, for instance, to \cite{BRTV}, \cite{BT}, \cite{DG}, \cite{C}, \cite{K}, \cite{M}, \cite{O}, \cite{PH} ,\cite{S1} and \cite{S2} as well as the references therein. With the pioneering works on mean-field stochastic differential games by Lasry and Lions (We refer to \cite{LL} and the papers cited therein, but also to \cite{PL}), new impulses and new applications for mean-field problems were given.
So recently Buckdahn, Djehiche, Li and Peng~\cite{BLP1} studied
a special mean field problem by a purely stochastic approach and deduced a new kind of backward SDE (BSDE) which they called mean-field BSDE; they showed that the BSDE can be obtained by an approximation involving $N$-particle systems with weak interaction. They completed these studies of the approximation with associating a kind of Central Limit Theorem for the approximating systems and obtained as limit some forward-backward SDE of mean-field type, which is not only governed by a Brownian motion but also by an independent Gaussian
field. In ~\cite{BLP}, deepening the investigation of mean-field SDEs and associated mean-field BSDEs,  Buckdahn, Li and Peng generalised their previous results on mean-field BSDEs, and in a ``Markovian'' framework in which the initial data $(t,x)$ were frozen in the law variable of the coefficients, they investigated the associated non local PDE. However, our objective has been to overcome this partial freezing of initial data in the mean-field SDE and to study the associated PDE, and this is done in our present manuscript.
Our approach is highly inspired by the courses given by P.-L.Lions \cite{PL} at {\it Coll\`{e}ge de France} (redacted by P.Cardaliaguet) and by recent works of R.Carmona and F.Delarue , who, directly inspired by the works of J.M.Lasry P.-L.Lions \cite{LL} and the courses of P.-L.Lions, translated his rather analytical approach into a stochastic one; let us cite \cite{CD}, \cite{CD1} and the refences indicated therein.

Our manuscript is organised as follows: In Section 2 ``Preliminaries'' we introduce the framework of our study. A particular attention is paid to a recall of the notion of the derivative of a function defined over the space of square integrable probability measures over $R^d$. On the basis of this notion of first order derivative we introduce in the same spirit the second order derivatives of such a function. The first and the second order differentiability of a function with respect to the probability measure allows in the following to derive a second order Taylor expansion which turns out to be crucial in our approach. Section 2 finishes with the discussion of an example. In Section 3 we introduce our mean-field SDE with the standard assumptions on its coefficients (their twicefold differentiability with respect to $(x,\mu)$ with bounded Lipschitz derivatives of first and second order), and we study useful properties of the mean-field SDE. A central property studied in this section is the differentiability of the solution process $X^{t,x,P_\xi}$ with respect to the probability law $P_\xi.$ These investigations are completed by Section 3, which is devoted to the study of the second order derivatives of $X^{t,x,P_\xi}$, and so namely for that with respect to the probability law. The first and the second order derivatives of $X^{t,x,P_\xi}$ are characterised as the unique solution of associated SDEs which on their part allow to get estimates for the derivatives of order 1 and 2 of $X^{t,x,P_\xi}$. The results obtained for the process $X^{t,x,P_\xi}$ and so also for $X^{t,\xi}$ in the Sections 2 and 3 are used for the proof of the regularity of the value function $V(t,x,P_\xi)$. Finally, Section 6 is devoted to an It\^{o} formula associated with mean-field problems and it gives our main result, Theorem 5.2, stating that our value function $V$ is the unique classical solution of the PDE of mean-field type given above.

\section{ {\large Preliminaries}}

Let us begin with introducing some notations and concepts, which we
will need in our further computations. We shall in particular
introduce the notion of differentiability of a function $f$ defined
over the space ${\cal P}_2({\mathbb R}^d)$ of all square integrable
probability measures $\mu$ over $({\mathbb R}^d,{\cal B}({\mathbb R}^d))$, where ${\cal B}({\mathbb R}^d)$
denotes the Borel $\sigma$-field over ${\mathbb R}^d$; the space ${\cal P}_2({\mathbb R}^{2d})$
is endowed with the $2$-Wasserstein metric

\smallskip

\be\label{2.1} W_2(\mu,\nu):=\inf\big\{\left(\int_{{\mathbb R}^d\times
{\mathbb R}^d}|x-y|^2\rho(dxdy)\right)^{1/2},\, \rho\in{\cal P}_2({\mathbb R}^{2d})
\mbox{ with }\rho(.\times {\mathbb R}^d)=\mu,\rho({\mathbb R}^d\times.) =\nu\big\},\ee

\smallskip

\noindent $\mu,\ \nu\in{\cal P}_2({\mathbb R}^d).$ Among the different notions of
differentiability of a function $f$ defined over ${\cal P}_2({\mathbb R}^d)$ we
adopt for our approach that introduced by Lions in his lectures at
{\it Coll\`{e}ge de France} in Paris and revised in the notes by Cardaliaguet \cite{PL}; we refer the reader also, for instance, to Carmona
and Delarue \cite{CD}. Let us consider a probability space $(\Omega,{\cal F},P)$
which is ``rich enough'' (The precise space we will work with will be
introduced later). ``Rich enough'' means that for every
$\mu\in {\cal P}_2({\mathbb R}^d)$ there is a random variable $\vartheta\in L^2({\cal
F};{\mathbb R}^d)(:=L^2(\Omega,{\cal F},P;{\mathbb R}^d))$ such that $P_\vartheta=\mu.$\ It is well-known that the probability space $([0,1],{\cal
B}([0,1]),dx)$ has this property.

Identifying the random variables in $L^2({\cal F};{\mathbb R}^d)$, which
coincide $P$-a.e., we can regard $L^2({\cal F};{\mathbb R}^d)$ as a Hilbert
space with inner product $(\xi,\eta)_{L^2}=E[\xi\cdot\eta],\ \xi,\ \eta\in
L^2({\cal F}; {\mathbb R}^d)$, and norm $|\xi|_{L^2}=(\xi,\xi)_{L^2}^{\frac{1}{2}}$. Recall that, due to the definition made by Lions \cite{PL}
(see Cardaliaguet \cite{PC}), a function
$f:{\cal P}_2({\mathbb R}^d)\rightarrow
{\mathbb R}$ is said to be differentiable in $\mu\in {\cal P}_2({\mathbb R}^d)$ if, for
$\widetilde{f}(\vartheta):=f(P_\vartheta),\ \vartheta\in L^2({\cal
F})$, there is some $\vartheta_0\in L^2({\cal F})$ with
$P_{\vartheta_0}=\mu$, such that the function $\widetilde{f}: L^2({\cal F};{\mathbb R}^d)\rightarrow {\mathbb R}$ is differentiable (in Fr\'{e}chet
sense) in $\vartheta_0$, i.e., there exists a linear continuous mapping
$D\widetilde{f}(\vartheta_0): L^2({\cal F};{\mathbb R}^d)\rightarrow {\mathbb R}$
($D\widetilde{f}(\vartheta_0)\in L(L^2({\cal F};{\mathbb R}^d);{\mathbb R})$) such that
\be\label{2.2}\widetilde{f}(\vartheta_0+\eta)-\widetilde{f}(\vartheta_0)=D\widetilde{f}(\vartheta_0)(\eta)+
o(|\eta|_{L^2}), \ee
with $|\eta|_{L^2}\rightarrow 0$\ for $\eta\in L^2({\cal
F};{\mathbb R}^d)$. Since
$D\widetilde{f}(\vartheta_0)\in L(L^2({\cal F};{\mathbb R}^d);{\mathbb R})$, Riesz'
Representation Theorem yields the existence of a ($P$-a.s.) unique random
variable $\theta_0\in L^2({\cal F};{\mathbb R}^d)$ such that
$D\widetilde{f}(\vartheta_0)(\eta)=(\theta_0,\eta)_{L^2}=E\left[\theta_0
\eta\right],$ for all
$\eta\in L^2({\cal F};{\mathbb R}^d)$. In \cite{PL} it has
been proved that there is a Borel function $h_0:{\mathbb R}^d\rightarrow {\mathbb R}^d$
such that $\theta_0=h_0(\vartheta_0),\ P$-a.e. Taking into account
the definition of $\widetilde{f}$, this allows to write
\be\label{2.3}f(P_\vartheta)-f(P_{\vartheta_0})=E[h_0(\vartheta_0)\cdot
(\vartheta-\vartheta_0)]+o(|\vartheta-\vartheta_0|_{L^2}),\ee
$\vartheta\in L^2({\cal F},{\mathbb R}^d)$.

We call $\partial_\mu f(P_{\vartheta_0},y):=h_0(y),\ y\in {\mathbb R}^d,$ the
derivative of $f: {\cal P}_2({\mathbb R}^d)\rightarrow {\mathbb R}$ at $P_{\vartheta_0}.$
Note that $\partial_\mu f(P_{\vartheta_0},y)$ is only
$P_{\vartheta_0}(dy)$-a.e. uniquely determined.

However, in our approach we have to consider functions $f:{\cal
P}_2({\mathbb R}^d)\rightarrow {\mathbb R}$ which are differentiable in all elements of
${\cal P}_2({\mathbb R}^d)$. In order to simplify the argument, we suppose
that $\widetilde{f}: L^2({\cal F},{\mathbb R}^d)\rightarrow {\mathbb R}$ is Fr\'{e}chet differential
over the whole space $L^2({\cal F},{\mathbb R}^d)$. This corresponds to a large
class of important examples. In this case we have the
derivative $\partial_\mu f(P_{\vartheta},y),$ defined
$P_{\vartheta}(dy)$-a.e., for all $\vartheta\in L^2({\cal F},{\mathbb R}^d).$\ In Lemma 3.2 \cite{CD} it is shown that, if, furthermore, the Fr\'{e}chet derivative $D\widetilde{f}: L^2({\cal F},{\mathbb R}^d)\rightarrow L(L^2({\cal F},{\mathbb R}^d), {\mathbb R})$\ is Lipschitz continuous (with a Lipschitz constant $K \in {\mathbb R}_+$), then there is for all $\vartheta\in L^2({\cal F},{\mathbb R}^d)$\ a $P_\vartheta$-version of $\partial_\mu f(P_\vartheta, .): {\mathbb R}^d\rightarrow {\mathbb R}^d$\ such that
$$|\partial_\mu f(P_\vartheta, y)-\partial_\mu f(P_\vartheta, y')|\leq K|y-y'|,\ \mbox{for all}\ y, y'\in {\mathbb R}^d.$$

\noindent This motivates us to make the following definition.

\begin{definition} We say that $f\in C^{1,1}_b
({\cal P}_2({\mathbb R}^d))$ (continuously differentiable over ${\cal P}_2({\mathbb R}^d)$
with Lipschitz-continuous bounded derivative), if there exists for all $\vartheta\in
L^2({\cal F},{\mathbb R}^d)$  a $P_\vartheta$-modification of $\partial_\mu
f(P_\vartheta,.),$ again denoted by $\partial_\mu f(P_\vartheta,.)$,
such that $\partial_\mu f:{\cal P}_2({\mathbb R}^d)\times {\mathbb R}^d\rightarrow {\mathbb R}^d$
is bounded and Lipschitz continuous, i.e., there is some real
constant $C$ such that

${\rm i)}\hskip 0.4cm |\partial_\mu f(\mu,x)|\le C,\, \mu\in{\cal P}_2({\mathbb R}^d),
\ x\in {\mathbb R}^d,$

${\rm ii)}\hskip 0.3cm |\partial_\mu f(\mu,x)-\partial_\mu f(\mu',x')|
\le C(W_2(\mu,\mu')+|x-x'|),\ \mu,\ \mu'\in{\cal P}_2({\mathbb R}^d),\ x,\ x'\in {\mathbb R}^d;$

\smallskip

\noindent we consider this function $\partial_\mu f$ as the derivative
of $f$.\end{definition}

\br Let us point out that, if $f\in C_b^{1,1}({\cal P}_2({\mathbb R}^d))$, the version of $\partial_\mu f(P_{\vartheta},.)$, $\vartheta\in L^2({\cal F},{\mathbb R}^d)$, indicated in Definition 2.1 is unique. Indeed, given $\vartheta\in L^2({\cal F};{\mathbb R}^d)$, let $\theta$\ be a d-dimensional vector of independent standard normally distributed random variables, which are independent of $\vartheta$. Then, since $\partial_\mu f(P_{\vartheta+\varepsilon \theta}, \vartheta+\varepsilon \theta)$\ is P-a.s. defined, $\partial_\mu f(P_{\vartheta+\varepsilon \theta}, y)$ is defined $dy$-a.e. From the Lipschitz continuity ii) of $\partial_\mu f$ in Definition 2.1 it then follows that $\partial_\mu f(P_{\vartheta+\varepsilon \theta}, y)$ is defined for all $y\in {\mathbb R}^d$, and taking the limit $0<\varepsilon\downarrow 0$\ yields that $\partial_\mu f(P_{\vartheta}, y)$\ is uniquely defined for all $y\in {\mathbb R}^d$.
\er

Given now $f\in C^{1,1}_b({\cal P}_2({\mathbb R}^d))$, for fixed $y\in {\mathbb R}^d$ the
question of the differentiability of its components $(\partial_\mu
f)_j(.,y):{\cal P}_2({\mathbb R}^d)\rightarrow {\mathbb R},\ 1\le j\le d,$ raises,
and it can be discussed in the same way as the first order
derivative $\partial_\mu f$ above. If $(\partial_\mu f)_j(.,y):{\cal
P}_2({\mathbb R}^d)\rightarrow {\mathbb R}$ belongs to $C^{1,1}_b({\cal P}_2({\mathbb R}^d))$, we
have that its derivative $\partial_\mu((\partial_\mu
f)_j(.,y))(.,.):{\cal P}_2({\mathbb R}^d)\times {\mathbb R}^d\rightarrow {\mathbb R}^d$ is a
Lipschitz-continuous function, for every $y\in {\mathbb R}^d$. Then
\be\label{2.4}\partial_\mu^2f(\mu,x,y):=\left(\partial_\mu((\partial_\mu
f)_j(.,y))(\mu,x)\right)_{1\le j\le d},\quad (\mu,x,y)\in{\cal
P}_2({\mathbb R}^d)\times {\mathbb R}^d\times {\mathbb R}^d,\ee

\noindent defines a function
$\partial_\mu^2f:{\cal P}_2({\mathbb R}^d)\times {\mathbb R}^d\times {\mathbb R}^d\rightarrow
{\mathbb R}^d\otimes {\mathbb R}^d$.

\begin{definition}  We say that $f\in C^{2,1}_b({\cal
P}_2({\mathbb R}^d))$, if $f\in C^{1,1}_b({\cal P}_2({\mathbb R}^d))$ and\\
{\rm i)} $(\partial_\mu f)_j(.,y)\in  C^{1,1}_b({\cal P}_2({\mathbb R}^d)),$ for all
$y\in {\mathbb R}^d,\, 1\le j\le d$, and

\noindent\ \ \ $\partial_\mu^2f:{\cal P}_2({\mathbb R}^d)\times {\mathbb R}^d\times {\mathbb R}^d\rightarrow
{\mathbb R}^d\otimes {\mathbb R}^d$ is bounded and Lipschitz-continuous;\\
{\rm ii)} $\partial_\mu f(\mu, .): {\mathbb R}^d\rightarrow {\mathbb R}^d$ is differentiable, for
every $\mu\in{\cal P}_2({\mathbb R}^d)$, and its derivative
$\partial_{y}\partial_\mu f:{\cal P}_2({\mathbb R}^d)$

\noindent\ \ \ $\times {\mathbb R}^d\rightarrow
{\mathbb R}^d\otimes {\mathbb R}^d$ is bounded and Lipschitz-continuous.\end{definition}

Adopting the above introduced notations, we consider a function
$f\in C^{2,1}_b({\cal P}_2({\mathbb R}^d))$ and discuss its
second order Taylor expansion. For this end we have still to introduce
some notations. Let $(\widetilde{\Omega},\widetilde{\cal F},\widetilde{P})$
be a copy of the probability space $(\Omega,{\cal F},P)$. For any random
variable (of arbitrary dimension) $\vartheta$ over $(\Omega,{\cal F},P)$
we denote by $\widetilde{\vartheta}$ a copy (of the same law as $\vartheta$,
but defined over $(\widetilde{\Omega},\widetilde{\cal F},\widetilde{P}):$
$\widetilde{P}_{\widetilde{\vartheta}}=P_\vartheta$. The expectation
$\widetilde{E}[.]=\int_{\widetilde{\Omega}}(.)d\widetilde{P}$ acts only
over the variables endowed with a tilde. This can be made rigorous by working
with the product space
$(\Omega,{\cal F},P)\otimes(\widetilde{\Omega},\widetilde{\cal F},
\widetilde{P})=(\Omega,{\cal F},P)\otimes(\Omega,{\cal F},P)$ and putting
$\widetilde{\vartheta}(\widetilde{\omega},\omega):=\vartheta(\widetilde{\omega}),\,
(\widetilde{\omega},\omega)\in\widetilde{\Omega}\times\Omega=\Omega\times
\Omega$, for $\vartheta$ random variable defined over $(\Omega,{\cal F},P)$).
Of course, this formalism can be easily extended from random variables to stochastic processes.

\smallskip

With the above notation and writing $a\otimes b:=(a_{i}b_j)_{1\le i,j\le d},$\
for $a=(a_{i})_{1\le i\le d},\ b=(b_j)_{1\le j\le d}\in {\mathbb R}^d$, we can state now
the following result.

\bl Let $f\in C^{2,1}_b({\cal P}_2({\mathbb R}^d))$. Then, for any given $\vartheta_0\in L^2({\cal F};{\mathbb R}^d)$ we have
the following second order expansion:
\be\label{2.5}\begin{array}{lll}
&f(P_{\vartheta})-f(P_{\vartheta_0})=E\left[\partial_\mu
f(P_{\vartheta_0},\vartheta_0)\cdot\eta\right]+\frac{1}{2}
E\left[\widetilde{E}\left[tr(\partial_\mu^2f(P_{\vartheta_0},
\widetilde{\vartheta_0},\vartheta_0)\cdot\widetilde{\eta}\otimes
\eta)\right]\right]\\
&\hskip 3.2cm+\frac{1}{2}E\left[\partial_y
\partial_\mu f(P_{\vartheta_0},\vartheta_0)\cdot\eta\otimes\eta\right] +R(P_{\vartheta},P_{\vartheta_0}),\ \vartheta\in L^2({\cal F},{\mathbb R}^d),\end{array}\ee

\noindent where $\eta:=\vartheta-\vartheta_0$, and for all $\vartheta\in L^2({\cal F};{\mathbb R}^d)$ the
remainder $R(P_{\vartheta},P_{\vartheta_0})$ satisfies the estimate
\be\label{2.6} |R(P_{\vartheta},P_{\vartheta_0})|\le
CE\left[|\vartheta-\vartheta_0|^{3}\wedge |\vartheta-\vartheta_0|^{2}\right].\ee

\noindent The constant $C\in {\mathbb R}_+$ only depends on the H\"{o}lder
norm of $\partial_\mu^2f$ and $\partial_y\partial_\mu f$.\el
 We observe that the above second order expansion doesn't
constitute a second order Taylor expansion for the associated
function $\widetilde{f}:L^2({\cal F}; {\mathbb R}^d)\rightarrow {\mathbb R}$, since the
reminder is of order $E\left[|\vartheta-\vartheta_0|^{3}\wedge |\vartheta-\vartheta_0|^{2}\right]$
and not of order $o(|\vartheta-\vartheta_0|^2_{L^2})$. Indeed, as the following example shows, in general we only have $E\left[|\vartheta-\vartheta_0|^{3}\wedge |\vartheta-\vartheta_0|^{2}\right]=O(|\vartheta-\vartheta_0|^2_{L^2})$.

\begin{example} Let $\vartheta_\ell=I_{A_\ell}$, with $A_\ell\in {\cal F}$\ such that $P(A_\ell)\rightarrow 0$, as $\ell\rightarrow \infty$. Then $\vartheta_\ell\rightarrow 0$\ in $L^2\ (\ell\rightarrow \infty)$, and $E\left[|\vartheta_\ell|^{3}\wedge |\vartheta_\ell|^{2}\right]=P(A_\ell)=E[\vartheta_\ell^2]=|\vartheta_\ell|^2_{L^2}\rightarrow 0$\ ($\ell\rightarrow \infty$).
\end{example}

However, for our purposes the above expansion is fine.
\begin{proof} (of Lemma 2.1) Let $\vartheta_0\in L^2({\cal F};{\mathbb R}^d)$.
Then, for all $\vartheta\in L^2({\cal F};{\mathbb R}^d)$, putting
$\eta:=\vartheta-\vartheta_0$ and using the fact that $f\in
C^{2,1}_b({\cal P}_2({\mathbb R}^d))$ , we have
\be\label{2.7}
f(P_{\vartheta})-f(P_{\vartheta_0})=\int_0^1\frac{d}{d\lambda}
f(P_{\vartheta_0+\lambda\eta})d\lambda=\int_0^1 E\left[
\partial_\mu f(P_{\vartheta_0+\lambda\eta},\vartheta_0+
\lambda\eta)\cdot \eta\right]d\lambda.\ee
 Let us now compute $\displaystyle\frac{d}{d\lambda}\partial_\mu
f(P_{\vartheta_0+\lambda\eta}, \vartheta_0+ \lambda\eta)$. Since
$f\in C^{2,1}_b({\cal P}_2({\mathbb R}^d))$, the lifted function $\widetilde{\partial_\mu f}(\xi, y)={\partial_\mu f}(P_\xi, y)$, $\xi\in L^2({\cal F};{\mathbb R}^d)$, is Fr\'{e}chet differentiable in $\xi$, and
\be\label{2.8}\frac{d}{d\lambda}\partial_\mu
f(P_{\vartheta_0+\lambda\eta},y)=\frac{d}{d\lambda}\widetilde{\partial_\mu f}(\vartheta_0+\lambda\eta, y) =E\left[\partial_\mu^2
f(P_{\vartheta_0+\lambda\eta},\vartheta_0+\lambda\eta,y)\cdot
\eta\right],\  \lambda\in {\mathbb R},\ y\in {\mathbb R}^d.\ee

\noindent Then, choosing an independent copy
$(\widetilde{\vartheta},\widetilde{\vartheta_0})$ of
$(\vartheta,\vartheta_0)$ defined over $(\widetilde{\Omega},
\widetilde{\cal F},\widetilde{P})$
(i.e., in particular, $\widetilde{P}_{(\widetilde{\vartheta},
\widetilde{\vartheta_0})}=P_{(\vartheta,\vartheta_0)}$),
we have for $\widetilde{\eta}:=\widetilde{\vartheta}-\widetilde{\vartheta_0},$
\be\label{2.9}\frac{d}{d\lambda}\partial_\mu f(P_{\vartheta_0+\lambda\eta},y)=\widetilde{E}
\left[\partial_\mu^2f(P_{\vartheta_0+\lambda\eta},\widetilde{\vartheta_0}+
\lambda\widetilde{\eta},y)\cdot\widetilde{\eta}\right],\ \lambda\in{\mathbb R},\ y\in {\mathbb R}^d.\ee

\noindent Consequently,
\be\label{2.10}\frac{d}{d\lambda}\partial_\mu
f(P_{\vartheta_0+\lambda\eta},\vartheta_0+
\lambda\eta)=\widetilde{E}
\left[\partial_\mu^2f(P_{\vartheta_0+\lambda\eta},
\widetilde{\vartheta_0}+\lambda\widetilde{\eta},\vartheta_0+
\lambda\eta)\cdot\widetilde{\eta}\right]+\partial_y\partial_\mu f
(P_{\vartheta_0+\lambda\eta},\vartheta_0+\lambda\eta)
\cdot\eta.\ee

\noindent Thus,
\be\label{2.11}\begin{array}{lll}
f(P_{\vartheta})-f(P_{\vartheta_0})&=&\displaystyle \int_0^1 E\left[\partial_\mu f(P_{\vartheta_0+\lambda\eta},
\vartheta_0+ \lambda\eta)\cdot\eta\right]d\lambda\\
 &=\displaystyle &E\left[
\partial_\mu f(P_{\vartheta_0},\vartheta_0)\cdot
\eta\right]+\displaystyle \int_0^1\int_0^\lambda E\left[
\frac{d}{d\rho}\partial_\mu f(P_{\vartheta_0+\rho\eta},
\vartheta_0+\rho\eta)\cdot \eta\right]d\rho d\lambda\\
&=\displaystyle &E\left[
\partial_\mu f(P_{\vartheta_0},\vartheta_0)\cdot \eta\right]\\
& &+\displaystyle \int_0^1\int_0^\lambda
E\left[\widetilde{E}\left[tr(\partial_\mu^2f(P_{\vartheta_0+
\rho\eta},\widetilde{\vartheta_0}+\rho\widetilde{\eta},\vartheta_0
+\rho\eta)\cdot\widetilde{\eta}\otimes
\eta)\right]\right] d\rho d\lambda\\
& &+\displaystyle \int_0^1\int_0^\lambda
E\left[tr(\partial_y\partial_\mu f(P_{\vartheta_0+\rho\eta},
\vartheta_0+\rho\eta)\cdot\eta\otimes\eta)\right]d\rho d\lambda.\end{array}\ee

\noindent From this latter relation we get
\be\label{2.11}\begin{array}{lll} f(P_{\vartheta})-f(P_{\vartheta_0})&=&E\left[\partial_\mu
f(P_{\vartheta_0},\vartheta_0)\cdot\eta\right]+\frac{1}{2}
E\left[\widetilde{E}\left[tr(\partial_\mu^2f(P_{\vartheta_0},
\widetilde{\vartheta_0},\vartheta_0)\cdot\widetilde{\eta}\otimes
\eta)\right]\right]\\
& &+\frac{1}{2}E\left[tr(\partial_y
\partial_\mu f(P_{\vartheta_0},\vartheta_0)\cdot\eta\otimes\eta)\right]+R_1(P_{\vartheta},P_{\vartheta_0})+
R_2(P_{\vartheta}, P_{\vartheta_0}),\end{array}\ee

\noindent with the remainders
\be\label{2.13}
R_1(P_{\vartheta},P_{\vartheta_0})=\int_0^1\int_0^\lambda
E\left[\widetilde{E}\left[tr\left(\left(\partial_\mu^2f(P_{\vartheta_0+\rho\eta},
\widetilde{\vartheta_0}+\rho\widetilde{\eta},\vartheta_0+\rho\eta)-
\partial_\mu^2f(P_{\vartheta_0},\widetilde{\vartheta_0},\vartheta_0)
\right)\cdot\widetilde{\eta}\otimes\eta\right)\right]\right] d\rho
d\lambda\ee

\noindent and
\be\label{2.14}\displaystyle R_2(P_{\vartheta},P_{\vartheta_0})=
\int_0^1\int_0^\lambda E\left[tr\left(\left(\partial_y\partial_\mu
f(P_{\vartheta_0+\rho\eta},
\vartheta_0+\rho\eta)-\partial_y\partial_\mu f(P_{\vartheta_0},
\vartheta_0)\right) \cdot\eta\otimes\eta\right)\right]d\rho d\lambda.\ee

Finally, from the boundedness and the Lipschitz continuity of the
functions $\partial_\mu^2f$ and $\partial_y\partial_\mu f$ we conclude
that, for some $C\in {\mathbb R}_+$ only depending on $\partial_\mu^2f,\ \partial_y\partial_\mu f$,
$\displaystyle |R_1(P_{\vartheta},P_{\vartheta_0})|\le
C\left(E\left[|\eta|^2\wedge 1\right]\right)^{\frac{3}{2}},$ while $\displaystyle
|R_2(P_{\vartheta},P_{\vartheta_0})|\le CE\left[|\eta|^{2}(|\eta|\wedge 1)\right]=CE\left[|\eta|^{3}\wedge |\eta|^2\right].$
This proves the statement.\end{proof}

Let us finish our preliminary discussion with an illustrating
example.
\begin{example} Given two twice continuously differentiable functions $h: {\mathbb R}^d\rightarrow {\mathbb R}$ and $g: {\mathbb R}\rightarrow
{\mathbb R}$\ with bounded derivatives, we consider $f(P_\vartheta):=g\left(E[h(\vartheta)]\right),$
$\vartheta\in L^2({\cal F};{\mathbb R}^d)$. Then, given any $\vartheta_0\in
L^2({\cal F};{\mathbb R}^d)$, $\widetilde{f}(\vartheta):=f(P_\vartheta)=
g\left(E[h(\vartheta)] \right)$ is Fr\'{e}chet differentiable in $\vartheta_0$, and

\smallskip

$$\begin{array}{rcl}\displaystyle \widetilde{f}(\vartheta_0+\eta)-\widetilde{f}(\vartheta_0)&=\int_0^1g'(E[h(\vartheta_0+
s\eta)])E[h'(\vartheta_0+s\eta)\eta]ds\\
&=g'(E[h(\vartheta_0)])E[h'(\vartheta_0)\eta]+o(|\eta|_{L^2})\\
&=E[g'(E[h(\vartheta_0)])h'(\vartheta_0)\eta]+o(|\eta|_{L^2}).\\
\end{array}$$
Thus, $D\widetilde{f}(\vartheta_0)(\eta)=E[g'(E[h(\vartheta_0)])h'(\vartheta_0)\eta]$, $\eta\in L^2({\cal F};{\mathbb R}^d)$, i.e.,
$$\displaystyle \partial_\mu
f(P_{\vartheta_0},y)=g'(E[h(\vartheta_0)])(\partial_yh)(y), y\in {\mathbb R}.$$

\smallskip

\noindent With the same argument we see that

\smallskip

$\displaystyle\partial_\mu^2f(P_{\vartheta_0},x,y)=
g''(E[h(\vartheta_0)])(\partial_xh)(x)\otimes(\partial_yh)
(y)$ and $\displaystyle\partial_y\partial_\mu f(P_{\vartheta_0},y)=g'(E[h(\vartheta_0)])
(\partial_y^2h)(y)$.

\smallskip

\noindent Consequently, if $g$ and $h$ are three times continuously
differentiable with bounded derivatives of all order, then
the second order expansion stated in the above Lemma 2.1 takes for this example
the special form

\smallskip

$g(E[h(\vartheta)])-g(E[h(\vartheta_0)])$

$=g'(E[h(\vartheta_0)])E\left[\partial_y h(\vartheta_0)
\cdot(\vartheta-\vartheta_0)\right]+\frac{1}{2}g''(E[h(\vartheta_0)])
\left(E\left[\partial_y h(\vartheta_0)\cdot(\vartheta-\vartheta_0)
\right]\right)^2$

$\quad+\frac{1}{2}g'(E[h(\vartheta_0)])E\left[tr\left(
\partial^2_y h(\vartheta_0)\cdot(\vartheta-\vartheta_0)\otimes
(\vartheta-\vartheta_0)\right)\right]+O\left(E[|\vartheta-
\vartheta_0|^3\wedge 1]\right).$
\end{example}
\section{The mean-field stochastic differential equation}

Let us now consider a complete probability space $(\Omega,{\cal
F},P)$ on which is defined a $d$-dimensional Brownian motion
$B(=(B^1,\dots,B^d))=(B_t)_{t\in [0,T]}$, and $T>0$ denotes an
arbitrarily fixed time
horizon. We suppose that there is a sub-$\sigma$-field ${\cal
F}_0\subset {\cal F}$ such that

i) the Brownian motion $B$ is independent of ${\cal F}_0$, and

ii) ${\cal F}_0$ is ``rich enough'', i.e., ${\cal P}_2({\mathbb R}^d)=
\{P_\vartheta,\, \vartheta\in L^2({\cal F}_0;{\mathbb R}^d)\}.$

\noindent By $\mathbb{F}=({\cal F}_t)_{t\in[0,T]}$ we denote the
filtration generated by $B$, completed and augmented by ${\cal
F}_0$.

\smallskip

Given deterministic Lipschitz functions $\sigma:{\mathbb R}^d\times
{\cal P}_2({\mathbb R}^d)\rightarrow {\mathbb R}^{d\times d}$ and $b:{\mathbb R}^d\times
{\cal P}_2({\mathbb R}^d)\rightarrow {\mathbb R}^{d}$, we consider for the
initial data $(t,x)\in[0,T]\times {\mathbb R}^d$ and $\xi\in L^2({\cal
F}_t;{\mathbb R}^d)$ the stochastic differential equations (SDEs)
\be\label{3.1} X^{t,\xi}_s=\xi+\int_t^s\sigma(X^{t,\xi}_r,P_{X^{t,\xi}_r})dB_r+
\int_t^sb(X^{t,\xi}_r,P_{X^{t,\xi}_r})dr,\ s\in[t,T],\ee and
\be\label{3.2} X^{t,x,\xi}_s=x+\int_t^s\sigma(X^{t,x,\xi}_r,P_{X^{t,\xi}_r})dB_r+
\int_t^sb(X^{t,x,\xi}_r,P_{X^{t,\xi}_r})dr,\ s\in[t,T].\ee We observe
that under our Lipschitz assumption on the coefficients the both
SDEs have a unique solution in ${\cal S}^2([t,T];{\mathbb R}^d)$, the space of
$\mathbb{F}$-adapted continuous processes $Y=(Y_s)_{s\in[t,T]}$ with
$E\left[\sup_{s\in[t,T]}|Y_s|^2\right]<+\infty$ (see, for example,
 Carmona and Delarue \cite{CD}). We see, in particular, that the solution
$X^{t,\xi}$ of the first equation allows to determine that of the
second equation. As SDE standard estimates show, we have for some
$C\in {\mathbb R}_+$ depending only on the Lipschitz constants of $\sigma$
and $b$,
\be\label{3.3} E[\sup_{s\in[t,T]}|X_s^{t,x,\xi}
-X_s^{t,x',\xi}|^2]\le C|x-x'|^2,\ee
\noindent for all $t\in[0,T],\ x,\ x'\in {\mathbb R}^d,\ \xi\in L^2(
{\cal F}_t;{\mathbb R}^d).$ This allows to substitute in the second SDE
for $x$ the random variable $\xi$ and shows that $X^{t,x,\xi}
\big|_{x=\xi}$ solves the same SDE as $X^{t,\xi}$. From the
uniqueness of the solution we conclude
\be\label{3.4}X^{t,x,\xi}_s\big|_{x=\xi}=X^{t,\xi}_s,\quad s\in[t,T].\ee
Moreover, from the uniqueness of the solution of the both
equations we deduce the following {\it flow property}
\be\label{3.5}\left(X^{s,X^{t,x,\xi}_s,X^{t,\xi}_s}_r,X^{s,X^{t,\xi}_s}_r\right)
=\left(X^{t,x,\xi}_r,X^{t,\xi}_r\right),\ r\in[s,T],
\mbox{ for all } 0\le t\le s\le T,\, x\in {\mathbb R}^d,\ \xi\in L^2({\cal
F}_t;{\mathbb R}^d).\ee
In fact, putting $\eta=X^{t,\xi}_s\in L^2({\cal
F}_s;{\mathbb R}^d)$, and considering the SDEs (\ref{3.1}) and (\ref{3.2}) with the initial data $(s, y)$\ and $(s, \eta)$, respectively,
$$\label{3.1-1} X^{s,\eta}_r=\eta+\int_s^r\sigma(X^{s,\eta}_u,P_{X^{s,\eta}_u})dB_u+
\int_s^rb(X^{s,\eta}_u,P_{X^{s,\eta}_u})du,\ r\in[s,T],$$ and
$$\label{3.2-2} X^{s,y,\eta}_r=y+\int_s^r\sigma(X^{s,y,\eta}_u,P_{X^{s,\eta}_u})dB_u+
\int_s^rb(X^{s,y,\eta}_u,P_{X^{s, \eta}_u})du,\ r\in[s,T],$$
we get from the uniqueness of the solution that $X^{s,\eta}_r=X^{t,\xi}_r$, $r\in [s, T]$, and, consequently, $X^{s,X^{t,x,\xi}_s,\eta}_r=X^{t,x,\xi}_r,$\
$r\in [t, T]$, i.e., we have (\ref{3.5}).

Having this flow property, it is natural to define for a sufficiently
regular function $\Phi:{\mathbb R}^d\times{\cal P}_2({\mathbb R}^d)\rightarrow {\mathbb R}$
an associated value function
\be\label{3.6}V(t,x,\xi):=E\left[\Phi(X^{t,x,\xi}_T,P_{X^{t,\xi}_T})\right],\
(t,x)\in [0,T]\times {\mathbb R}^d,\ \xi\in L^2({\cal F}_t;{\mathbb R}^d),\ee and to ask
which partial differential equation is satisfied by this function
$V$. In order to be able to answer to this question in the frame of
the concept we have introduced above, we have to show that the
function $V(t,x,\xi)$ does not depend on $\xi$ itself but only on its
law $P_\xi$, i.e., that we have to do with a function $V:[0,T]\times
{\mathbb R}^d\times {\cal P}_2({\mathbb R}^d)\rightarrow {\mathbb R}$. For this the following lemma is useful.
\bl For all $p\ge 2$ there is
constant $C_p\in {\mathbb R}_+$ only depending on
the Lipschitz constants of $\sigma$ and $b$, such that we have the
following estimate
\be \label{3.7}E[\sup_{s\in[t,T]}|X^{t,x_1,\xi_1}_s-X^{t,x_2,\xi_2}_s|^p]\le C_p\left(|x_1-x_2|^p+W_2(P_{\xi_1},P_{\xi_2})^p\right),\ee
for all $t\in[0,T],\ x_1,\ x_2\in {\mathbb R}^d,\ \xi_1,\ \xi_2\in L^2({\cal
F}_t;{\mathbb R}^d).$
\el
\begin{proof} Recall that for the 2-Wasserstein
metric $W_2(.,.)$ we have
\be \label{3.8}\begin{array}{lll}
 W_2(P_\vartheta,P_\theta)& = \inf\left\{(E[|\vartheta'-\theta'|^2])^{1/2},
\mbox{ for all }\vartheta',\theta'\in L^2({\cal F}_0;{\mathbb R}^d) \mbox{
with }P_{\vartheta'}=P_\vartheta,\ P_{\theta'}=P_\theta\right\}\\
& \le  (E[|\vartheta-\theta|^2])^{1/2},\ \ \mbox{for all }\vartheta,\ \theta\in L^2({\cal F};{\mathbb R}^d),\end{array}\ee
because we have chosen ${\cal F}_0$\ ``rich enough". Since our coefficients $\sigma$ and $b$ are Lipschitz over
${\mathbb R}^d\times{\cal P}_2({\mathbb R}^d)$, this allows to get with the help of standard
estimates for the SDEs for $X^{t,\xi}$ and $X^{t,x,\xi}$ that,
for some constant $C\in {\mathbb R}_+$ only depending on the
Lipschitz constants of $\sigma$ and $b$,
\be \label{3.9}
E[\sup_{s\in[t,T]}|X^{t,\xi_1}_s-X^{t,\xi_2}_s|^2]\le
CE\left[|\xi_1-\xi_2|^2\right],\ \xi_1,\ \xi_2\in L^2({\cal
F}_t;{\mathbb R}^d),\ t\in[0,T],\ee

\noindent and, for some $C_p\in {\mathbb R}$ depending only on $p$ and the Lipschitz
constants of the coefficients,
\be \label{3.10}
E[\sup_{s\in[t,v]}|X^{t,x_1,\xi_1}_s-X^{t,x_2,\xi_2}_s|^p]\le
C_p (|x_1-x_2|^p+\int_t^vW_2(P_{X^{t,\xi_1}_r},P_{X^{t,\xi_2}_r})^pdr),\ee

\noindent for all $x_1,\ x_2\in {\mathbb R}^d,\ \xi_1,\ \xi_2\in L^2({\cal
F}_t;{\mathbb R}^d)$, $0\le t\le v\le T.$\ On the other hand, from the SDE for
$X^{t,x,\xi}$ we derive easily that
$X^{t,\xi',\xi}(:=X^{t,x,\xi}\big|_{x=\xi'})$ obeys the same law as
$X^{t,\xi}(=X^{t,x,\xi}\big|_{x=\xi})$, whenever $\xi,\ \xi'\in L^2({\cal F}_t;
{\mathbb R}^d)$ have the same law. This allows to deduce from the latter estimate, for $p=2$,
\be \label{3.11}\begin{array}{lll}
&\displaystyle\sup_{s\in[t,v]}W_2(P_{X^{t,\xi_1}_s},P_{X^{t,\xi_2}_s})^2
\le\displaystyle \sup_{s\in[t,v]}E[|X^{t,\xi'_1,\xi_1}_s-X^{t,\xi'_2,\xi_2}_s|^2]\leq E[\sup_{s\in[t,v]}|X^{t,\xi'_1,\xi_1}_s-X^{t,\xi'_2,\xi_2}_s
|^2]\\
&\le\displaystyle C(E[|\xi'_1-\xi'_2|^2]+\int_t^vW_2
(P_{X^{t,\xi_1}_r},P_{X^{t,\xi_2}_r})^2dr),\ v\in[t,T],\end{array}\ee

\noindent for all $\xi'_1,\ \xi'_2\in L^2({\cal F}_t;{\mathbb R}^d)$\ with
$P_{\xi'_1}=P_{\xi_1}$ and  $P_{\xi'_2}=P_{\xi_2}$. Hence, taking
at the right-hand side of (\ref{3.11}) the infimum over all such $\xi'_1,\ \xi'_2
\in L^2({\cal F}_t;{\mathbb R}^d)$ and considering the above characterization of
the 2-Wasserstein metric, we get
\be \label{3.12}
\sup_{s\in[t,v]}W_2(P_{X^{t,\xi_1}_s},P_{X^{t,\xi_2}_s})^2
\le C(W_2(P_{\xi'_1},P_{\xi'_2})^2+\int_t^vW_2(P_{X^{t,
\xi_1}_r},P_{X^{t,\xi_2}_r})^2dr),\ v\in[t,T].\ee

\noindent Then Gronwall's inequality implies
\be \label{3.13}
\sup_{s\in[t,T]}W_2(P_{X^{t,\xi_1}_s},P_{X^{t,\xi_2}_s})^2
\le CW_2(P_{\xi_1},P_{\xi_2})^2,\ t\in[0,T],\ \xi_1,\ \xi_2\in
L^2({\cal F}_t;{\mathbb R}^d),\ee

\noindent which allows to deduce from the estimate (\ref{3.10})

\be \label{3.14}
E[\sup_{s\in[t,T]}|X^{t,x_1,\xi_1}_s-X^{t,x_2,
\xi_2}_s|^p]\le C_p(|x_1-x_2|^p+W_2(P_{\xi_1},P_{\xi_2})^p),\ee

\noindent for all $t\in[0,T],\ x_1,\ x_2\in {\mathbb R}^d,\ \xi_1,\ \xi_2\in
L^2({\cal F}_t;{\mathbb R}^d).$ The proof is complete now.\end{proof}
\begin{remark} An immediate consequence of the above
Lemma 3.1 is that, given $(t,x)\in[0,T]\times {\mathbb R}^d$, the processes
$X^{t,x,\xi_1}$ and $X^{t,x,\xi_2}$ are indistinguishable, whenever the laws
of $\xi_1\in  L^2({\cal F}_t;{\mathbb R}^d)$ and $\xi_2\in  L^2({\cal F}_t;{\mathbb R}^d)$
are the same. But this means that we can define
\be\label{3.15} X^{t,x,P_\xi}:=X^{t,x,\xi},\ (t,x)\in
[0,T]\times {\mathbb R}^d,\ \xi\in L^2({\cal F}_t;{\mathbb R}^d),\ee

\noindent and, extending the notation introduced in the preceding
section for functions to random variables and processes, we
should consider $\displaystyle \widetilde{X}^{t,x,\xi}_s=X^{t,x,P_\xi}_s=
X^{t,x,\xi}_s,\ s\in[t,T],\ (t,x)\in[0,T]\times {\mathbb R}^d,\ \xi\in
L^2({\cal F}_t;{\mathbb R}^d).$ However, we will
prefer to write $X^{t,x,\xi}$ and reserve the notation
$\widetilde{X}^{t,x,P_\xi}$ for an independent copy of
$X^{t,x,P_\xi}$, which we will introduce later.\end{remark}

Having now by the above relation the process $X^{t,x,\mu}$ defined
for all $\mu\in{\cal P}_2({\mathbb R}^d)$, the question of its
differentiability with respect to $\mu$ raises; it will be studied
through the Fr\'{e}chet differentiability of the mapping $L^2({\cal
F}_t;{\mathbb R}^d)\ni\xi\rightarrow X^{t,x,\xi}_s\in L^2({\cal
F}_s;{\mathbb R}^d),\ s\in[t,T]$. For this we suppose that

\medskip

\noindent\underline{\textbf{Hypothesis (H.1)}} The couple of coefficients $(\sigma, b)$ belongs to $C^{1,1}_b({\mathbb R}^d\times {\cal P}_2({\mathbb R}^d)
\rightarrow {\mathbb R}^{d\times d}\times {\mathbb R}^d)$, i.e., the components
$\sigma_{i,j},\ b_j$, $1\le i, j\le d$, have the following properties:

i) $\sigma_{i,j}(x,.),\ b_j(x,.)$ belong to $C^{1,1}_b({\cal
P}_2({\mathbb R}^d))$, for all $x\in {\mathbb R}^d$;

ii)$\sigma_{i,j}(,\mu),\ b_j(.,\mu)$ belong to $C^1_b({\mathbb R}^d)$, for all
$\mu\in{\cal P}_2({\mathbb R}^d)$;

iii) The derivatives
$\partial_x\sigma_{i,j},\ \partial_xb_j: {\mathbb R}^d\times{\cal
P}_2({\mathbb R}^d)\rightarrow {\mathbb R}^d$, and
$\partial_\mu\sigma_{i,j},\ \partial_\mu b_j: {\mathbb R}^d\times{\cal
P}_2({\mathbb R}^d)\times {\mathbb R}^d\rightarrow {\mathbb R}^d$, are bounded and Lipschitz continuous.
\begin{theorem} Let $(\sigma,\ b)\in
C^{1,1}_b({\mathbb R}^d\times {\cal P}_2({\mathbb R}^d)\rightarrow {\mathbb R}^{d\times d}\times
{\mathbb R}^d)$\ satisfy assumption (H.1). Then, for all $0\le t\le s\le T$ and $x\in {\mathbb R}^d$,
the mapping $L^2({\cal F}_t;{\mathbb R}^d)\ni\xi\mapsto X^{t,x,\xi}_s=
X^{t,x,P_\xi}_s\in L^2({\cal F}_s;{\mathbb R}^d)$ is Fr\'{e}chet differentiable,
with Fr\'{e}chet derivative
\be\label{3.16}\begin{array}{lll}
&DX_s^{t,x,\xi}(\eta)=
\widetilde{E}[U^{t,x,P_\xi}_s(\widetilde{\xi})\cdot
\widetilde{\eta}]\\
& =(\widetilde{E}[\sum_{j=1}^d U^{t,x,P_\xi}_{s,i,j}
(\widetilde{\xi})\cdot\widetilde{\eta}_j])_{1\le i\le d},\ \mbox{for all}\ \eta=(\eta_1,\dots,\eta_d)\in L^2({\cal F}_t;{\mathbb R}^d),
\end{array}\ee
\noindent where, for all $y\in {\mathbb R}^d$, $U^{t,x,P_\xi}(y)=\big((U^{t,x,
P_\xi}_{s,i,j}(y))_{s\in[t,T]}\big)_{1\le i,j\le d}\in S^2_{\mathbb{F}}(t,T;
{\mathbb R}^{d\times d})$ is the unique solution of the SDE
\noindent\be\label{3.17}\begin{array}{lll} &\!\!\!\!\!\!U^{t,x,P_\xi}_{s,i,j}(y)\\
&\!\!\!\!\!\!=\displaystyle\sum_{k,\ell=1}^d\int_t^s
\partial_{x_k}\sigma_{i,\ell}(X^{t,x,P_\xi}_r,P_{X^{t,\xi}_r})\cdot
U^{t,x,P_\xi}_{r,k,j}(y)dB^\ell_r \displaystyle+\sum_{k=1}^d\int_t^s
\partial_{x_k}b_{i}(X^{t,x,P_\xi}_r,P_{X^{t,\xi}_r})\cdot
U^{t,x,P_\xi}_{r,k,j}(y)dr\\
&\!\!\!\!\!\!+\displaystyle\sum_{k,\ell=1}^d\int_t^s\!\!\! E\!\!\left[
(\partial_\mu\sigma_{i,\ell})_k(z,P_{X^{t,\xi}_r},
X_r^{t,y,P_\xi})\cdot\partial_{x_j}
X^{t,y,P_\xi}_{r,k}\!\!+(\partial_\mu\sigma_{i,\ell})_k(z,P_{X^{t,\xi}_r},
X_r^{t,\xi})\cdot U^{t,\xi}_{r,k,j}(y)\!\right]_{\big|z=X^{t,x,P_\xi}_r}dB^\ell_r\ \  \\
&\!\!\!\!\!\!+\displaystyle\sum_{k=1}^d\int_t^sE\left[(\partial_\mu
b_{i})_k(z,P_{X^{t,\xi}_r},X_r^{t,y,P_\xi})\cdot \partial_{x_j}X^{t,y,P_\xi}_{r,k}+(\partial_\mu
b_{i})_k(z,P_{X^{t,\xi}_r},X_r^{t,\xi})\cdot U^{t,\xi}_{r,k,j}(y)\right]_{\big|z=X^{t,x,P_\xi}_r}dr,\ \end{array}\ee

\noindent $s\in[t,T],\ 1\le i,j\le d,$ and $U^{t,\xi}(y)=\big((U^{t,\xi}_{s,i,j}
(y))_{s\in[t,T]}\big)_{1\le i,j\le d}\in S^2_{\mathbb{F}}(t,T; {\mathbb R}^{d\times d})$ is that of the SDE
\be\label{3.18}\displaystyle\begin{array}{lll} &\!\!\!\!U^{t,\xi}_{s,i,j}(y)\\
&\!\!\!\!=\displaystyle\sum_{k,\ell=1}^d
\int_t^s\partial_{x_k}\sigma_{i,\ell}(X^{t,\xi}_r,P_{X^{t,\xi}_r})\cdot
U^{t,\xi}_{r,k,j}(y)dB^\ell_r+\displaystyle\sum_{k=1}^d\int_t^s
\partial_{x_k}b_{i}(X^{t,\xi}_r,P_{X^{t,\xi}_r})\cdot
U^{t,\xi}_{r,k,j}(y)dr\\
&\!\!\!\!\displaystyle +\sum_{k,\ell=1}^d\int_t^s\!\!\! E\left[(\partial_\mu
\sigma_{i,\ell})_k(z,P_{X^{t,\xi}_r},X_r^{t,y,P_\xi})\cdot
\partial_{x_j}X^{t,y,P_\xi}_{r,k}+(\partial_\mu
\sigma_{i,\ell})_k(z,P_{X^{t,\xi}_r},X_r^{t,\xi})\cdot U^{t,\xi}_{r,k,j}(y)\right]_{\big|z=X^{t,\xi}_r}dB^\ell_r\\
&\!\!\!\!\displaystyle +\sum_{k=1}^d\int_t^sE\left[(\partial_\mu
b_{i})_k(z,P_{X^{t,\xi}_r},X_r^{t,y,P_\xi})\cdot \partial_{x_j}X^{t,y,P_\xi}_{r,k}+(\partial_\mu
b_{i})_k(z,P_{X^{t,\xi}_r},X_r^{t,\xi})\cdot U^{t,\xi}_{r,k,j}(y)\right]_{\big|z=X^{t,\xi}_r}dr,\\ \end{array}\ee

\noindent $s\in[t,T],\ 1\le i,\ j\le d.$\end{theorem}

\begin{remark} Following the definition of the derivative
of a function $f: {\cal P}_2({\mathbb R}^d)\rightarrow {\mathbb R}$ explained Section 2,
we can consider $U^{t,x,P_\xi}_s(y)=(U^{t,x,P_\xi}_{s,i,j}(y))_{1\le i,j\le d}$
as derivative over ${\cal P}_2({\mathbb R}^d)$ of $X^{t,x,P_\xi}_s
=(X^{t,x,P_\xi}_{s,i})_{1\le i\le d}$ at $P_\xi.$ As notation for this
derivative we use that already introduced for functions:
\be\label{3.19}\partial_\mu X^{t,x,P_\xi}_s(y)=(\partial_\mu X^{t,x,P_\xi}_{s,i,j}
(y))_{1\le i,j\le d}:=U^{t,x,P_\xi}_s(y)=(U^{t,x,P_\xi}_{s,i,j}(y))_{1\le
i,j\le d},\ s\in [t,T].\ee
\end{remark}
Let us prove in a first step Theorem 3.1, but with the G\^{a}teaux differentiability instead of that in Fr\'{e}chet's sense, i.e., we prove that
$$L^2({\cal F}_t;{\mathbb R}^d)\ni \xi\mapsto X^{t,x,\xi}_s= X^{t,x,P_\xi}_s\in L^2({\cal F}_s;{\mathbb R}^d)$$
is G\^{a}teaux differentiable, with $\partial_\xi X^{t,x,\xi}_s(\eta)=L^2-\lim_{h\rightarrow 0}\frac{1}{h}(X^{t,x,\xi+h\eta}_s-X^{t,x,\xi}_s)=\widetilde{E}[U^{t,x,P_\xi}_s(\widetilde{\xi})\widetilde{\eta}]$, $s\in [t, T]$, $\xi, \eta\in L^2({\cal F}_t;{\mathbb R}^d).$
\begin{proof} (G\^{a}teaux differentiability) Let us make the proof for
simplicity for dimension $d=1$; using the same argument, the proof can be easily
extended to the case $d\ge 1$. Note that under our assumptions the functions
$\widetilde{\sigma}(x,\vartheta): =\sigma(x,P_\vartheta),$
$\widetilde{b}(x,\vartheta): =b(x,P_\vartheta),$
$(x,\vartheta)\in {\mathbb R}\times L^2({\cal F})$ ($L^2({\cal F})$ stands for $L^2
({\cal F};{\mathbb R})$) are continuously Fr\'{e}chet differentiable over ${\mathbb R}\times L^2({\cal F})$,
and
\be\label{3.20}\displaystyle
D\widetilde{\sigma}(x,\vartheta)(\eta)=E\left[\partial_\mu\sigma(x,
P_\vartheta,\vartheta)\cdot\eta\right],\ \ D\widetilde{b}(x,\vartheta)(\eta)=E\left[\partial_\mu b(x,
P_\vartheta,\vartheta)\cdot\eta\right],\ee

\noindent for all $(x,\vartheta)\in {\mathbb R}\times L^2({\cal F}),$ $\eta\in
L^2({\cal F}).$\ On the other hand, by the definition of
$\widetilde{\sigma}$ and $\widetilde{b}$,
\be\label{3.21}\displaystyle\partial_x\widetilde{\sigma}(x,\vartheta)=
\partial_x\sigma(x, P_\vartheta),\ \ \partial_x\widetilde{b}(x,\vartheta)=\partial_xb(x, P_\vartheta),\ee

\noindent for all $(x,\vartheta)\in {\mathbb R}\times L^2({\cal F}).$

\medskip

\noindent\underline{Step 1}: The G\^{a}teaux derivative of $X^{t,x,\xi}_s\
(=X^{t,x,P_\xi}_s)$ with respect to $\xi$.

\medskip

The objective of this first step is to characterize the G\^{a}teaux
derivative of $X^{t,x,\xi}_s:=X^{t,x,P_\xi}_s$ in $\xi\in
L^2({\cal F}_t)$. For this end let $(t,x)\in[0,T]\times {\mathbb R}^d$ and $\xi,\ \eta\in L^2({\cal F}_t)$. Since the coefficients $\sigma(.,P_\vartheta)$
and $b(.,P_\vartheta)$ are continuously differentiable with derivatives
which are bounded, uniformly with respect to $\vartheta$, it is well
known that $\partial_xX^{t,x,P_\xi}=(\partial_xX^{t,x,P_\xi}_s)_{s
\in[t,T]}$ is $L^2$-differentiable in $x$, and its $L^2$-derivative
\be\label{3.22}\displaystyle
\partial_xX^{t,x,P_\xi}_s=\lim_{h\rightarrow
0}\frac{1}{h}\left(X^{t,x+h,P_\xi}_s-X^{t,x,P_\xi}_s\right),\
s\in[t,T],\ee

\noindent (in $L^2$, uniformly in $s$) is the unique solution of the SDE
\be\label{3.23}\displaystyle \partial_xX^{t,x,P_\xi}_s=1+\int_t^s(\partial_x
\sigma)(X^{t,x,P_\xi}_r,P_{X^{t,\xi}_r})\partial_xX^{t,x,P_\xi}_r
dB_r+\int_t^s(\partial_x b)(X^{t,x,P_\xi}_r,P_{X^{t,\xi}_r})
\partial_xX^{t,x,P_\xi}_r dr,\, s\in[t,T].\ee

\noindent From this linear SDE with bounded derivatives $\partial_x
\sigma$\ and $\partial_x b$ we deduce
by substituting $x=\xi$ that, for
$\partial_xX^{t, \xi,P_\xi}_s:=\partial_xX^{t,x,P_\xi}_s\big|_{x=\xi},
\ s\in[t,T],$ it holds
\be\label{3.24} E[\sup_{s\in[t,T]}|\partial_x X^{t,
\xi,P_\xi}_s|^2\ |\ {\cal F}_t]\le \sup_{x\in {\mathbb R}} E[\sup_{s\in[t,T]}|\partial_x X^{t,
x,P_\xi}_s|^2]\le C,\ \mbox{P-a.s.},\ee

\smallskip

\noindent for some real constant $C$ only depending on the bounds of
$\partial_x\sigma$ and $\partial_x b$.

We remark that, since $\partial_xX_s^{t,x,P_\xi},\ x\in {\mathbb R},$\ is
independent of ${\cal F}_t$, the process $\partial_xX_s^{t,\xi,P_\xi}
:=\partial_xX^{t,x,P_\xi}_{s}\big|_{x=\xi}$ is $P$-a.s. well defined.
Moreover, $\partial_xX^{t,\xi,P_\xi}=(\partial_xX_s^{t,\xi,P_\xi})_{s
\in[t,T]}\in{\cal S}^2_{\mathbb{F}}(t,T)$ is the unique solution of the SDE
\be\label{3.25} \displaystyle \partial_xX^{t,\xi,P_\xi}_s=1+\int_t^s(\partial_x
\sigma)(X^{t,\xi}_r,P_{X^{t,\xi}_r})\partial_xX^{t,\xi,P_\xi}_r
dB_r+\int_t^s(\partial_x b)(X^{t,\xi}_r,P_{X^{t,\xi}_r})
\partial_xX^{t,\xi,P_\xi}_r dr,\ s\in[t,T].\ee

\smallskip

\noindent On the other hand, our assumptions on the
coefficients $\sigma$ and $b$ allow to show that the following
SDE has a unique solution $Y^{t,\xi}(\eta)=(Y^{t,\xi}_s(\eta))_{s
\in[t,T]}\in{\cal S}^2_{\mathbb{F}}(t,T)\ (:={\cal S}^2_{\mathbb{F}}(t,T;{\mathbb R}))$:
\be\label{3.26}\begin{array}{lll}
Y^{t,\xi}_s(\eta)&=&\displaystyle\int_t^s\left(\partial_x\sigma(X^{t,\xi}_r,P_{X^{t,\xi}_r})
Y^{t,\xi}_r(\eta)+D\widetilde{\sigma}(X^{t,\xi}_r,X^{t,\xi}_r)(\partial_xX^{t,
\xi,P_\xi}_r\cdot\eta+Y^{t,\xi}_r(\eta))\right)dB_r\\
& &\displaystyle+\int_t^s\left(\partial_xb(X^{t,\xi}_r,
P_{X^{t,\xi}_r})Y^{t,\xi}_r(\eta)+
D\widetilde{b}(X^{t,\xi}_r,X^{t,\xi}_r)(\partial_xX^{t,
\xi,P_\xi}_r\cdot\eta+Y^{t,\xi}_r(\eta))\right)dr,\end{array}\ee

\noindent $s\in[t,T].$\ Here, we notice that
\be\label{3.27}\begin{array}{lll}
&\!\!\!\! D\widetilde{\sigma}(X^{t,\xi}_r,X^{t,\xi}_r)(\partial_xX^{t,
\xi,P_\xi}_r\cdot\eta+Y^{t,\xi}_r(\eta))=E[
\partial_\mu\sigma(z,P_{X^{t,\xi}_r},X^{t,\xi}_r)\cdot (\partial_x X^{t,
\xi,P_\xi}_r\cdot\eta+Y^{t,\xi}_r(\eta))]_{\big|z=X^{t,\xi}_r},\\
&\!\!\!\! D\widetilde{b}(X^{t,\xi}_r,X^{t,\xi}_r)(\partial_xX^{t,
\xi,P_\xi}_r\cdot\eta+Y^{t,\xi}_r(\eta))=E[
\partial_\mu b(z,P_{X^{t,\xi}_r},X^{t,\xi}_r)\cdot (\partial_x X^{t,
\xi,P_\xi}_r\cdot\eta+Y^{t,\xi}_r(\eta))]_{\big|z=X^{t,\xi}_r},\\ \end{array}\ee

\noindent where, due to our assumptions, the coefficients $\partial_\mu\sigma$\ and $\partial_\mu b$ are bounded. Consequently,
\be\label{3.28}\begin{array}{lll}
& &E [|D\widetilde{\sigma}(X^{t,\xi}_r,X^{t,\xi}_r)(\partial_xX^{t,
\xi,P_\xi}_r\cdot\eta+Y^{t,\xi}_r(\eta))|^2 ]+E [|D\widetilde{b}
(X^{t,\xi}_r,X^{t,\xi}_r)(\partial_xX^{t,\xi,P_\xi}_r\cdot\eta+Y^{t,
\xi}_r(\eta))|^2 ]\\
& & \le CE [E [|\partial_x X^{t,\xi,P_\xi}_r|^2\ |\
{\cal F}_t ]|\eta|^2 ]+CE [|Y_r^{t,\xi}(\eta)|^2 ]\\
& &  \le C E [|\eta|^2 ]+CE [|Y_r^{t,\xi}(\eta)|^2 ].\end{array}\ee

\noindent Using this above estimate for SDE (\ref{3.26}) for $Y^{t,\xi}$, we
obtain that
\be\label{3.29}E [\sup_{s\in[t,T]}
|Y^{t,\xi}_s(\eta)|^2 ]\le CE[|\eta|^2],\ \eta\in L^2({\cal
F}_t).\ee

\noindent Moreover, from the uniqueness of the solution of SDE (\ref{3.26}) for
$Y^{t,\xi}$ we see that the mapping $Y^{t,\xi}_s(.): L^2({\cal
F}_t)\rightarrow L^2({\cal F}_s)$ is linear. Consequently, $Y^{t,\xi}_s(.): L^2({\cal F}_t)\rightarrow L^2({\cal F}_s)$\ is a linear and continuous
mapping.

Putting now $Z^{t,\xi}_s(\eta):=\partial_x X^{t,\xi,P_\xi}_s\cdot\eta
+Y^{t,\xi}_s(\eta),\ s\in[t,T]$, we see from equation (\ref{3.25}) for $\partial_x X^{t,\xi,P_\xi}$ and (\ref{3.26}) for
$Y^{t,\xi}$, that $Z^{t,\xi}\in {\cal S}^2_{\mathbb{F}}(t,T)$
solves the SDE
\be\label{3.30}\begin{array}{lll}
Z^{t,\xi}_s(\eta)&=&\displaystyle\eta+\int_t^s\left(\partial_x\sigma(X^{t,\xi}_r,
P_{X^{t,\xi}_r})Z^{t,\xi}_r(\eta)+D\widetilde{\sigma}(X^{t,\xi}_r,X^{t,
\xi}_r)(Z^{t,\xi}_r(\eta))\right)dB_r\\
& &+\displaystyle\int_t^s\left(\partial_xb(X^{t,\xi}_r,
P_{X^{t,\xi}_r})Z^{t,\xi}_r(\eta)+
D\widetilde{b}(X^{t,\xi}_r,X^{t,\xi}_r)(Z^{t,\xi}_r(\eta))\right)dr,\end{array}\ee

\noindent $s\in[t,T].$ But this is just the equation which is satisfied by the
directional derivative $\partial_\xi X^{t,\xi}_s(\eta)$ of $L^2({\cal F}_t)\ni\xi\rightarrow X^{t,\xi}_s\in
L^2({\cal F}_s),\ s\in[t,T],$ in the direction $\eta\in L^2({\cal F}_t).$ On
the other hand, from our estimate for $Y^{t,\xi}$ we also have that the
mapping $Z^{t,\xi}_s(.): L^2({\cal F}_t)\rightarrow L^2({\cal F}_s)$ is linear and continuous, i.e., $L^2({\cal F}_t)\ni\xi\rightarrow X^{t,\xi}_s\in
L^2({\cal F}_s)$ is G\^{a}teaux differentiable, and, for all $\eta\in
L^2({\cal F}_t)$,
\be\label{3.31} Z^{t,\xi}_s(\eta)=\partial_\xi X^{t,\xi}_s(\eta):=L^2\mbox{-}
\lim_{h\rightarrow 0}\frac{1}{h}\left(X^{t,\xi+h\eta}_s-X^{t,\xi}_s\right),\
s\in[t,T].\ee

\noindent For $x\in {\mathbb R}$, let us now consider the unique solution
$Y^{t,x,\xi}(\eta)=(Y^{t,x,\xi}_s(\eta))_{s\in[t,T]}\in{\cal
S}^2_{\mathbb{F}}(t,T)$ of the SDE
\be\label{3.32}\begin{array}{lll}
Y^{t,x,\xi}_s(\eta)&=&\displaystyle\int_t^s\left(\partial_x\sigma(X^{t,
x,P_\xi}_r,P_{X^{t,\xi}_r})Y^{t,x,\xi}_r(\eta)+
D\widetilde{\sigma}(X^{t,x,P_\xi}_r,X^{t,\xi}_r)(Z^{t,\xi}_r(\eta))
\right)dB_r \\
& &\displaystyle+\int_t^s\left(\partial_xb(X^{t,
x,P_\xi}_r,P_{X^{t,\xi}_r})Y^{t,x,\xi}_r(\eta)+
D\widetilde{b}(X^{t,x,P_\xi}_r,X^{t,\xi}_r)(Z^{t,\xi}_r(\eta))\right)dr.\end{array}\ee

\noindent $s\in[t,T].$ Taking into account the characterization of
$Z^{t,\xi}(\eta)$ obtained above and making standard estimates, we
see that also $L^2({\cal F}_t)\ni\xi\rightarrow
X^{t,x,\xi}_s(=X^{t,x,P_\xi}_s)\in L^2({\cal F}_s)$
is G\^{a}teaux differentiable, and the G\^{a}teau derivative
is just $Y^{t,x,\xi}_s(.)$:
\be\label{3.33}Y^{t,x,\xi}_s(\eta)=\partial_\xi X^{t,x,P_\xi}_s(\eta):
=L^2\mbox{-}\lim_{h\rightarrow
0}\frac{1}{h}\left(X^{t,x,\xi+h\eta}_s-X^{t,x,\xi}_s\right),\
s\in[t,T],\ \eta\in L^2({\cal F}_t).\ee

\noindent We remark that, since the following coefficients

\noindent$(\partial_x\sigma)(X_r^{t,x,P_\xi},P_{X^{t,\xi}_r}),\
(\partial_x b)(X_r^{t,x,P_\xi},P_{X^{t,\xi}_r}),$
$(D\tilde{\sigma}(X_r^{t,x,P_\xi},X^{t,\xi}_r)(Z^{t,\xi}_r(\eta))$, $(D\tilde{b}(X_r^{t,x,P_\xi},X^{t,\xi}_r)(Z^{t,\xi}_r(\eta))$

\noindent are independent of ${\cal F}_t$ (Recall
that $(D\tilde{\sigma}(z,X^{t,\xi}_r)(Z^{t,\xi}_r(
\eta))=E[\partial_\mu\sigma(z,P_{X^{t,\xi}_r},X^{t,\xi}_r)\cdot
Z^{t,\xi}_r(\eta)]$ is deterministic, for all $z\in {\mathbb R}$), also the
solution $Y^{t,x,\xi}(\eta)$ is independent of ${\cal F}_t$ and,
thus, of $\xi$. This allows to substitute for $x$ the random variable
$\xi$ in $Y^{t,x,\xi}(\eta)$, and from the equation (\ref{3.32}) for
$Y^{t,x,\xi}(\eta)$ and (\ref{3.26}) for $Y^{t,\xi}(\eta)$ and the uniqueness of their
solutions we get
\be\label{3.34}Y^{t,\xi}_s(\eta)=Y^{t,x,\xi}_s(\eta)\big|_{x=\xi},\,
s\in[t,T],\ \mbox{P-a.s.}\ee

\noindent Consequently, (\ref{3.33}) yields
\be\label{3.35}Y^{t,\xi}_s(\eta)=\partial_\xi X^{t,\xi,
P_\xi}_s(\eta)(:=\partial_\xi X^{t,x,P_\xi}_{s}(\eta)\big|_{x=\xi})=L^2\mbox{-}\lim_{h\rightarrow
0}\frac{1}{h}\left(X^{t,\xi,\xi+h\eta}_s-X^{t,\xi,\xi}_s\right), \ s\in[t, T],\ \eta\in L^2({\cal F}_t).\ee

\noindent\underline{Step 2}: A representation formula for $\partial_\xi X^{t,\xi,P_\xi}_s(\eta)$.

\medskip

Since $\xi\rightarrow X^{t,x,\xi}_s$ is G\^{a}teaux
differentiable, $\mu\rightarrow X^{t,x,\mu}_s$ is differentiable over ${\cal P}_2({\mathbb R})$ due to our
definition. Let us determine
this derivative at $P_{\xi}$.  For given
$\xi,\ \eta\in L^2({\cal F}_t)$ we denote by
$(\widetilde{\xi},\widetilde{\eta},\widetilde{B})$ a copy of
$(\xi,\eta,B)$ on $(\widetilde{\Omega},\widetilde{\cal F},
\widetilde{P})$; by $\widetilde{X}^{t,\widetilde{\xi}}$  we denote
the solution of the SDE for $X^{t,\xi}$ but now driven by the
Brownian motion $\widetilde{B}$ and with initial value
$\widetilde{\xi}$ instead of $\xi$, and $\widetilde{X}^{t,x,
\widetilde{P}_{\widetilde{\xi}}}$ denotes the solution of the SDE
for $X^{t,x,\xi}$ governed by $\widetilde{B}$:
\be\label{3.36}\widetilde{X}^{t,\widetilde{\xi}}_s=\widetilde{\xi}+\int_t^s\sigma
(\widetilde{X}^{t,\widetilde{\xi}}_r,\widetilde{P}_{\widetilde{X}^{t,
\widetilde{\xi}}_r})d\widetilde{B}_r+\int_t^sb(\widetilde{X}^{t,\widetilde{\xi}}_r,\widetilde{P}_{\widetilde{X}^{t,
\widetilde{\xi}}_r})dr,\ s\in[t,T],\ee
and
\be\label{3.37}\widetilde{X}^{t,x,\widetilde{P}_{\widetilde{\xi}}}_s=x+\int_t^s\sigma
(\widetilde{X}^{t,x,\widetilde{P}_{\widetilde{\xi}}}_r,
\widetilde{P}_{\widetilde{X}^{t,\widetilde{\xi}}_r})d\widetilde{B}_r+
\int_t^sb(\widetilde{X}^{t,x,\widetilde{P}_{\widetilde{\xi}}}_r,
\widetilde{P}_{\widetilde{X}^{t,\widetilde{\xi}}_r})dr,\
s\in[t,T].\ee

\noindent Obviously, $\widetilde{X}^{t,x,\widetilde{P}_{
\widetilde{\xi}}}=\widetilde{X}^{t,x,P_\xi},\ x\in {\mathbb R},$ and
$(\widetilde{\xi},\widetilde{\eta},\widetilde{X}^{t,x,P_\xi},
\widetilde{B})$ is an independent copy of $(\xi,\eta,X^{t,x,P_\xi},B)$,
defined over $(\widetilde{\Omega},\widetilde{\cal F},\widetilde{P})$.
Moreover, let $\partial_x\widetilde{X}^{t,x,P_\xi}_r$
denote the $L^2$ derivative of $\widetilde{X}^{t,
x,P_\xi}_r$, let $\partial_x\widetilde{X}^{t, \widetilde{\xi},P_\xi}_r
:=\partial_x\widetilde{X}^{t,x,P_\xi}_{r}\big|_{x=\widetilde{\xi}}$,
and by $\widetilde{Y}^{t,x,\widetilde{\xi}}(\widetilde{\eta})$ and
$\widetilde{Z}^{t,\widetilde{\xi}}(\widetilde{\eta})$
we denote the solutions of the equations for $Y^{t,x,\xi}(\eta)$ and $Z^{t,\xi}(\eta)$, respectively,
but with the data $(\widetilde{\xi},\widetilde{\eta},\widetilde{B},\widetilde{X}^{t,
x,P_\xi},\allowbreak \widetilde{X}^{t, \widetilde{\xi}})$
instead of $(\xi,\eta,B,X^{t, x,P_\xi},X^{t,\xi}).$

Using the such introduced notations, we have
\be\label{3.38}\begin{array}{lll}
& &D\widetilde{\sigma}(X^{t,x,\xi}_r,X^{t,\xi}_r)(Z^{t,\xi}_r(\eta)) =E [
\partial_\mu\sigma(z,P_{X^{t,\xi}_r},X^{t,\xi}_r)\cdot Z^{t,\xi}_r
(\eta) ]\big|_{z=X^{t,x,P_\xi}_r}\\
& & =
\widetilde{E} [\partial_\mu\sigma(X^{t,x,P_\xi}_r,P_{X^{t,\xi}_r},
\widetilde{X}^{t,\widetilde{\xi}}_r)\cdot\widetilde{Z}^{t,
\widetilde{\xi}}_r(\widetilde{\eta}) ],\end{array}\ee

\noindent and the same formula holds for $\widetilde{b}$. Thus, with the corresponding formula for
$D\widetilde{b}(X^{t,x,\xi}_r,X^{t,\xi}_r)(Z^{t,\xi}_r(\eta))$
we can give to SDE (\ref{3.26}) for $Y^{t,\xi}(\eta)$ the following form:
\be\label{3.39}\begin{array}{lll} Y^{t,\xi}_s(\eta)&=&\displaystyle\int_t^s\partial_x\sigma(X^{t,
\xi}_r,P_{X^{t,\xi}_r})Y^{t,\xi}_r(\eta)dB_r+\int_t^s
\partial_xb(X^{t,\xi}_r,P_{X^{t,\xi}_r})Y^{t,\xi}_r(\eta)dr\\
& &\displaystyle+\int_t^s
\widetilde{E} [\partial_\mu\sigma(X^{t,\xi}_r,P_{X^{t,\xi}_r},
\widetilde{X}^{t,\widetilde{\xi}}_r)\cdot
\widetilde{Z}^{t,\widetilde{\xi}}_r(\widetilde{\eta}) ]dB_r\\
& &\displaystyle+\int_t^s \widetilde{E} [\partial_\mu
b(X^{t,\xi}_r,P_{X^{t,\xi}_r},\widetilde{X}^{t,\widetilde{\xi}}_r)
\cdot\widetilde{Z}^{t,\widetilde{\xi}}_r(\widetilde{\eta}) ]
dr,\ s\in[t,T],\end{array}\ee
\noindent where $\widetilde{Z}^{t,\widetilde{\xi}}_r
(\widetilde{\eta})=\partial_x\widetilde{X}^{t,\widetilde{\xi},
P_\xi}_r\cdot\widetilde{\eta}+\widetilde{Y}^{t,\widetilde{\xi}}_r
(\widetilde{\eta}),\ r\in[t,T].$

\smallskip

\noindent Rewriting the equation for $Y^{t,x,P_\xi}(\eta)$ in the
same way, we obtain
\be\label{3.40}\begin{array}{lll} Y^{t,x,P_\xi}_s(\eta)&=&\displaystyle\int_t^s\partial_x
\sigma(X^{t,x,P_\xi}_r,P_{X^{t,\xi}_r})Y^{t,x,P_\xi}_r(\eta)dB_r
+\int_t^s\partial_x b(X^{t,x,P_\xi}_r,P_{X^{t,\xi}_r})Y^{t,x,P_\xi}_r(\eta)dr\\
& &\displaystyle+\int_t^s
\widetilde{E} [\partial_\mu\sigma(X^{t,x,P_\xi}_r,
P_{X^{t,\xi}_r},\widetilde{X}^{t,\widetilde{\xi}}_r)\cdot
\widetilde{Z}^{t,\widetilde{\xi}}_r(\widetilde{\eta}) ]dB_r\\
& &\displaystyle+\int_t^s \widetilde{E} [\partial_\mu
b(X^{t,x,P_\xi}_r,P_{X^{t,\xi}_r},\widetilde{X}^{t,
\widetilde{\xi}}_r)\cdot\widetilde{Z}^{t,\widetilde{\xi}}_r
(\widetilde{\eta}) ]
dr,\ s\in[t,T].\end{array}\ee

\smallskip

In order to analyze the structure of the above equations and
their solutions, let us also consider, for $y\in {\mathbb R}$, the unique
solution $U^{t,\xi}(y)=(U^{t,\xi}_s(y))_{s\in[t,T]}\in {\cal
S}_{\mathbb{F}}^2(t,T)$ of the SDE
\be\label{3.41}\begin{array}{lll} & U^{t,\xi}_s(y)= \displaystyle\int_t^s\partial_x\sigma(X^{t,\xi}_r,
P_{X^{t,\xi}_r})U^{t,\xi}_r(y)dB_r+\int_t^s\partial_xb(X^{t,
\xi}_r,P_{X^{t,\xi}_r})U^{t,\xi}_r(y)dr\\
&\ \ \ \displaystyle + \int_t^s \widetilde{E} [(\partial_\mu\sigma)(X^{t,\xi}_r,P_{X^{t,\xi}_r},
\widetilde{X}^{t,y,P_\xi}_r)\cdot \partial_x\widetilde{X}^{t,
y,P_\xi}_r+(\partial_\mu\sigma)(X^{t,\xi}_r,P_{X^{t,\xi}_r},
\widetilde{X}^{t,\widetilde{\xi}}_r)\cdot
\widetilde{U}^{t,\widetilde{\xi}}_r(y) ]dB_r\\
&\ \ \ \displaystyle + \int_t^s
\widetilde{E} [(\partial_\mu b)(X^{t,\xi}_r,P_{X^{t,\xi}_r},
\widetilde{X}^{t,y,P_\xi}_r)\cdot \partial_x\widetilde{X}^{t,
y,P_\xi}_r+(\partial_\mu b)(X^{t,\xi}_r,P_{X^{t,\xi}_r},
\widetilde{X}^{t,\widetilde{\xi}}_r)\cdot
\widetilde{U}^{t,\widetilde{\xi}}_r(y) ]dr,\end{array}\ee

\noindent $s\in[t,T]$, where $(\widetilde{U}^{t,\widetilde{\xi}}(y),
\widetilde{B})$ is supposed to follow under $\widetilde{P}$ exactly
the same law as $(U^{t,\xi}(y),B)$ under $B$ (one can consider
$\widetilde{U}^{t,\widetilde{\xi}}(y)$ as the unique solution of the
SDE for $U^{t,\xi}(y)$, but with the data $(\widetilde{\xi},
\widetilde{B})$ instead of $(\xi,B)$). Since the derivatives
$\partial_x\sigma,\ \partial_xb,\, \partial_\mu\sigma$ and
$\partial_\mu b$ are bounded and the process $\partial_xX^{t,y,P_\xi}$
is bounded in $L^2$ by a constant independent of $y\in {\mathbb R}^d$, it is
easy to prove the existence of the solution $U^{t,\xi}(y)$ for the
above SDE (\ref{3.41}) and to show that it is bounded in $L^2$ by a constant
independent of $y\in {\mathbb R}^d$.

The process $U^{t,\xi}(y)$ introduced above also allows to consider,
for all $x\in {\mathbb R}$, the unique solution $U^{t,x,P_\xi}(y)\in{\cal S}^2_{
\mathbb{F}}(t,T)$ of the following SDE:
\be\label{3.42}\begin{array}{lll} &U^{t,x,P_\xi}_s(y)=\displaystyle\int_t^s\partial_x
\sigma(X^{t,x,P_\xi}_r,P_{X^{t,\xi}_r})U^{t,x,P_\xi}_r(y)dB_r+
\int_t^s\partial_xb(X^{t,x,P_\xi}_r,P_{X^{t,\xi}_r})U^{t,x,
P_\xi}_r(y)dr\\
&\ \ \ \displaystyle+\int_t^s\widetilde{E} [(\partial_\mu
\sigma)(X^{t,x,P_\xi}_r,P_{X^{t,\xi}_r},
\widetilde{X}^{t,y,P_\xi}_r)\cdot \partial_x\widetilde{X}^{t,
y,P_\xi}_r+(\partial_\mu\sigma)(X^{t,x,P_\xi}_r,P_{X^{t,
\xi}_r},\widetilde{X}^{t,\widetilde{\xi}}_r)\widetilde{U}^{t,
\widetilde{\xi}}_r(y) ]dB_r\\
&\ \  \ \displaystyle+\int_t^s
\widetilde{E} [(\partial_\mu b)(X^{t,x,P_\xi}_r,
P_{X^{t,\xi}_r},\widetilde{X}^{t,y,P_\xi}_r)\cdot
\partial_x\widetilde{X}^{t,y,P_\xi}_r+(\partial_\mu b)(X^{t,x,
P_\xi}_r,P_{X^{t,\xi}_r},\widetilde{X}^{t,\widetilde{\xi}}_r)
\widetilde{U}^{t,\widetilde{\xi}}_r(y) ]dr,\end{array}\ee

\noindent $s\in[t,T]$. It is easy to verify that the solution
$U^{t,x,P_\xi}(y)$ is $(\sigma\{B_r-B_t,\ r\in[t,s]\})$-adapted
and, hence, independent of ${\cal F}_t$ and, in particular, of
$\xi$ (See the corresponding discussion
we made for $Y^{t,x,P_\xi}(\eta)$). Consequently, we can
substitute in $U^{t,x,P_\xi}(y)$ the random variable $\xi$
for $x$, and from the uniqueness of the solution of the
equation for $U^{t,\xi}(y)$ we deduce that
\be\label{3.43}U^{t,\xi}_s(y)=U^{t,x,P_\xi}_s(y)_{\big|x=\xi}
,\, s\in[t,T],\  \mbox{P-a.s.}\ee

\noindent The same argument of $(\sigma\{B_r-B_t,\
r\in[t,s]\})$-adaptedness allows also to substitute a random
variable for $y$ in  $U^{t,\xi}_s(y),$ which is independent
of the $\sigma\{B_r-B_t,\ r\in[t,T]\}$.

Let now $(\widehat{\xi}, \widehat{\eta}, \widehat{B})$
be a copy of $(\xi,\eta, B)$, independent of $(\xi,\eta,B)$
and $(\widetilde{\xi},\widetilde{\eta},\widetilde{B})$,
and defined over a new probability space $(\widehat{\Omega},
\widehat{\cal F},\widehat{P})$\ which is different from $(\Omega,
{\cal F}, P)$ and $(\widetilde{\Omega},\widetilde{\cal F},
\widetilde{P})$; the expectation $\widehat{E}[.]$ applies
only to random variables over $(\widehat{\Omega},
\widehat{\cal F},\widehat{P})$. This extension to
$(\widehat{\xi},\widehat{\eta},\widehat{E}[.])$ here is done
in the same spirit as that from $(\xi,\eta,B)$ to
$(\widetilde{\xi},\widetilde{\eta},\widetilde{B}).$

We will show that $\partial_\xi X^{t,x,\xi}_s(\eta)=Y^{t,x,\xi}_s
(\eta)=\widehat{E}[U^{t,x,P_\xi}_s(\widehat{\xi})\cdot
\widehat{\eta}],\ s\in[t,T]$, which would complete the proof concerning the representation formula for the G\^{a}teaux derivative (Recall that the expectation $\widehat{E}$ acts only on $\widehat{\eta}$). For this end, we prove first in a preparing
step that
\be\label{3.44}Y^{t,\xi}_s(\eta)=\widehat{E}[U^{t,\xi}_s
(\widehat{\xi})\cdot\widehat{\eta}],\ s\in[t,T].\ee

\noindent In order to obtain the above relation, we substitute in the
equation for $U^{t,\xi}(y)$ for $y$ the random variable
$\widehat{\xi}$ and we multiply both sides of the such
obtained equation by $\widehat{\eta}$ (Recall that
$\widehat{\xi}$ and $\widehat{\eta}$ are independent of all terms in the equation for $U^{t,\xi}(y)$). Then we take
the expectation $\widehat{E}[.]$ on both sides of the new
equation. This yields
\be\label{3.45}\begin{array}{lll} \widehat{E}[U^{t,\xi}_s
(\widehat{\xi})\cdot\widehat{\eta}]&=\displaystyle\int_t^s\partial_x\sigma(X^{t,\xi}_r,
P_{X^{t,\xi}_r})\widehat{E}[U^{t,\xi}_r(\widehat{\xi})\cdot
\widehat{\eta}]dB_r+\int_t^s\partial_xb(X^{t,
\xi}_r,P_{X^{t,\xi}_r})\widehat{E}[U^{t,\xi}_r(
\widehat{\xi})\cdot\widehat{\eta}]dr\\
&\ \ \ \displaystyle+\int_t^s
\widetilde{E} [\widehat{E} [(\partial_\mu\sigma)(X^{t,
\xi}_r,P_{X^{t,\xi}_r},\widetilde{X}^{t,\widehat{\xi},P_\xi}_r)
\cdot \partial_x\widetilde{X}^{t,\widehat{\xi},P_\xi}_r
\cdot\widehat{\eta} ] ]dB_r\\
&\ \ \ \displaystyle+\int_t^s
\widetilde{E} [(\partial_\mu\sigma)(X^{t,\xi}_r,
P_{X^{t,\xi}_r},\widetilde{X}^{t,\widetilde{\xi}}_r)\cdot
\widehat{E}[\widetilde{U}^{t,\widetilde{\xi}}_r(\widehat{\xi})
\cdot\widehat{\eta}] ]dB_r\\
&\ \ \ \displaystyle+\int_t^s
\widetilde{E} [\widehat{E} [(\partial_\mu b)(X^{t,
\xi}_r,P_{X^{t,\xi}_r},\widetilde{X}^{t,\widehat{\xi},P_\xi}_r)
\cdot \partial_x\widetilde{X}^{t,\widehat{\xi},P_\xi}_r
\cdot\widehat{\eta} ] ]dB_r\\
&\ \ \ \displaystyle+\int_t^s
\widetilde{E} [(\partial_\mu b)(X^{t,\xi}_r,
P_{X^{t,\xi}_r},\widetilde{X}^{t,\widetilde{\xi}}_r)\cdot
\widehat{E}[\widetilde{U}^{t,\widetilde{\xi}}_r(\widehat{\xi})
\cdot\widehat{\eta}] ]dr,\  s\in[t,T].\end{array}\ee

\noindent Taking into account that $(\widehat{\xi},
\widehat{\eta})$ is independent of $(\xi,\eta,B)$ and
$(\widetilde{\xi},\widetilde{\eta},\widetilde{B})$, and
of the same law under $\widehat{P}$ as $(\widetilde{\xi},
\widetilde{\eta})$ under $\widetilde{P}$, we see that
\be\label{3.46}\begin{array}{lll}
&\widetilde{E} [\widehat{E} [(\partial_\mu\sigma)(X^{t,
\xi}_r,P_{X^{t,\xi}_r},\widetilde{X}^{t,\widehat{\xi},P_\xi}_r)
\cdot \partial_x\widetilde{X}^{t,\widehat{\xi},P_\xi}_r
\cdot\widehat{\eta} ] ]\\
&=\widetilde{E} [(\partial_\mu\sigma)(X^{t,
\xi}_r,P_{X^{t,\xi}_r},\widetilde{X}^{t,\widetilde{\xi}}_r)
\cdot \partial_x\widetilde{X}^{t,\widetilde{\xi},P_\xi}_r
\cdot\widetilde{\eta} ],\end{array}\ee

\noindent  and the same relation also holds true for
$\partial_\mu b$ instead of $\partial_\mu\sigma.$ Thus,
the above equation takes the form
\be\label{3.47}\begin{array}{lll} \widehat{E}[U^{t,\xi}_s
(\widehat{\xi})\cdot\widehat{\eta}]=&\displaystyle\int_t^s\partial_x\sigma(X^{t,\xi}_r,
P_{X^{t,\xi}_r})\widehat{E}[U^{t,\xi}_r(\widehat{\xi})\cdot
\widehat{\eta}]dB_r+\int_t^s\partial_xb(X^{t,
\xi}_r,P_{X^{t,\xi}_r})\widehat{E}[U^{t,\xi}_r(
\widehat{\xi})\cdot\widehat{\eta}]dr\\
&\displaystyle +\int_t^s
\widetilde{E}[(\partial_\mu\sigma)(X^{t,\xi}_r,
P_{X^{t,\xi}_r},\widetilde{X}^{t,\widetilde{\xi}}_r)
\cdot (\partial_x\widetilde{X}^{t,\widetilde{\xi},P_\xi}_r\cdot
\widetilde{\eta}+\widehat{E}[\widetilde{U}^{t,\widetilde{\xi}}_r
(\widehat{\xi})\cdot\widehat{\eta}])]dB_r\\
&\displaystyle +\int_t^s
\widetilde{E}[(\partial_\mu b)(X^{t,\xi}_r,
P_{X^{t,\xi}_r},\widetilde{X}^{t,\widetilde{\xi}}_r)
\cdot(\partial_x\widetilde{X}^{t,\widetilde{\xi},P_\xi}_r\cdot
\widetilde{\eta}+\widehat{E}[\widetilde{U}^{t,\widetilde{\xi}}_r
(\widehat{\xi})\cdot\widehat{\eta}])]dr,\  s\in[t,T].\\ \end{array}\ee

\noindent But this latter SDE is just (\ref{3.26}) for $Y^{t,\xi}(\eta)$,
and from the uniqueness of the solution of this equation it
follows that
\be\label{3.48} Y^{t,\xi}_s(\eta)=\widehat{E}[U^{t,\xi}_s
(\widehat{\xi})\cdot\widehat{\eta}],\ s\in[t,T],\ee
\noindent i.e.,
\be\label{3.49} Z^{t,\xi}_s(\eta)=\partial_xX^{t,\xi,P_\xi}_s
\cdot\eta+\widehat{E}[U^{t,\xi}_s
(\widehat{\xi})\cdot\widehat{\eta}],\ s\in[t,T].\ee

\noindent Finally, we substitute $\widehat{\xi}$ for $y$
in the SDE for $U^{t,x,P_\xi}(y)$, we multiply both sides of the
such obtained equation by $\widehat{\eta}$ and take after
the expectation $\widehat{E}[.]$ at both sides of the relation.
Using the results of the above discussion, we see that this
yields
\be\label{3.50}\begin{array}{lll} &\widehat{E}[U^{t,x,P_\xi}_s
(\widehat{\xi})\cdot\widehat{\eta}]\\
&=\displaystyle\int_t^s\partial_x\sigma(X^{t,x,P_\xi}_r,
P_{X^{t,\xi}_r})\widehat{E}[U^{t,x,P_\xi}_r(\widehat{\xi})\cdot
\widehat{\eta}]dB_r+\int_t^s\partial_x b(X^{t,
x,P_\xi}_r,P_{X^{t,\xi}_r})\widehat{E}[U^{t,x,P_\xi}_r(
\widehat{\xi})\cdot\widehat{\eta}]dr\\
&\ \ \displaystyle+\int_t^s
\widetilde{E}[(\partial_\mu\sigma)(X^{t,x,P_\xi}_r,
P_{X^{t,\xi}_r},\widetilde{X}^{t,\widetilde{\xi}}_r)
\cdot (\partial_x\widetilde{X}^{t,\widetilde{\xi},P_\xi}_r\cdot
\widetilde{\eta}+\widehat{E}[\widetilde{U}^{t,\widetilde{\xi}}_r
(\widehat{\xi})\cdot\widehat{\eta}])]dB_r\\
&\ \ \displaystyle+\int_t^s
\widetilde{E}[(\partial_\mu b)(X^{t,x,P_\xi}_r,
P_{X^{t,\xi}_r},\widetilde{X}^{t,\widetilde{\xi}}_r)
\cdot(\partial_x\widetilde{X}^{t,\widetilde{\xi},P_\xi}_r\cdot
\widetilde{\eta}+\widehat{E}[\widetilde{U}^{t,\widetilde{\xi}}_r
(\widehat{\xi})\cdot\widehat{\eta}])]dr,\  s\in[t,T].\end{array}\ee

\noindent Using that $\widetilde{Z}^{t,\widetilde{\xi}}_s
(\widetilde{\eta})=\partial_x\widetilde{X}^{t,
\widetilde{\xi},P_\xi}_s\cdot\widetilde{\eta}+
\widehat{E}[\widetilde{U}^{t,\widetilde{\xi}}_s
(\widehat{\xi})\cdot\widehat{\eta}],\ s\in[t,T]$ (see (\ref{3.49})), we see that the latter SDE is just that satisfied by
$Y^{t,x,P_\xi}(\eta)$. Therefore, from the uniqueness of the
solution of this SDE it follows that
\be\label{3.51}\partial_\xi {X}^{t,x,\xi}_s(\eta)=Y^{t,x,P_\xi}_s(\eta)=\widehat{E}[U^{t,x,P_\xi}_s
(\widehat{\xi})\cdot\widehat{\eta}],\ s\in[t,T],\ \eta\in L^2({\cal F}_t, P).\ee

\noindent The proof for the G\^{a}teaux derivative is complete now.\end{proof}

After having proved the existence of the G\^{a}teaux derivative
$\partial_\xi X^{t,x,\xi}_s(\eta)$\ and its representation formula, let us study the regularity of
$U^{t,x,P_\xi}(y)$. In a first step we study it
under the assumption of the preceding proposition, i.e.,
under Hypothesis (H.1).
\begin{lemma} Assume (H.1) and let $U^{t,x,P_\xi}$\ denote the unique solution of (\ref{3.42}). Then, for all
$p\ge 2$ there is some constant $C_p\in {\mathbb R}$ such that, for all $t\in [0,T],\ x,\ x',\ y,\ y'\in {\mathbb R}^d$ and $\xi,\ \xi'\in L^2({\cal F}_t;{\mathbb R}^d)$,
\be\label{3.52}\begin{array}{lll}
&{\rm i)}\ E [\sup_{s\in[t,T]}|U^{t,x, P_\xi}_s(y)|^p]\le C_p,\\
&{\rm ii)}\ E [\sup_{s\in[t,T]}|U^{t,x,P_\xi}_s(y)- U^{t,x',P_{\xi'}}_s(y')|^p ]\le C_p(|x-x'|^p+|y-y'|^p+W_2(P_\xi,P_{\xi'})^p).\\
\end{array}\ee
\end{lemma}
\begin{proof} As in the preceding proof we restrict ourselves to the one-dimensional case $d=1$,
and to simplify a bit more we suppose also that $b=0$. Estimates
which are standard, are only indicated here without proof.

Let us recall from Lemma 3.1 and its proof that, for all $p\ge 2$,
there is a constant $C_p\in {\mathbb R}$ such that, for all $t\in[0,T],\
x,\ x'\in {\mathbb R},\ \xi,\ \xi'\in L^2({\cal F}_t)$,
\be\label{3.53}\begin{array}{lll}
&{\rm i)}\ E [\sup_{s\in[t,T]}|X^{t,x,P_\xi}_s|^p ]\le C_p(1+|x|^p),\\
&{\rm ii)}\ E [\sup_{s\in[t,T]}|X^{t,x,P_\xi}_s- X^{t,x',P_{\xi'}}_s|^p ]\le C_p(|x-x'|^p+W_2(P_\xi,P_{\xi'})^p),\\
&{\rm iii)}\ W_2(P_{X^{t,\xi}_s},P_{X^{t,\xi'}_s})\le C_2W_2(P_\xi,P_{\xi'}).\end{array}\ee

\noindent Next we consider the SDE for $\partial_xX^{t,x,P_\xi}$,
\be\label{3.54}\partial_xX^{t,x,P_\xi}_s=1+\int_t^s(\partial_x\sigma)
(X^{t,x,P_\xi}_r,P_{X^{t,\xi}_r})\partial_xX^{t,x,P_\xi}_r dB_r,\
s\in[t,T].\ee

\noindent With standard estimates we see here that, for all $p\ge 2,$
there exists a constant $C_p$ such that, for all $(t,x),\ (t',x')\in[0,T]
\times {\mathbb R},\ \xi,\ \xi'\in L^2({\cal F}_t)$,
\be\label{3.55}\begin{array}{lll}
&{\rm iv)}\  E [\sup_{s\in[0,T]}|\partial_xX^{t,x,P_\xi}_s|^p ]\le C_p,\\
&{\rm v)}\   E [\sup_{s\in[0,T]}|\partial_xX^{t,x,P_\xi}_s-\partial_xX^{t,x',P_{\xi'}}_s|^p ]\le C_p(|x-x'|^p+W_2(P_{\xi},P_{\xi'})^p).\end{array}\ee

\noindent Using the notations from the proof of Theorem 3.1, $U^{t,x,P_\xi}(y)$ is the unique
solution of the SDE
\be\label{3.56}\begin{array}{lll} &U^{t,x,P_\xi}_s(y)=\displaystyle\int_t^s\partial_x
\sigma(X^{t,x,P_\xi}_r,P_{X^{t,\xi}_r})U^{t,x,P_\xi}_r(y)dB_r\\
&\ \displaystyle+\int_t^s\widetilde{E}[(\partial_\mu
\sigma)(X^{t,x,P_\xi}_r,P_{X^{t,\xi}_r},
\widetilde{X}^{t,y,P_\xi}_r)\cdot \partial_x\widetilde{X}^{t,
y,P_\xi}_r+(\partial_\mu\sigma)(X^{t,x,P_\xi}_r,P_{X^{t,
\xi}_r},\widetilde{X}^{t,\widetilde{\xi}}_r)\widetilde{U}^{t,
\widetilde{\xi}}_r(y)]dB_r,\end{array}\ee

\noindent where
\be\label{3.57}\begin{array}{lll} &U^{t,\xi}_s(y)=\displaystyle\int_t^s\partial_x\sigma(X^{t,\xi}_r,
P_{X^{t,\xi}_r})U^{t,\xi}_r(y)dB_r\\
&\ \ \displaystyle+\int_t^s
\widetilde{E}[(\partial_\mu\sigma)(X^{t,\xi}_r,P_{X^{t,\xi}_r},
\widetilde{X}^{t,y,P_\xi}_r)\cdot \partial_x\widetilde{X}^{t,
y,P_\xi}_r+(\partial_\mu\sigma)(X^{t,\xi}_r,P_{X^{t,\xi}_r},
\widetilde{X}^{t,\widetilde{\xi}}_r)\cdot
\widetilde{U}^{t,\widetilde{\xi}}_r(y)]dB_r,\ \end{array}\ee

\noindent$s\in[t,T].$ Again by standard arguments we see that, for all
$t\in[0,T],\ x, x', y, y'\in {\mathbb R}$ and $\xi,\ \xi'\in L^2({\cal F}_t)$,
\be\label{3.58}\begin{array}{lll}
&{\rm vi)}\  E[\sup_{s\in[t,T]}(|U^{t,x,
P_\xi}_s(y)|^p+|U^{t,\xi}_s(y)|^p)]\le C_p,\\
&{\rm vii)}\  E[\sup_{s\in[t,T]}(|U^{t,x,
P_\xi}_s(y)- U^{t,x',P_{\xi'}}_s(y')|^p+|U^{t,\xi}_s(y)- U^{t,\xi'}_s(y')|^p
)]\\
&\ \ \ \ \ \ \ \le C_p(|x-x'|^p+|y-y'|^p+W_2(P_\xi,P_{\xi'})^p).\end{array}\ee

\noindent The proof of vi) is trivial. In order to prove vii), we notice that its central ingredient is the following
estimate, which uses the Lipschitz property of $\partial_\mu\sigma$
with respect to all its variables as well as the boundedness of $\partial_\mu\sigma$
and the estimate vi):
\be\label{3.59}\begin{array}{lll}
&E[|\widetilde{E}[(\partial_\mu
\sigma)(X^{t,x,P_\xi}_r,P_{X^{t,\xi}_r},
\widetilde{X}^{t,y,P_\xi}_r)\cdot \partial_x\widetilde{X}^{t,
y,P_\xi}_r-(\partial_\mu
\sigma)(X^{t,x',P_{\xi'}}_r,P_{X^{t,\xi'}_r},
\widetilde{X}^{t,y',P_{\xi'}}_r)\cdot \partial_x\widetilde{X}^{t,
y',P_{\xi'}}_r]|^p]\\
&\le C_p (E[|X^{t,x,P_\xi}_r-X^{t,x',P_{\xi'}}_r|^{2p}+
(\widetilde{E}[|\widetilde{X}^{t,y,P_\xi}_r-
\widetilde{X}^{t,y',P_{\xi'}}_r|^2])^p]
+W_2(P_{X^{t,\xi}_r},P_{X^{t,\xi'}_r})^{2p})^{1/2}\\
&\ \ + C_p(\widetilde{E}[|\partial_x
\widetilde{X}^{t,y,P_\xi}_s- \partial_x\widetilde{X}^{t,y',
P_{\xi'}}_s|^2])^{p/2}.\end{array}\ee

\noindent Once having the above estimate, we can use now the estimates
ii), iii) and v), in order to deduce that
\be\label{3.60}\begin{array}{lll}
&E[\|\widetilde{E}[(\partial_\mu
\sigma)(X^{t,x,P_\xi}_r,P_{X^{t,\xi}_r},
\widetilde{X}^{t,y,P_\xi}_r)\cdot \partial_x\widetilde{X}^{t,
y,P_\xi}_r-(\partial_\mu
\sigma)(X^{t,x',P_{\xi'}}_r,P_{X^{t,\xi'}_r},
\widetilde{X}^{t,y',P_{\xi'}}_r)\cdot \partial_x\widetilde{X}^{t,
y',P_{\xi'}}_r]\|^p]\\
&\ \ \le C_p(|x-x'|^p+|y-y'|^p+W_2(P_\xi,P_{\xi'})^p).\end{array}\ee

\smallskip

\noindent Using this latter estimate, the proof of vii) reduces to an
application of Gronwall's Lemma. This completes the proof.\end{proof}

We are now able to complete the proof of Theorem 3.1 by showing that the G\^{a}teaux derivative $\partial_{\xi} {X}^{t,
x,P_\xi}_s(.)$\ of $L^2({\cal F}_t, {\mathbb R}^d)\ni \xi \rightarrow {X}^{t,x,P_\xi}_s\in L^2({\cal F}_s, {\mathbb R}^d)$\ is in fact a Fr\'{e}chet derivative.
\begin{proof} (of Theorem 3.1 (sequel)). We restrict ourselves again to the case $d=1$. In order to show that the G\^{a}teaux derivative $\partial_{\xi} {X}^{t,
x,\xi}_s(.)$\ of $L^2({\cal F}_t)\ni \xi \mapsto {X}^{t,x,\xi}_s\in L^2({\cal F}_s)$\ is even a Fr\'{e}chet derivative, we have to prove that, for all $(t, x)\in [0, T]\times {\mathbb R}$, and $\xi\in L^2({\cal F}_t),$
$$E[\sup_{s\in [t, T]}|{X}^{t,x,\xi+\eta}_s-{X}^{t,x,\xi}_s-\partial_{\xi} {X}^{t,x,\xi}_s(\eta)|^2] =o(|\eta|^2_{L^2}),$$

\noindent as $|\eta|_{L^2}\rightarrow 0\ (\eta\in L^2({\cal F}_t))$.

Fixing $(t, x)\in [0, T]\times {\mathbb R}$, and $\xi\in L^2({\cal F}_t)$, and letting $\eta\in L^2({\cal F}_t)$, we observe that the G\^{a}teaux differentiability of $\xi\mapsto {X}^{t,x,\xi}_s$ has as consequence the differentiability of the $L^2({\cal F}_s)$-valued function $$\varphi_s(r):={X}^{t,x,\xi+r\eta}_s,\ r\in {\mathbb R},$$
in all $r\in {\mathbb R}$, and $\frac{d}{dr}\varphi_s(r)=\partial_\xi{X}^{t,x,\xi+r\eta}_s,\ r\in {\mathbb R}$. Consequently,
$${X}^{t,x,\xi+\eta}_s-{X}^{t,x,\xi}_s=\int_0^1\partial_\xi{X}^{t,x,\xi+r\eta}_s(\eta)dr,\ s\in [t, T],$$
and
$$\begin{array}{lll}
& E[\sup_{s\in [t, T]}|{X}^{t,x,\xi+\eta}_s-{X}^{t,x,\xi}_s-\partial_{\xi} {X}^{t,x,\xi}_s(\eta)|^2]\\
&\leq \displaystyle E[\sup_{s\in [t, T]}\int_0^1|\partial_\xi{X}^{t,x,\xi+r\eta}_s(\eta)-\partial_\xi{X}^{t,x,\xi}_s(\eta)|^2dr].
\end{array}
$$
Hence, due to (\ref{3.51}), with the notations of the preceding proof,
$$\begin{array}{lll}
& E[\sup_{s\in [t, T]}|{X}^{t,x,\xi+\eta}_s-{X}^{t,x,\xi}_s-\partial_{\xi} {X}^{t,x,\xi}_s(\eta)|^2]\\
&\leq \displaystyle E[\sup_{s\in [t, T]}\int_0^1| \widehat{{E}}[({U}^{t,x,P_{\xi+r\eta}}_s(\widehat{\xi}+r\widehat{\eta})-U^{t,x,P_{\xi}}_s( \widehat{\xi}))\widehat{\eta}]|^2dr]\\
&\leq \displaystyle \widehat{E}[|\widehat{\eta}|^2]\int_0^1 \widehat{E}[E[\sup_{s\in [t, T]}|{U}^{t,x,P_{\xi+r\eta}}_s(\widehat{\xi}+r\widehat{\eta})-U^{t,x,P_{\xi}}_s(\widehat{\xi})|^2]]dr.
\end{array}
$$
We observe that (\ref{3.52}) yields
$$\begin{array}{lll}
&\displaystyle E[\sup_{s\in [t, T]}|{U}^{t,x,P_{\xi+r\eta}}_s(\widehat{\xi}+r\widehat{\eta})-U^{t,x,P_{\xi}}_s( \widehat{\xi})|^2]\\
&\leq \displaystyle C(|\widehat{\eta}|^2+W_2(P_{\xi+r\eta}, P_{\xi})^2)\\
&\leq \displaystyle C(|\widehat{\eta}|^2+E[\eta^2]).
\end{array}
$$
Consequently,
$$E[\sup_{s\in [t, T]}|{X}^{t,x,\xi+\eta}_s-{X}^{t,x,\xi}_s-\partial_{\xi} {X}^{t,x,\xi}_s(\eta)|^2]\leq CE[\eta^2]=o(|\eta|^2_{L^2}).$$
Therefore, $\xi\mapsto {X}^{t,x,\xi}_s$ is Fr\'{e}chet differentiable, the Fr\'{e}chet derivative $D{X}^{t,x,\xi}_s(.)$ coincides with $\partial_{\xi}{X}^{t,x,\xi}_s(.)$ and we have (\ref{3.16}). The proof of Theorem 3.1 is complete.
\end{proof}

\section{Second order derivatives of ${X}^{t,x,P_\xi}$}

Let us come now to the study of the second order derivatives of
the process $X^{t,x,P_\xi}$. For this we shall suppose the following in the remaining part of the paper:

\noindent\underline{\textbf{Hypothesis (H.2)}} Let $(\sigma,b)$ belong to
$C^{2,1}_b({\mathbb R}^d\times {\cal P}_2({\mathbb R}^d)\rightarrow {\mathbb R}^{d\times d}\times
{\mathbb R}^d)$, i.e.,  $(\sigma, b)\in C^{1,1}_b({\mathbb R}^d\times {\cal P}_2({\mathbb R}^d)
\rightarrow {\mathbb R}^{d\times d}\times {\mathbb R}^d)$ (See Hypothesis (H.1)) and the
derivatives of the components $\sigma_{i,j},\ b_j$, $1\le i, j\le d$,
have the following properties:

\smallskip

i) $\partial_k\sigma_{i,j}(.,.),\partial_kb_j(.,.)$ belong to $C^{1,l}
({\mathbb R}^d \times {\cal P}_2({\mathbb R}^d))$, for all $1\le k\le d$;

ii)$\partial_\mu\sigma_{i,j}(.,.,.),\partial_\mu b_j(.,.,.)$ belongs to
$C^{1,1}({\mathbb R}^d\times {\cal P}_2({\mathbb R}^d)\times {\mathbb R}^d)$;

iii) All the derivatives of $\sigma_{i,j},b_j$ up to order 2 are
bounded and Lipschitz.

\br With the existence of the second order mixed derivatives, $\partial_{x_l}(\partial_\mu\sigma_{i,j}(x,\mu,y))$\ and $\partial_\mu(\partial_{x_l}\sigma_{i,j}(x,\mu))(y)$, $(x, \mu, y)\in {\mathbb R}^d\times {\cal P}_2({\mathbb R}^d)\times {\mathbb R}^d$, the question of their equality raises. Indeed, under hypothesis (H.2) they coincide, and similar to those for $b$. More precisely, we have the following statement.
\er
\bl Let $g\in C^{2,1}_b({\mathbb R}^d\times {\cal P}_2({\mathbb R}^d)\rightarrow {\mathbb R}^{d\times d}\times
{\mathbb R}^d)$ (in the sense of hypothesis (H.2)). Then, for all $1\leq l\leq d$,
$$\partial_{x_l}(\partial_\mu g(x,\mu,y))=\partial_\mu(\partial_{x_l}g(x,\mu))(y),\ (x, \mu, y)\in {\mathbb R}^d\times {\cal P}_2({\mathbb R}^d)\times {\mathbb R}^d.$$
\el
\begin{proof} Let us restrict to $d=1$. Following the argument of Clairot's Theorem, we have, for all $(x, \xi),\ (z, \eta)\in {\mathbb R}^d\times L^2(\cal{F})$,
$$\begin{array}{lll}
I:=&\displaystyle(g(x+z, P_{\xi+\eta})-g(x+z, P_{\xi}))-(g(x, P_{\xi+\eta})-g(x, P_{\xi}))\\
=&\displaystyle\int_0^1 E[\left((\partial_\mu g)(x+z, P_{\xi+s\eta}, \xi+s\eta)- (\partial_\mu g)(x, P_{\xi+s\eta}, \xi+s\eta)\right)\eta]ds\\
=&\displaystyle\int_0^1\int_0^1E[\partial_x(\partial_\mu g)(x+tz, P_{\xi+s\eta}, \xi+s\eta)\eta z]dsdt\\
=&\displaystyle E[\partial_x(\partial_\mu g)(x, P_{\xi}, \xi)\eta z]+R_1((x,\xi), (z, \eta)),
\end{array}
$$
\noindent with $|R_1((x,\xi), (z, \eta))|\leq 2C |z|^2 E[\eta^2]$, and at the same time
$$\begin{array}{lll}
I:=&\displaystyle(g(x+z, P_{\xi+\eta})-g(x, P_{\xi+\eta}))-(g(x+z, P_{\xi})-g(x, P_{\xi}))\\
=&\displaystyle\int_0^1  \left((\partial_x g)(x+tz, P_{\xi+\eta})- (\partial_x g)(x+tz, P_\xi)\right)dt\cdot z\\
=&\displaystyle\int_0^1\int_0^1E[\partial_\mu(\partial_x g)(x+tz, P_{\xi+s\eta})(\xi+s\eta)\eta]dsdt\cdot z\\
=&\displaystyle E[\partial_\mu(\partial_x g)(x, P_{\xi}, \xi)\eta] z+R_2((x,\xi), (z, \eta)),
\end{array}
$$
\noindent with $|R_2((x,\xi), (z, \eta))|\leq 2C |z|^2 E[\eta^2]$, where $C$ is the Lipschitz constant of $\partial_\mu(\partial_x g)$ and $\partial_x(\partial_\mu g)$. It follows that $\partial_x(\partial_\mu g)(x, P_\xi, \xi)=\partial_\mu((\partial_x g)(x, P_\xi))( \xi)$, P-a.s., and, hence,
$$\partial_x(\partial_\mu g)(x, P_\xi, y)=\partial_\mu((\partial_x g)(x, P_\xi))(y),\ P_\xi(dy)\mbox{-a.e.}$$

Letting $\varepsilon>0$ and $\theta$ be a standard normally distributed random variable, which is independent of $\xi$ , and taking $\xi+\varepsilon \theta$\ instead of $\xi$, we have
$$\partial_x(\partial_\mu g)(x, P_{\xi+\varepsilon \theta}, y)=\partial_\mu(\partial_x g)(x, P_{\xi+\varepsilon \theta}, y),\ P_{\xi+\varepsilon \theta}(dy)\mbox{-a.e.},$$
\noindent and, thus, $dy$-a.s. on ${\mathbb R}$. Taking into account that $\partial_x(\partial_\mu g), \partial_\mu(\partial_x g)$ are Lipschitz, this yields
$$\partial_x(\partial_\mu g)(x, P_{\xi+\varepsilon \theta}, y)=\partial_\mu(\partial_x g)(x, P_{\xi+\varepsilon \theta}, y),\  \mbox{for all}\ y\in
{\mathbb R}.$$
\noindent Finally, using $W_2(P_{\xi+\varepsilon \theta}, P_{\xi})\leq \varepsilon(E[\theta^2])^{\frac{1}{2}}=C\varepsilon$, and again the Lipschitz property of $\partial_x(\partial_\mu g)$\ and $\partial_\mu(\partial_x g)$, we obtain
$$\partial_x(\partial_\mu g)(x, P_{\xi}, y)=\partial_\mu(\partial_x g)(x, P_{\xi})(y),\  \mbox{for all}\ x,\ y\in
{\mathbb R},\ \xi\in L^2({\cal F}).$$
\noindent The proof is complete.
\end{proof}

\begin{proposition} Under Hypothesis (H.2) the
first order derivatives $\partial_{x_i}X^{t,x,P_\xi}_s$ and
$\partial_{\mu}X^{t,x,P_\xi}_s(y)$ are in $L^2$-sense
differentiable with respect to $x$ and $y$, and interpreted as functional
of $\xi\in L^2({\cal F}_t)$ they are also Fr\'{e}chet differentiable
with respect to $\xi$. Moreover, for all $t\in [0,T],\ x,y,z\in {\mathbb R}$ and
$\xi\in L^2({\cal F}_t)$, there are stochastic processes
$\partial_\mu(\partial_{x_i} X^{t,x,P_\xi})(y)$, $1\le i\le d$, and
$\partial^2_\mu X^{t,x,P_\xi}(y,z)=\partial_\mu(\partial_\mu
X^{t,x,P_\xi})(y,z)$ in ${\cal S}^2_{\mathbb{F}}(t,T)$ such that, for
all $\eta \in L^2({\cal F}_t)$ the Fr\'{e}chet derivatives in $\xi$, $D_\xi[\partial_x X^{t,x,P_\xi}_s](.)$ and $D_\xi[\partial_\mu
X^{t,x,P_\xi}_s(y)](.)$ satisfy
\be\label{3.61}\begin{array}{lll}
&{\rm i)}\ D_\xi[\partial_{x_i} X^{t,x,P_\xi}_s](\eta)=\widehat{E}[\partial_\mu (\partial_{x_i}X^{t,x,P_\xi})(\widehat{\xi})\cdot\widehat{\eta}],\\
&{\rm ii)}\ D_\xi[\partial_\mu X^{t,x,P_\xi}_s(y)](\eta)=\widehat{E}[\partial^2_\mu X^{t,x,P_\xi}(y,\widehat{\xi})\cdot\widehat{\eta}],\ s\in[t,T],\ y\in {\mathbb R}^d.\end{array}\ee
\noindent Furthermore, the mixed second order derivatives $\partial_{x_i}(\partial_\mu X^{t,x,P_\xi}(y))$\ and $\partial_\mu(\partial_{x_i} X^{t,x,P_\xi})(y)$
coincide, i.e.,
$$\partial_{x_i}(\partial_\mu X^{t,x,P_\xi}_s(y))=\partial_\mu(\partial_{x_i} X^{t,x,P_\xi}_s)(y),\ s\in [t, T],\ \mbox{P-a.s.},$$
\noindent and for
$$M^{t,x,P_\xi}_s(y,z):=(\partial_{x_i x_j}^2
X^{t,x,P_\xi}_s,\partial_\mu (\partial_{x_i}X^{t,x,P_\xi})(y), \partial_\mu^2  X^{t,x,P_\xi}_s(y,z),
\partial_{y_i}(\partial_\mu X^{t,x,P_\xi}_s(y))),$$

\noindent $1\le i\le d,$
we have that, for all $p\ge 2$, there is some constant $C_p\in {\mathbb R}$, such that, for all $t\in [0,T],\ x,\ x',\ y,\ y',\ z,\ z'$ and $\xi,\ \xi'\in L^2({\cal F}_t)$,
\be\label{3.62}\begin{array}{lll}
&{\rm iii)}\ E[\sup_{s\in[t,T]}|M_s^{t,x,P_\xi}(y,z)|^p ]\le C_p,\\
&{\rm iv)}\  E[\sup_{s\in[t,T]}|M_s^{t,x,P_\xi}(y,z)-M_s^{t,x',P_{\xi'}}(y',z')|^p]\\
&\ \ \ \ \le C_p(|x-x'|^p+|y-y'|^p+|z-z'|^p+W_2(P_\xi,P_{\xi'})^p).\end{array}\ee\end{proposition}

\begin{proof} For simplicity of the redaction let us restrict ourselves again to the case of dimension $d=1$ and $b=0$. Recall that
\be\label{3.63}\partial_x X^{t,x,P_\xi}_s=1+\int_t^s(\partial_x\sigma)
(X^{t,x,P_\xi}_r,P_{X^{t,\xi}_r})\partial_x X^{t,x,P_\xi}_rdB_r,\
s\in[t,T].\ee
\noindent Thanks to our assumptions on $\sigma$ and our estimates for
$\partial_x X^{t,x,P_\xi}$ as well as the differentiability
properties of $X^{t,x,\xi}$ with respect to the measure, we see with the help of standard arguments that $L^2({\cal F}_t)\ni\xi\mapsto
\partial_x X^{t,x,\xi}:=\partial_x X^{t,x,P_\xi}\in
L^2({\cal F}_s)$ is G\^{a}teaux differentiable, and its
directional derivative in direction $\eta\in L^2({\cal F}_t)$,
\be\label{3.64}\partial_\xi[\partial_x X^{t,x,\xi}_s](\eta)=L^2\mbox{-}
\lim_{h\rightarrow 0}\frac{1}{h}\left(\partial_x X^{t,x,\xi+h\eta}_s
-\partial_x X^{t,x,\xi}_s\right) \ee

\noindent satisfies the equation
\be\label{3.65}\begin{array}{lll} \partial_\xi[\partial_x X^{t,x,\xi}](\eta)&=\displaystyle
\int_t^s(\partial_x\sigma)(X^{t,x,P_\xi}_r,P_{X^{t,\xi}_r})
\cdot \partial_\xi[\partial_x X^{t,x,P_\xi}_r](\eta)dB_r\\
&\ \ \displaystyle+\int_t^s(\partial_x^2\sigma)(X^{t,x,
P_\xi}_r,P_{X^{t,\xi}_r})\partial_x X^{t,x,\xi}_r\partial_\xi[X^{t,x,
\xi}_r](\eta)dB_r \hfill(=:I_1)\\
&\ \ \displaystyle+\int_t^s \partial_x X^{t,x,P_\xi}_r(D_\xi (\widetilde{\partial_x \sigma}))(X^{t,x,P_\xi}_r, \widetilde{X}^{t,\widetilde{\xi}}_r)(\partial_\xi[\widetilde{X}^{t,\widetilde{\xi}}_r](\widetilde{\eta}))dB_r, \hfill(=:I_2)\end{array}\ee

\noindent $s\in[t,T]$, where $\widetilde{\partial_x \sigma}(X^{t,x,P_\xi}_r, \widetilde{X}^{t,\widetilde{\xi}}_r)= ({\partial_x \sigma})(X^{t,x,P_\xi}_r, P_{{X}^{t,{\xi}}_r})$, $r\in [t, T]$, is the lifted process, and $D_\xi (\widetilde{\partial_x \sigma})$ the Fr\'{e}chet derivative of $\widetilde{\partial_x \sigma}$ with respect to its second variable.

Note that, with the results of Theorem 3.1 and the notations and the arguments we have used in its proof, we have $$D_\xi (\widetilde{\partial_x \sigma})(X^{t,x,P_\xi}_r,\widetilde{X}^{t,\widetilde{\xi}}_r)(\partial_\xi[\widetilde{X}^{t,\widetilde{\xi}}_r](\widetilde{\eta}))=\widetilde{E}[\partial_\mu(\partial_x\sigma)
(X^{t,x,P_\xi}_r, P_{{X}^{t,{\xi}}_r}, \widetilde{{X}}^{t,\widetilde{\xi}}_r)\partial_\xi[\widetilde{X}^{t,\widetilde{\xi}}_r(\widetilde{\eta})]], r\in [t, T],$$
and
\be\label{3.66}\begin{array}{lll}
&I_1=\displaystyle\widehat{E}\left[\int_t^s(\partial_x^2\sigma)(X^{t,x,P_\xi}_r,
P_{X^{t,\xi}_r})\partial_x X^{t,x,\xi}_r\partial_\mu X^{t,x,P_\xi}_r(\widehat{\xi})dB_r
\cdot\widehat{\eta}\right],\\
&I_2=\displaystyle\widehat{E}\bigg[\int_t^s\widetilde{E}\big[\partial_\mu(\partial_x
\sigma)(X^{t,x,P_\xi}_r,P_{X^{t,\xi}_r},\widetilde{X}^{t,\widehat{\xi},P_\xi}_r)
\cdot\partial_x \widetilde{X}^{t,\widehat{\xi},P_\xi}_r\\
&\ \ \ \ \ \ \ \   +\partial_\mu(\partial_x
\sigma)(X^{t,x,P_\xi}_r,P_{X^{t,\xi}_r},\widetilde{X}^{t,\widetilde{\xi}}_r)
\cdot\partial_\mu\widetilde{X}^{t,\widetilde{\xi},P_\xi}(\widehat{\xi})\big]\partial_x X^{t,x,P_\xi}_r
dB_r\cdot\widehat{\eta}\bigg].\end{array}\ee

\noindent Consequently,
\be\label{3.67} \partial_\xi[\partial_x X^{t,x,\xi}_s](\eta)=\widehat{E}
\left[\partial_\mu\partial_x X^{t,x,P_\xi}_s(\widehat{\xi})\cdot\widehat{\eta}
\right],\ s\in[t,T],\ee

\smallskip

\noindent where the process $\partial_\mu\partial_x X^{t,x,P_\xi}
(y)$ ($x,y\in {\mathbb R}$) is the unique solution in ${\cal S}^2_{\mathbb{F}}(t,T)$
of the equation
\be\label{3.68}\begin{array}{lll}\partial_\mu\partial_x X^{t,x,P_\xi}_s(y)=&\displaystyle
\int_t^s(\partial_x\sigma)(X^{t,x,P_\xi}_r,P_{X^{t,\xi}_r})
\cdot \partial_\mu\partial_x X^{t,x,P_\xi}_r(y)dB_r\\
&+\displaystyle\int_t^s(\partial_x^2\sigma)(X^{t,x,P_\xi}_r,
P_{X^{t,\xi}_r})\partial_x X^{t,x,P_\xi}_r\partial_\mu X^{t,x,P_\xi}_r
(y)dB_r\\
&+\displaystyle\int_t^s\widetilde{E}\big[\partial_\mu(\partial_x\sigma)
(X^{t,x,P_\xi}_r,P_{X^{t,\xi}_r},\widetilde{X}^{t,y,P_\xi}_r)\cdot
\partial_x\widetilde{X}^{t,y,P_\xi}_r \big]\partial_x X^{t,x,P_\xi}_r dB_r\\
&+\displaystyle\int_t^s\widetilde{E}\big[\partial_\mu(\partial_x\sigma)(X^{t,x,P_\xi}_r,P_{X^{t,\xi}_r},
\widetilde{X}^{t,\widetilde{\xi}}_r)\partial_\mu \widetilde{X}^{t,
\widetilde{\xi},P_\xi}_r(y)\big]\partial_x X^{t,x,P_\xi}_r dB_r,\end{array}\ee

\noindent $s\in[t,T]$. Moreover, on the basis of the estimates we have already
gotten before, we see that, for all $p\ge 2$,  there is some
constant $C_p\in {\mathbb R}$, such that, for all $t\in {\mathbb R},\ x,\ x',\ y,\ y'\in {\mathbb R},$
and $\xi,\ \xi'\in L^2({\cal F}_t)$,
\be\label{3.69}\begin{array}{lll}
&{\rm i)}\ E\left[\sup_{s\in[t,T]}|\partial_\mu\partial_x
X^{t,x,P_\xi}_s(y)|^p\right]\le C_p,\\
&{\rm ii)}\ E\left[\sup_{s\in[t,T]}|\partial_\mu\partial_x
X^{t,x,P_\xi}_s(y)-\partial_\mu\partial_x
X^{t,x',P_{\xi'}}_s(y')|^p\right]\le C_p(|x-x'|^p+|y-y'|^p+W_2(P_\xi,
P_{\xi'})^p).\end{array}\ee

This second estimate combined with (\ref{3.67}) allows to show in analogy to the second part of the proof of Theorem 3.1 that
$L^2({\cal F}_t)\ni\xi\mapsto \partial_x X^{t,x,\xi}_s\in L^2({\cal F}_s)$\ is Fr\'{e}chet differentiable, and the Fr\'{e}chet derivative satisfies
$$D_\xi[\partial_x X^{t,x,\xi}_s](\eta)=\widehat{E}[\partial_\mu\partial_x X^{t,x,P_\xi}_s(\widehat{\xi})\widehat{\eta}],\ s\in [t, T],\ \eta\in L^2({\cal F}_t).$$

This shows that, due to our definition, $\partial_\mu\partial_x X^{t,x,P_\xi}_s(y)$\ is the derivative of $\partial_x X^{t,x,P_\xi}_s$ with respect to the probability measure at $P_\xi$. Let us now consider equation (\ref{3.56}) for $\partial_\mu X^{t,x,P_\xi}_s(y)=U^{t,x,P_\xi}_s(y)$. It is standard to show that $x\mapsto U^{t,x,P_\xi}_s(y)\in L^2({\cal F}_s)$ is differentiable, and the derivative satisfies the equation:

$$\label{100}\begin{array}{lll} &\partial_x U^{t,x,P_\xi}_s(y)\\
&=\displaystyle\int_t^s\left((\partial^2_x
\sigma)(X^{t,x,P_\xi}_r,P_{X^{t,\xi}_r})\partial_x X^{t,x,P_\xi}_r\cdot U^{t,x,P_\xi}_r(y)+
 (\partial_x\sigma)(X^{t,x,P_\xi}_r,P_{X^{t,\xi}_r})\partial_x U^{t,x,
P_\xi}_r(y)\right)dB_r\\
&\ \ \ \displaystyle+\int_t^s\widetilde{E}\left[\partial_x(\partial_\mu
\sigma)(X^{t,x,P_\xi}_r,P_{X^{t,\xi}_r},
\widetilde{X}^{t,y,P_\xi}_r)\partial_x X^{t,x,P_\xi}_r \cdot \partial_x\widetilde{X}^{t,
y,P_\xi}_r\right.\\ 
&\ \ \ \ \ \ +\left.\partial_x(\partial_\mu\sigma)(X^{t,x,P_\xi}_r,P_{X^{t,
\xi}_r},\widetilde{X}^{t,\widetilde{\xi}}_r)\partial_x X^{t,x,P_\xi}_r \widetilde{U}^{t,
\widetilde{\xi}}_r(y) \right]dB_r,\ \ s\in [t, T].\end{array}$$

However, since $\partial_x(\partial_\mu \sigma)=\partial_\mu(\partial_x \sigma)$ (see Lemma 3.3), the above equation coincides with (\ref{3.68}), and it follows from the uniqueness of the solution that 
$$\partial_x\partial_\mu X^{t,x,P_\xi}_s(y)= \partial_x U^{t,x,P_\xi}_s(y)=\partial_\mu\partial_x X^{t,x,P_\xi}_s(y),$$
\noindent $s\in [t, T],\ x,\ y\in {\mathbb R},\ \xi\in L^2({\cal F}_t).$

  After having studied the second order mixed derivatives $\partial_\mu\partial_x X^{t,x,P_\xi}_s(y)$ and $\partial_x\partial_\mu X^{t,x,P_\xi}_s(y)$, and identified them, let us come now to
the second order derivative with respect to the measure, $\partial^2_\mu
X^{t,x,P_\xi}_s(y,z)$, $y,z\in {\mathbb R}.$

Recall the SDE solved by $\partial_\mu X^{t,x,P_\xi}(y)
=U^{t,x,P_\xi}(y)$ and that solved by $U^{t,\xi}(y)
(=U^{t,x,P_\xi}(y)_{\big|x=\xi}$). Taking under account the
assumptions made on $\sigma$ and its derivatives, we see
that the mappings $L^2({\cal F}_t)\ni\xi\mapsto
U^{t,x,\xi}_s(y):=U^{t,x,P_\xi}_s(y)\in L^2({\cal F}_s)$, and $\xi \mapsto U^{t,\xi}_s(y)
\in L^2({\cal F}_s)$ are G\^{a}teaux differentiable, and the
directional derivative $\partial_\xi[U^{t,x,\xi}_s(y)](\eta)$ in direction
$\eta\in L^2({\cal F}_t)$ satisfies the SDE 
\be\label{3.70}\begin{array}{lll}
&\partial_\xi[U^{t,x,P_\xi}_s(y)](\eta) =\displaystyle\int_t^s(\partial_x
\sigma)(X^{t,x,P_\xi}_r,P_{X^{t,\xi}_r})\partial_\xi[U^{t,x,P_\xi}_r(y)]
(\eta)dB_r\\
& \ \ \displaystyle+\int_t^s(\partial_x^2\sigma)(X^{t,x,P_\xi}_r,
P_{X^{t,\xi}_r})U^{t,x,P_\xi}_r(y)\widehat{E}[U_r^{t,x,P_\xi}
(\widehat{\xi})\cdot\widehat{\eta}]dB_r\\
& \ \ \displaystyle+\int_t^s\widetilde{E}[\partial_\mu
(\partial_x\sigma)(X^{t,x,P_\xi}_r,P_{X^{t,\xi}_r},\widetilde{
X}^{t,\widetilde{\xi}}_r)U^{t,x,P_\xi}_r(y)(\partial_x
\widetilde{X}^{t,\widetilde{\xi},P_\xi}_r\cdot\widetilde{\eta}+
\widehat{E}[\widetilde{U}^{t,\widetilde{\xi}}_r(\widehat{\xi})
\cdot\widehat{\eta}])]dB_r\\
& \ \ \displaystyle+\int_t^s\widetilde{E}[(\partial_\mu\sigma)
(X^{t,x,P_\xi}_r,P_{X^{t,\xi}_r},\widetilde{X}^{t,y,P_\xi}_r)
\cdot\widehat{E}[\partial_\mu\partial_x\widetilde{X}^{t,y,P_\xi}_r
(\widehat{\xi})\cdot\widehat{\eta}]]dB_r\\
& \ \ \displaystyle+\int_t^s\widetilde{E}[\partial_x
(\partial_\mu\sigma)(X^{t,x,P_\xi}_r,P_{X^{t,\xi}_r},
\widetilde{X}^{t,y,P_\xi}_r)\cdot\partial_x\widetilde{X}^{t,
y,P_\xi}_r]\cdot\widehat{E}[U^{t,x,P_\xi}_r(\widehat{\xi})
\cdot\widehat{\eta}] dB_r\\
& \ \ \displaystyle+\int_t^s\widetilde{E}[\overline{E}[\partial_\mu(
\partial_\mu\sigma)(X^{t,x,P_\xi}_r,P_{X^{t,\xi}_r},
\widetilde{X}^{t,y,P_\xi}_r,\overline{X}^{t,\overline{\xi}}_r)\cdot
\partial_x\widetilde{X}^{t,y,P_\xi}_r(\partial_x\overline{X}^{t,
\overline{\xi},P_\xi}_r\cdot\overline{\eta}+\widehat{E}
[\overline{U}^{t,\overline{\xi}}_r(\widehat{\xi})\cdot\widehat{\eta}
])]]dB_r\\
& \ \ \displaystyle+\int_t^s\widetilde{E}[\partial_y
(\partial_\mu\sigma)(X^{t,x,P_\xi}_r,P_{X^{t,\xi}_r},
\widetilde{X}^{t,y,P_\xi}_r)\cdot\partial_x\widetilde{X}^{t,
y,P_\xi}_r\cdot\widehat{E}[\widetilde{U}^{t,y,P_\xi}_r(\widehat{\xi})
\cdot\widehat{\eta}]]dB_r\\
& \ \ \displaystyle+\int_t^s\widetilde{E}[(\partial_\mu\sigma)
(X^{t,x,P_\xi}_r,P_{X^{t,\xi}_r},\widetilde{X}^{t,\widetilde{\xi}}_r)
\cdot \partial_\xi[\widetilde{U}^{t,\widetilde{\xi}}_r(y)](\widetilde{\eta})]
dB_r\\
& \ \ \displaystyle+\int_t^s\widetilde{E}[\partial_x
(\partial_\mu\sigma)(X^{t,x,P_\xi}_r,P_{X^{t,\xi}_r},
\widetilde{X}^{t,\widetilde{\xi}}_r)\cdot\widetilde{U}^{t,
\widetilde{\xi}}_r(y)]\cdot\widehat{E}[U^{t,x,P_\xi}_r(\widehat{\xi})
\cdot\widehat{\eta}]dB_r\\
& \ \ \displaystyle+\int_t^s\widetilde{E}[\overline{E}[\partial_\mu(
\partial_\mu\sigma)(X^{t,x,P_\xi}_r,P_{X^{t,\xi}_r},
\widetilde{X}^{t,\widetilde{\xi}}_r,\overline{X}^{t,
\overline{\xi}}_r)\cdot\widetilde{U}^{t,\widetilde{\xi}}_r(y)
(\partial_x\overline{X}^{t,\overline{\xi},P_\xi}_r\cdot\overline{\eta}+ \widehat{E}[\overline{U}^{t,\overline{\xi}}_r(\widehat{\xi})
\cdot\widehat{\eta}])]]dB_r\\
& \ \ \displaystyle+\int_t^s\widetilde{E}[\partial_y
(\partial_\mu\sigma)(X^{t,x,P_\xi}_r,P_{X^{t,\xi}_r},
\widetilde{X}^{t,\widetilde{\xi}}_r)\cdot\widetilde{U}^{t,
\widetilde{\xi}}_r(y)\cdot(\partial_x\widetilde{X}^{t,
\widetilde{\xi},P_\xi}_r\cdot\widetilde{\eta}+\widehat{E}[\widetilde{U}^{t,
\widetilde{\xi}}_r(\widehat{\xi})\cdot\widehat{\eta}])]dB_r,\end{array}\ee

\noindent $s\in[t,T].$ Here we have used the notation
$(\overline{X}^{t,\overline{\xi}},\overline{U}^{t,\overline{\xi}}(y))$;
it is used in the same sense as the corresponding processes endowed with
$\widetilde{\quad}$ or $\widehat{\quad}$: We consider a copy $(
\overline{\xi},\overline{\eta},\overline{B})$ independent of $(\xi,
\eta,B),(\widetilde{\xi},\widetilde{\eta},\widetilde{B})$ and
$(\widehat{\xi},\widehat{\eta},\widehat{B})$, and the process
$\overline{X}^{t,\overline{\xi}}$ is the solution of the SDE for
$X^{t,\xi}$ and $\overline{U}^{t,\overline{\xi}}$ that of the SDE
for $U^{t,\xi}$, but both with the data $(\overline{\xi},\overline{B})$
instead of $(\xi,B)$.

Let us comment also the expression $\partial^2_\mu\sigma(x,P_\vartheta,y,z)
=\partial_\mu
(\partial_\mu\sigma)(x,P_\vartheta,y,z)$ in the above formula.
Recalling that $\partial^2_\mu\sigma(x,P_\vartheta,y,z)=\partial_\mu
(\partial_\mu\sigma)(x,P_\vartheta,y,z)$ is defined through the relation
$\partial_\vartheta[\widetilde{\partial_\mu\sigma}(x,\vartheta,y)](\theta)=
E[\partial^2_\mu\sigma(x,P_\vartheta,y,\vartheta)\cdot\theta],$ for
$\vartheta,\ \theta\in L^2({\cal F}),\ x,\ y\in {\mathbb R}$, we have namely for the
G\^{a}teaux derivative of $L^2({\cal F})\ni\vartheta\mapsto
\widetilde{\partial_\mu\sigma}(x,\vartheta,y):=(\partial_\mu\sigma)(x,
P_\vartheta,y)$ in direction $\theta\in L^2({\cal F})$,
\be\label{3.71}\partial_\vartheta[\widetilde{\partial_\mu\sigma}(x,\vartheta,y)](\theta)=
\overline{E}[(\partial_\mu^2\sigma)(x,P_\vartheta,y,
\overline{\vartheta})\cdot\overline{\theta}].\ee

\noindent Then, of course,
\be\label{3.72}\begin{array}{lll}&\partial_\vartheta[\widetilde{\partial_\mu\sigma}(X^{t,x,P_\xi}_r,X^{t,\xi}_r,
\widetilde{X}^{t,y,P_\xi})](D_\xi[X^{t,\xi}_r](\eta))\\
=&\overline{E}[(\partial_\mu^2\sigma)(X^{t,x,P_\xi}_r, P_{X^{t,\xi}_r},
\widetilde{X}^{t,y,P_\xi},\overline{X}^{t,\overline{\xi}}_r)\cdot(\partial_x X^{t,\xi,P_\xi}_r\cdot \eta+
\widehat{E}[\overline{U}^{t,\overline{\xi}}_r(\widehat{\xi})
\cdot\widehat{\eta}])],\end{array}\ee

\smallskip

\noindent $r\in[t,T],$ but this is just, what has been used for the
above formula in combination with arguments already developed in the
proof of Theorem 3.1. The SDE solved by $\partial_\xi [U^{t,\xi}(y)]$ is obtained by substituting $x=\xi$ in the equation for $\partial_\xi[U^{t,x,P_\xi}(y)]$
(Recall that $X^{t,\xi}=X^{t,x,P_\xi}|_{x=\xi}$\ and $U^{t,\xi}=U^{t,x,P_\xi}|_{x=\xi}$).

\smallskip

Let us now consider the process $U^{t,x,P_\xi}(y,z)=(U^{t,x,P_\xi}_s(y,z))_{s\in[t,T]}
\in {\cal S}^2_{\mathbb{F}}(t,T)$ defined as the unique solution of
the following SDE:
\be\label{3.74}\begin{array}{lll}
&U^{t,x,P_\xi}_s(y,z)=\displaystyle\int_t^s(\partial_x
\sigma)(X^{t,x,P_\xi}_r,P_{X^{t,\xi}_r})U^{t,x,P_\xi}_r(y,z)dB_r\\
&\ \ +\int_t^s(\partial_x^2\sigma)(X^{t,x,P_\xi}_r,
P_{X^{t,\xi}_r})U^{t,x,P_\xi}_r(y)\cdot U_r^{t,x,P_\xi}(z)dB_r\\
&\ \ \displaystyle+\int_t^s\widetilde{E}[\partial_\mu
(\partial_x\sigma)(X^{t,x,P_\xi}_r,P_{X^{t,\xi}_r},\widetilde{
X}^{t,z,P_\xi}_r)U^{t,x,P_\xi}_r(y)\cdot\partial_x
\widetilde{X}^{t,z,P_\xi}_r]dB_r\\
&\ \ \displaystyle+\int_t^s\widetilde{E}[\partial_\mu
(\partial_x\sigma)(X^{t,x,P_\xi}_r,P_{X^{t,\xi}_r},\widetilde{
X}^{t,\widetilde{\xi}}_r)U^{t,x,P_\xi}_r(y)
\widetilde{U}^{t,\widetilde{\xi}}_r(z)]dB_r\\
&\ \ \displaystyle+\int_t^s\widetilde{E}[(\partial_\mu\sigma)
(X^{t,x,P_\xi}_r,P_{X^{t,\xi}_r},\widetilde{X}^{t,y,P_\xi}_r)
\cdot\partial_\mu\partial_x\widetilde{X}^{t,y,P_\xi}_r
(z)]dB_r\\
&\ \ \displaystyle+\int_t^s\widetilde{E}[\partial_x
(\partial_\mu\sigma)(X^{t,x,P_\xi}_r,P_{X^{t,\xi}_r},
\widetilde{X}^{t,y,P_\xi}_r)\cdot\partial_x\widetilde{X}^{t,
y,P_\xi}_r]\cdot U^{t,x,P_\xi}_r(z)dB_r\\
&\ \ \displaystyle+\int_t^s\widetilde{E}[\overline{E}[\partial_\mu(
\partial_\mu\sigma)(X^{t,x,P_\xi}_r,P_{X^{t,\xi}_r},
\widetilde{X}^{t,y,P_\xi}_r,\overline{X}^{t,z,P_{\xi}}_r)\cdot
\partial_x\widetilde{X}^{t,y,P_\xi}_r\cdot\partial_x\overline{X}^{t,
z,P_\xi}_r]]dB_r\\
&\ \ \displaystyle+\int_t^s\widetilde{E}[\overline{E}[\partial_\mu(
\partial_\mu\sigma)(X^{t,x,P_\xi}_r,P_{X^{t,\xi}_r},
\widetilde{X}^{t,y,P_\xi}_r,\overline{X}^{t,\overline{\xi}}_r)\cdot
\partial_x\widetilde{X}^{t,y,P_\xi}_r\overline{U}^{t,
\overline{\xi}}_r(z)]]dB_r\\
&\ \ \displaystyle+\int_t^s\widetilde{E}[\partial_y
(\partial_\mu\sigma)(X^{t,x,P_\xi}_r,P_{X^{t,\xi}_r},
\widetilde{X}^{t,y,P_\xi}_r)\cdot\partial_x\widetilde{X}^{t,
y,P_\xi}_r\widetilde{U}^{t,y,P_\xi}_r(z)]dB_r\\
&\ \ \displaystyle+\int_t^s\widetilde{E}[(\partial_\mu\sigma)
(X^{t,x,P_\xi}_r,P_{X^{t,\xi}_r},\widetilde{X}^{t,\widetilde{\xi}}_r)
\cdot \widetilde{U}^{t,\widetilde{\xi}}_r(y,z)]dB_r\\
&\ \ \displaystyle+\int_t^s\widetilde{E}[\partial_x
(\partial_\mu\sigma)(X^{t,x,P_\xi}_r,P_{X^{t,\xi}_r},
\widetilde{X}^{t,\widetilde{\xi}}_r)\cdot\widetilde{U}^{t,
\widetilde{\xi}}_r(y)]\cdot U^{t,x,P_\xi}_r(z)dB_r\\
&\ \ \displaystyle+\int_t^s\widetilde{E}[\overline{E}[\partial_\mu(
\partial_\mu\sigma)(X^{t,x,P_\xi}_r,P_{X^{t,\xi}_r},
\widetilde{X}^{t,\widetilde{\xi}}_r,\overline{X}^{t,z,
P_{\xi}}_r)\cdot\widetilde{U}^{t,\widetilde{\xi}}_r(y)\cdot
\partial_x\overline{X}^{t,z,P_\xi}_r]]dB_r\\
&\ \ \displaystyle+\int_t^s\widetilde{E}[\overline{E}[\partial_\mu(
\partial_\mu\sigma)(X^{t,x,P_\xi}_r,P_{X^{t,\xi}_r},
\widetilde{X}^{t,\widetilde{\xi}}_r,\overline{X}^{t,
\overline{\xi}}_r)\cdot\widetilde{U}^{t,\widetilde{\xi}}_r(y)\cdot
\overline{U}^{t,\overline{\xi}}_r(z)_r]]dB_r\\
&\ \ \displaystyle+\int_t^s\widetilde{E}[\partial_y
(\partial_\mu\sigma)(X^{t,x,P_\xi}_r,P_{X^{t,\xi}_r},
\widetilde{X}^{t,z,P_{\xi}}_r)\cdot\widetilde{U}^{t,
z,P_{\xi}}_r(y)\cdot\partial_x\widetilde{X}^{t,
z,P_{\xi}}_r]dB_r\\
&\ \ \displaystyle+\int_t^s\widetilde{E}[\partial_y
(\partial_\mu\sigma)(X^{t,x,P_\xi}_r,P_{X^{t,\xi}_r},
\widetilde{X}^{t,\widetilde{\xi}}_r)\cdot\widetilde{U}^{t,
\widetilde{\xi}}_r(y)\cdot\widetilde{U}^{t,
\widetilde{\xi}}_r(z)]dB_r,\ s\in[0,T],\end{array}\ee

\noindent combined with the SDE for $U^{t,\xi}(y,z)=(U^{t,\xi}_s
(y,z))_{s\in[t,T]}$, obtained by substituting $x=\xi$ in the
equation for $U^{t,x,P_\xi}(y,z)$ (recall namely that $X^{t,\xi}=
X^{t,x,P_\xi}\big|_{x=\xi},\, U^{t,\xi}=U^{t,x,P_\xi}\big|_{x=\xi}$).
We consider now the processes $\widehat{E}[U^{t,x,P_\xi}(y,
\widehat{\xi})\cdot\widehat{\eta}]$ and $\widehat{E}[U^{t,\xi}(y,
\widehat{\xi})\cdot\widehat{\eta}]$. Substituting first $z=
\widehat{\xi}$ in the SDE for $U^{t,x,P_\xi}(y,z)$ and that
for $U^{t,\xi}(y,z)$, then multiplying the both sides of these SDEs
with $\widehat{\eta}$ and taking the expectation $\widehat{E}[.]$,
we get just the SDEs solved by $\partial_\xi[U^{t,x,P_\xi}(y)](\eta)$
and $\partial_\xi[U^{t,\xi}_s(y)](\eta)$ (See also the proof of the preceding
Theorem 3.1 for the corresponding proof for the first order derivatives),
and from the uniqueness of the solution of these SDEs we conclude that
\be\label{3.75}\begin{array}{lll}
&\partial_\xi[U^{t,x,P_\xi}_s(y)](\eta)=\widehat{E}[U^{t,x,P_\xi}_s(y,
\widehat{\xi})\cdot\widehat{\eta}],\ \ \ \ \  \mbox{and}\\
&\partial_\xi[U^{t,\xi}_s(y)](\eta)=\widehat{E}[U^{t,\xi}_s(y,
\widehat{\xi})\cdot\widehat{\eta}],\ s\in[t,T],\ y\in {\mathbb R}.\end{array}\ee
We also observe that the SDEs for $U^{t,x,P_\xi}(y, z)$ and
$U^{t,\xi}(y,z)$ allow to make estimates. In particular, we see that,
for all $p\ge 2,$ there is some constant $C_p\in {\mathbb R}$ such that, for all
$t\in [0,T], x,\ x',\ y,\ y',\ z,\ z'\in {\mathbb R}$ and $\xi,\ \xi'\in L^2({\cal F}_t)$,
\be\label{3.77}\begin{array}{lll}
&{\rm iii)}\ E[\sup_{s\in[t,T]}|U_s^{t,x,P_\xi}(y,z)|^p]\le C_p,\\
&{\rm iv)}\ E[\sup_{s\in[t,T]}|U_s^{t,x,P_\xi}(y,z)-U_s^{t,x',P_{\xi'}}(y',z')|^p]\\ &\ \ \ \ \ \ \le C_p(|x-x'|^p+|y-y'|^p+|z-z'|^p+W_2(P_\xi,P_{\xi'})^p).\end{array}\ee

\smallskip

Estimate (\ref{3.77})-iv) allows to show in analogy to the argument of the second part of the proof of Theorem 3.1 that the mappings 
$L^2({\cal F}_t)\ni\xi\mapsto U^{t,x,P_\xi}_s(y)\in L^2({\cal F}_s)$ and $L^2({\cal F}_t)\ni\xi\mapsto U^{t,\xi}_s(y)\in L^2({\cal F}_s)$\ 
are even Frechet and 
$$D_\xi[\partial_\mu X_s^{t,x,P_\xi}(y)](\eta)=D_\xi[U_s^{t,x,P_\xi}(y)](\eta)=\widehat{E}[U_s^{t,x,P_\xi}(y,\widehat{\xi})\widehat{\eta}],$$
$$D_\xi[\partial_\mu X_s^{t,\xi}(y)](\eta)=D_\xi[U_s^{t,\xi}(y)](\eta)=\partial_x U_s^{t,x,\xi}(y)|_{x=\xi}\cdot \eta+\widehat{E}[U_s^{t,\xi}(y,\widehat{\xi})\widehat{\eta}].$$

\noindent But this means that
\be\label{3.76}\partial_\mu^2 X_s^{t,x,P_\xi}(y,z)=U^{t,x,
P_\xi}_s(y,z),\ s\in[0,T],\ee

\noindent for all $t\in[0,T],\ x,\ y,\ z\in {\mathbb R},\ \xi\in L^2({\cal F}_t).$

The arguments developed above allow also to obtain the SDEs satisfied
by the other second order derivatives ($\partial_x^2 X^{t,x,P_\xi}$\ and $\partial_y(\partial_\mu X^{t,x,P_\xi}(y))$) of the process $X^{t,x,P_\xi}$
and to prove the stated estimates. The proof is complete.\end{proof}

\section{Regularity of the value function}

Given a function $\Phi\in C^{2,1}_b({\mathbb R}^d\times{\cal P}_2({\mathbb R}^d))$, the
objective of this section is to study the regularity of the function
$V:[0,T]\times {\mathbb R}^d\times {\cal P}_2({\mathbb R}^d)\rightarrow {\mathbb R},$
\be\label{4.1}V(t,x,P_\xi)=E[\Phi(X^{t,x,P_\xi}_T,P_{X_T^{t,\xi}})],\ (t,x,\xi)\in [0,T]\times {\mathbb R}^d\times L^2({\cal F}_t;{\mathbb R}^d).\ee
\begin{lemma} Suppose that $\Phi\in C^{1,1}_b({\mathbb R}^d
\times{\cal P}_2({\mathbb R}^d))$. Then, under our Hypothesis (H.1), 
$V(t,.,.)$ $\in C^{1,1}({\mathbb R}^d\times{\cal P}_2({\mathbb R}^d))$, for all
$t\in[0,T]$, and the derivatives $\partial_x V(t,x,P_\xi)=
(\partial_{x_{i}} V(t,x,P_\xi,y))_{1\le i\le d}$ and
$\partial_\mu V(t,x,P_\xi,y)$ $= ((\partial_\mu V)_{i}(t,x,
P_\xi,y))_{1\le i\le d}$ are of the form
\be\label{4.2}\partial_{x_{i}}V(t,x,P_\xi)=\sum_{j=1}^d
E[(\partial_{x_j}\Phi)(X^{t,x,P_\xi}_T,P_{X^{t,\xi}_T})
\cdot \partial_{x_{i}}X^{t,x,P_\xi}_{T,j}],\ee
\be\label{4.3}\begin{array}{lll}
&(\partial_\mu V)_{i}(t,x,P_\xi,y)=\sum_{j=1}^d
E\[(\partial_{x_j}\Phi)(X^{t,x,P_\xi}_T,P_{X^{t,\xi}_T})
(\partial_\mu X^{t,x,P_\xi}_{T,j})_{i}(y)\\
&\ \ +
\widetilde{E}[(\partial_\mu\Phi)_j(X^{t,x,P_\xi}_T,
P_{X^{t,\xi}_T},\widetilde{X}_T^{t,y,P_\xi})\cdot\partial_{x_{i}}
\widetilde{X}^{t,y,P_\xi}_{T,j}+(\partial_\mu\Phi)_j(X^{t,x,
P_\xi}_T,P_{X^{t,\xi}_T},\widetilde{X}_T^{t,\widetilde{\xi}})
\cdot(\partial_\mu \widetilde{X}^{t,\widetilde{\xi},
P_\xi}_{T,j})_{i}(y)]\].\end{array}\ee

\noindent Moreover, there is some constant $C\in {\mathbb R}$ such that, for all
$t,\ t'\in[0,T],\ x,\ x',\ y,\ y'\in {\mathbb R},$ and $\xi,\ \xi'\in L^2({\cal F}_t),$
\be\label{4.4}\begin{array}{lll}
&{\rm i)}\ |\partial_{x_{i}}V(t,x,P_\xi)|+
|(\partial_\mu V)_{i}(t,x,P_\xi,y)|\le C,\\
&{\rm ii)}\ |\partial_{x_{i}}V(t,x,P_\xi)-\partial_{x_{i}}V(t,x',P_{\xi'})|+
|(\partial_\mu V)_{i}(t,x,P_\xi,y)-(\partial_\mu V)_{i}(t,x',P_{\xi'},y')|\\
&\ \ \ \ \ \le C(|x-x'|+|y-y'|+W_2(P_\xi,P_{\xi'})),\\
&{\rm iii)}\  |V(t,x,P_\xi)-V(t',x,P_\xi)|+|\partial_{x_{i}}V(t,x,P_\xi)
-\partial_{x_{i}}V(t',x,P_{\xi})|\\
&\ \ \ \ \ +|(\partial_\mu V)_{i}(t,x,P_\xi,y)-(\partial_\mu V)_{i}(t',x,P_{\xi},y)|
\le C|t-t'|^{1/2}.\end{array}\ee
\end{lemma}

\begin{proof} In order to simplify the presentation, we consider again the case of dimension $d=1$, but without restricting
the generality of the argument we use.

In accordance with the notations introduced in Section 2, we put $\widetilde{V}(t,x,\xi):=V(t,x,P_\xi)$,
and $\widetilde{\Phi}(z,\vartheta):=
\Phi(z,P_\vartheta),\, (z,\vartheta)\in {\mathbb R}\times L^2({\cal F}).$
Recall also that, in the same sense, $X^{t,x,\xi}=X^{t,x,P_\xi}$.
Then $\widetilde{\Phi}(X^{t,x,\xi}_T,X^{t,\xi}_T)=
\Phi(X^{t,x,P_\xi}_T,P_{X^{t,\xi}_T}).$

As $\Phi\in C_b^{1,1}({\mathbb R}\times{\cal P}_2({\mathbb R}))$, its first order
derivatives are bounded and Lipschitz continuous. Thus, standard
arguments combined with the results from the preceding section
show the existence of the Fr\'{e}chet derivative
$D_\xi\big(\widetilde{\Phi}(X^{t,x,\xi}_T,X^{t,\xi}_T)\big)$
of $L^2({\cal F}_t)\ni\xi\rightarrow
\widetilde{\Phi}(X^{t,x,\xi}_T,X^{t,\xi}_T)\in L^2({\cal F}_T),$
and for all $\eta\in L^2({\cal F}_t)$ we have
\be\label{4.5}\begin{array}{lll}
& D_\xi\left(\widetilde{\Phi}(X^{t,x,\xi}_T,X^{t,
\xi}_T)\right)(\eta)\\
\ =&\partial_x\widetilde{\Phi}(X^{t,x,\xi}_T,
X^{t,\xi}_T)D_\xi X^{t,x,\xi}_T(\eta)+(D\widetilde{\Phi})(X^{t,x,\xi}_T,
X^{t,\xi}_T)(D_\xi[X^{t,\xi}_T](\eta))\\
\ =&\partial_x\Phi(X^{t,x,P_\xi}_T,
P_{X^{t,\xi}_T})\widehat{E}[\partial_\mu X^{t,x,P_\xi}_T
(\widehat{\xi})\cdot\widehat{\eta}]
+\widetilde{E}[\partial_\mu\Phi(X^{t,x,P_\xi}_T,
P_{X^{t,\xi}_T},\widetilde{X}^{t,\widetilde{\xi}}_T)
D_{\widetilde{\xi}}[\widetilde{X}^{t,\widetilde{\xi}}_T(
\widetilde{\eta})]]\\
\ =&\partial_x\Phi(X^{t,x,P_\xi}_T,
P_{X^{t,\xi}_T})\widehat{E}[\partial_\mu X^{t,x,P_\xi}_T
(\widehat{\xi})\cdot\widehat{\eta}]\\
&\ +\widetilde{E}[\partial_\mu\Phi(X^{t,x,P_\xi}_T,
P_{X^{t,\xi}_T},\widetilde{X}^{t,\widetilde{\xi}}_T)
\cdot(\partial_x\widetilde{X}^{t,\widetilde{\xi},P_\xi}_T
\cdot\widetilde{\eta}+\widehat{E}[\partial_\mu
\widetilde{X}^{t,\widetilde{\xi}}_T(\widehat{\xi})\cdot
\widehat{\eta}])].\end{array}\ee

\smallskip

\noindent (For the notations used here the reader is referred
to the previous section). With the argument developed in the
preceding section we conclude that the derivative of
$\Phi(X^{t,x,\xi}_T,P_{X^{t,\xi}_T})$ with respect to the
measure in $P_\xi$ is given by
\be\label{4.6}\begin{array}{lll}
&\partial_\mu \big(\Phi(X^{t,x,P_\xi}_T,P_{X^{t,\xi}_T})
\big)(y)=\partial_x\Phi(X^{t,x,P_\xi}_T,P_{X^{t,\xi}_T})
\partial_\mu X^{t,x,P_\xi}_T(y)\\
&\ \ \  +\widetilde{E}\bigg[\partial_\mu
\Phi(X^{t,x,P_\xi}_T,P_{X^{t,\xi}_T},
\widetilde{X}_T^{t,y,P_\xi})\cdot\partial_x
\widetilde{X}^{t,y,P_\xi}_T+\partial_\mu\Phi(X^{t,x,
P_\xi}_T,P_{X^{t,\xi}_T},\widetilde X_T^{t,
\widetilde{\xi}})\cdot\partial_\mu
\widetilde{X}^{t,\widetilde{\xi},P_\xi}_T(y)
\bigg],\ y\in R.\end{array}\ee

\smallskip

\noindent In particular, we can deduce from this latter formula and the
estimates from the preceding section that, for all $p\ge 2$, there is
a constant $C_p\in {\mathbb R}$ such that, for all $x,\ x',\ y,\ y'\in {\mathbb R}$,
$\xi,\ \xi'\in L^2({\cal F}_t),$
\be\label{4.7}\begin{array}{lll}
& E\left[|\partial_\mu\big(\Phi(X^{t,x,P_\xi}_T,P_{X^{t ,
\xi}_T})\big)(y)|^p\right]\le C_p,\\
& E\left[|\partial_\mu\big(\Phi(X^{t,x,P_\xi}_T,P_{X^{t ,\xi}_T})\big)(y)-\partial_\mu\big(\Phi(X^{t,x',P_{\xi'}}_T,P_{X^{t ,
\xi'}_T})\big)(y')|^p\right]\\
&\ \ \ \ \ \le C_p(|x-x'|^p+|y-y'|^p+W_2(P_\xi,P_{\xi'})^p).\end{array}\ee

As the expectation $E[.]: L^2({\cal F})\rightarrow {\mathbb R}$\ is a bounded linear operator, it follows from (\ref{4.5}) and (\ref{4.6}) that 
$L^2({\cal F}_t\ni \xi\mapsto \widetilde{V}(t,x,\xi):=V(t,x,P_\xi)=E[\Phi(X^{t,x,P_\xi}_T,P_{X^{t,\xi}_T})]$ is Fr\'{e}chet differentiable and, for all $\eta\in L^2({\cal F}_t)$,
\be\label{4.8}\begin{array}{lll}
&D_\xi[\widetilde{V}(t,x,\xi)](\eta)=E\big[D_\xi\big(
\widetilde{\Phi}(X^{t,x,\xi}_T,X^{t,\xi}_T)\big)(\eta)]=E[\widetilde{E}[\partial_\mu
(\Phi(X^{t,x,P_\xi}_T,P_{X^{t,\xi}_T}))
(\widetilde{\xi})\cdot\widetilde{\eta}]]\\
&=\widetilde{E}[E[\partial_\mu (\Phi(X^{t,x,P_\xi}_T, P_{X^{t,\xi}_T}))(\widetilde{\xi})]\cdot\widetilde{\eta}],\end{array}\ee
i.e.,

\be\label{4.11} \partial_\mu {V}(t,x,P_\xi, y)=E[\partial_\mu (\Phi(X^{t,x,P_\xi}_T, P_{X^{t,\xi}_T}))(y)], \ y\in {\mathbb R}.\ee

\noindent But then from (\ref{4.7}) we obtain (\ref{4.4}) for $\partial_\mu {V}(t,x,P_\xi, y)$.

\noindent As concerns the derivative of $V(t,x,P_\xi)
=E[\Phi(X^{t,x,P_\xi}_T,P_{X^{t,\xi}_T})$ with respect to $x$,
since $z\mapsto\Phi(z,P_{X^{t,\xi}_T})$, is
a (deterministic) function with a bounded, Lipschitz continuous
derivative of first order, the computation of $\partial_xV(t,x,P_\xi)$ is
is standard.

Finally, concerning the estimates i) and ii) for the derivative $\partial_xV$
stated in Lemma 4.1, they are a direct consequence of the assumption
on $\Phi$ as well as the estimates for the involved processes,
studied in the preceding section.

In order to complete the proof, it remains still to prove iii). For
this end we observe, that due to Lemma 3.1, for arbitrarily given
$(t,x)\in [0,T]\times {\mathbb R}$ and $\xi\in L^2({\cal F}_t)$, $X^{t,x,P_\xi}=
X^{t,x,P_{\xi'}}$, for all $\xi'\in  L^2({\cal F}_t)$ with $P_{\xi'}=
P_\xi$. Since due to our assumption $L^2({\cal F}_0)$ is rich enough,
we can find some $\xi'\in  L^2({\cal F}_0)$ with $P_{\xi'}=
P_\xi$, which is independent of the driving Brownian
motion $B$. Using the time-shifted Brownian motion $B^t_s:=
B_{t+s}-B_t,\ s\ge 0$ (where we consider the Brownian motion $B$
extended beyond the time horizon $T$), we see that $X^{t,x,P_{\xi'}}$
and $X^{t,\xi'}$ solve the following SDEs:
\be\label{4.12} X^{t,\xi'}_{s+t}=\xi'+\int_0^s\sigma(
X^{t,\xi'}_{r+t},P_ {X^{t,\xi'}_{r+t}})dB^t_r,\ee
\be\label{4.13}X^{t,x,P_{\xi'}}_{s+t}=x+\int_0^s\sigma(
X^{t,x,P_{\xi'}}_{r+t},P_ {X^{t,\xi'}_{r+t}})dB^t_r,\ s\in[0,T-t].\ee

\smallskip

\noindent Consequently, $(X^{t,x,P_{\xi'}}_{.+t},X^{t,\xi'}_{.+t})$
and $(X^{0,x,P_{\xi'}},X^{0,\xi'})$ are solutions of the same system
of SDEs, only driven by different Brownian motions, $B^t$ and $B$,
respectively, both independent of $\xi'$. It follows that the laws
of $(X^{t,x,P_{\xi'}}_{.+t},X^{t,\xi'}_{.+t})$ and $(X^{0,x,P_{\xi'}},
X^{0,\xi'})$ coincide, and, hence,
\be\label{4.14}V(t,x,P_\xi)=V(t,x,P_{\xi'})=E[\Phi(X^{t,x,P_{\xi'}
}_{T},P_{X^{t,\xi'}_{T}})]=E[\Phi(X^{0,x,P_{\xi'}
}_{T-t},P_{X^{0,\xi'}_{T-t}})].\ee

\noindent Thus, for two different initial times $t,\ t'\in[0,T],$
using the fact that the derivatives of $\Phi$ are bounded, i.e., $\Phi$
is Lipschitz over ${\mathbb R}\times {\cal P}_2({\mathbb R})$, we obtain
\be\label{4.15}\begin{array}{lll}
& |V(t,x,P_\xi)-V(t',x,P_\xi)|\\
&\le E[|\Phi(X^{0,x,P_{\xi'}
}_{T-t},P_{X^{0,\xi'}_{T-t}})-\Phi(X^{0,x,P_{\xi'}
}_{T-t'},P_{X^{0,\xi'}_{T-t'}})|]\\
&\le C (E[|X^{0,x,P_{\xi'}}_{T-t}-X^{0,x,P_{\xi'}}_{T-t'}|]+W_2
 (P_{X^{0,\xi'}_{T-t}},P_{X^{0,\xi'}_{T-t'}}))\\
&\le C(E[|X^{0,x,P_{\xi'}
}_{T-t}-X^{0,x,P_{\xi'}}_{T-t'}|^2+
|X^{0,{\xi'}}_{T-t}-X^{0,{\xi'}}_{T-t'}|^2])^{1/2}.\end{array}\ee

\noindent But, taking into account the boundedness of the
coefficient $\sigma$ of the SDEs for $X^{0,x,P_{\xi'}
}$ and $X^{0,{\xi'}}$, we get
\be\label{4.16} |V(t,x,P_\xi)-V(t',x,P_\xi)|
\le C|t-t'|^{1/2}.\ee

\noindent The proof of the remaining estimate for the derivatives
$V$ is carried out by using the same kind of argument.
Indeed, considering the system of equations for $N^{t,x,P_{\xi'}}(y):=(X^{t,x,P_{\xi'}},\partial_x
X^{t,x,P_{\xi'}},U^{t,x,P_{\xi'}}(y)),$ $\ x,\ y\in {\mathbb R},$ and
$N^{t,\xi'}(y):=(X^{t,\xi'}, \partial_x
X^{t,\xi'}, U^{t,\xi'}(y)),\ y\in {\mathbb R},$
(see (\ref{3.2}), (\ref{3.63}), (\ref{3.42})) we see again that

$((N^{t,x,P_{\xi'}}_{.+t}(y))_{x,y\in {\mathbb R}},(N^{t,\xi'}_{.+t}
(y)_{y\in {\mathbb R}}))$
and $((N^{0,x,P_{\xi'}}(y))_{x,y\in {\mathbb R}},(N^{0,\xi'}
(y)_{y\in {\mathbb R}}))$ are equal in law.

\noindent Hence, from (\ref{4.11}) we deduce 
\be\label{4.17}\begin{array}{lll}
&\partial_\mu V(t,x,P_\xi,y)
=E\bigg[\partial_\mu(\Phi(X^{t,x,P_{\xi'}}_T,
P_{X^{t,\xi'}_T}))(y)\bigg]=E\bigg[\partial_\mu(\Phi(X^{0,x,P_{\xi'}}_{T-t},
P_{X^{0,\xi'}_{T-t}}))(y)\bigg].\end{array}\ee

\smallskip

\noindent Consequently, using the Lipschitz continuity and
the boundedness of $\partial_x\Phi:{\mathbb R}\times{\cal P}_2({\mathbb R})
\rightarrow {\mathbb R}$ and $\partial_\mu\Phi:{\mathbb R}\times{\cal P}_2({\mathbb R})
\times {\mathbb R}\rightarrow {\mathbb R}$ as well as the uniform boundedness
in $L^p$ ($p\ge 2$) of the first order derivatives
$\partial_x X^{0,x,P_{\xi'}}, \partial_\mu X^{0,x,P_{\xi'}}(y)$,
we get from (\ref{4.6}) with $(0, x, \xi', T-t)$ and $(0, x, \xi', T-t')$ instead of $(t, x, \xi, T)$, that
\be\label{4.18}\begin{array}{lll}
&|\partial_\mu V(t,x,P_\xi,y)-\partial_\mu
V(t',x,P_\xi,y)|\\
&\ \le C\bigg(E\big[|X^{0,x,P_{\xi'}}_{T-t}-
X^{0,x,P_{\xi'}}_{T-t'}|+|X^{0,y,P_{\xi'}}_{T-t}-
X^{0,y,P_{\xi'}}_{T-t'}|+|X^{0,{\xi'}}_{T-t}-
X^{0,{\xi'}}_{T-t'}| +|\partial_x
X^{0,y,P_{\xi'}}_{T-t}-\partial_x
X^{0,y,P_{\xi'}}_{T-t'}|\\
&\ \ \ \  \ +|\partial_\mu
X^{0,x,P_{\xi'}}_{T-t}(y)-\partial_\mu
X^{0,x,P_{\xi'}}_{T-t'}(y)|+|\partial_\mu
X^{0,\xi',P_{\xi'}}_{T-t}(y)-\partial_\mu
X^{0,\xi',P_{\xi'}}_{T-t'}(y)|\big]
+W_2(P_{X^{0,\xi'}_{T-t}},
P_{X^{0,\xi'}_{T-t'}})\bigg).\end{array}\ee

\noindent Thus, since $\xi'$ is independent of
$X^{0,x,P_{\xi'}},\ \partial_x X^{0,y,P_{\xi'}}$
and $\partial_\mu X^{0,x',P_{\xi'}}(y)$,
\be\label{4.19}\begin{array}{lll}
&|\partial_\mu V(t,x,P_\xi,y)-
\partial_\mu V(t',x,P_\xi,y)|\\
& \le C\cdot\sup_{x, y\in {\mathbb R}}
\left(E\big[|X^{0,x,P_{\xi'}}_{T-t}-
X^{0,x,P_{\xi'}}_{T-t'}|^2+|\partial_x
X^{0,x,P_{\xi'}}_{T-t}-\partial_x
X^{0,x,P_{\xi'}}_{T-t'}(y)|^2\right.\\
&\ \ \ \  \left.+|\partial_\mu
X^{0,x,P_{\xi'}}_{T-t}(y)-\partial_\mu
X^{0,x,P_{\xi'}}_{T-t'}(y)|^2\big]\right)^{1/2},\end{array}\ee

\noindent and the uniform boundedness in $L^2$ of the
derivatives of $X^{0,x,P_{\xi'}}$ allows to deduce from
the SDEs for $X^{0,x,P_{\xi'}},\partial_x
X^{0,x,P_{\xi'}}$ and $\partial_\mu
X^{0,x,P_{\xi'}}_{T-t'}(y)$ ((\ref{3.2}), (\ref{3.54}), (\ref{3.56})) that
\be\label{4.20} |\partial_\mu V(t,x,P_\xi,y)-
\partial_\mu V(t',x,P_\xi,y)|\le C|t-t'|^{1/2}.\ee

\noindent The proof of the corresponding estimate for
$\partial_x V(t,x,,P_\xi)$ is similar and, hence, omitted
here.
\end{proof}

Let us come now to the discussion of the second order
derivatives of our value function $V(t,x,P_\xi)$.
\begin{lemma} We suppose that Hypothesis (H.2) is
satisfied by the coefficients $\sigma$ and $b$, and we suppose
that $\Phi\in C^{2,1}({\mathbb R}^d\times{\cal P}_2({\mathbb R}^d)).$ Then, for all
$t\in [0,T],$ $V(t,.,.)\in C^{2,1}({\mathbb R}^d\times{\cal P}_2({\mathbb R}^d))$, the mixed second order derivatives are symmetric:
$$\partial_{x_i}(\partial_\mu V(t,x,P_\xi,y))=\partial_\mu(\partial_{x_i}V(t,x,P_\xi))(y),\ (t, x,y)\in [0, T]\times {\mathbb R}^d\times {\mathbb R}^d,\ \xi\in L^2({\cal F}_t, {\mathbb R}^d),\ 1\leq i\leq d,$$
and, for $$U(t,x,P_\xi,y,z)=(\partial_{x_i x_j}^2V(t,x,P_\xi),
\partial_{x_i}(\partial_\mu V(t,x,P_\xi,y)), \partial_\mu^2V(t,x,P_\xi,y,z), 
\partial_y(\partial_\mu V(t,x,P_\xi,y))),$$
there is some constant $C\in {\mathbb R}$ such that, for all
$t,\ t'\in[0,T],\ x,\ x',\ y,\ y'\in {\mathbb R},$ $\xi,\ \xi'\in L^2({\cal F}_t, {\mathbb R}^d),$
\be\label{4.21}\begin{array}{lll}
{\rm i)}\ |U(t,x,P_\xi,y,z)|\le C,\\
{\rm ii)}\ |U(t,x,P_\xi,y,z)-U(t,x',P_{\xi'},y',z')|
\le(|x-x'|+|y-y'|+W_2(P_\xi,P_{\xi'})),\\
{\rm iii)}\ |U(t,x,P_\xi,y,z)-U(t',x,P_\xi,y,z)| \le C|t-t'|^{1/2}.\end{array}\ee
\end{lemma}
\begin{proof} As in the preceding proofs, we make our
computations for the case of dimension $d=1$. Moreover, in our proof
we concentrate on the computation for the  second order derivative
with respect to the measure $\partial_\mu^2 V(t,x,P_\xi,y,z)$ and to its
estimates; using the preceding lemma on the derivatives of first order,
the computation of the  second order derivatives $\partial_x^2V(t,x,P_\xi),$
$\partial_x(\partial_\mu V(t,x,P_\xi)(y))$, $\partial_\mu(\partial_x
V(t,x,P_\xi))(y)$ and $\partial_y(\partial_\mu V(t,x,P_\xi)(y))$ and
their estimates are  rather direct and left to the interested reader. On the other hand, a direct computation based on (\ref{4.6}) and (\ref{4.11}) and using the symmetry of the mixed second order derivatives of $\Phi$ and the processes $X^{t,x,P_{\xi}}$ and $X^{t,{\xi}}$ (see Proposition 3.1) shows that 
$$\partial_{x}(\partial_\mu V(t,x,P_\xi,y))=\partial_\mu(\partial_{x}V(t,x,P_\xi))(y),\ (t, x,y)\in [0, T]\times {\mathbb R}\times {\mathbb R},\ \xi\in L^2({\cal F}_t).$$

For the computation of $\partial_\mu^2 V(t,x,P_\xi,y,z)$ we use the
formula for  $\partial_\mu V(t,x,P_\xi,y)$ stated in
Lemma 4.1, (\ref{4.11}) and (\ref{4.6}):
\be\label{4.22} \partial_\mu V(t,x,P_\xi,y)=
V_{1}(t,x,P_\xi,y)+V_{2}(t,x,P_\xi,y)+V_{3}(t,x,P_\xi,y),\ee

\noindent with

$\displaystyle V_1(t,x,P_\xi,y)=E[(\partial_x\Phi)(X^{t,x,P_\xi}_T,P_{X^{t,\xi}_T})
\cdot(\partial_\mu X^{t,x,P_\xi}_T)_{i}(y)],$

$\displaystyle V_2(t,x,P_\xi,y)=E[\widetilde{E}[(\partial_\mu\Phi)(X^{t,x,P_\xi}_T,
P_{X^{t,\xi}_T},\widetilde{X}_T^{t,y,P_\xi})\cdot
\partial_x \widetilde{X}^{t,y,P_\xi}_T]],$

$\displaystyle V_3(t,x,P_\xi,y)=E[\widetilde{E}[(\partial_\mu\Phi)(X^{t,x,P_\xi}_T,
P_{X^{t,\xi}_T},\widetilde{X}_T^{t,\widetilde{\xi}})
\cdot(\partial_\mu \widetilde{X}^{t,\widetilde{\xi},
P_\xi}_T)(y)]].$

\smallskip

\noindent Let us consider $V_3(t,x,P_\xi,y)$, the discussion for
$V_1(t,x,P_\xi,y)$ and $V_2(t,x,P_\xi,y)$ is analogous. Using the Fr\'{e}chet differentiability of the terms involved in the definition of $V_3$, 
we obtain for the Fr\'{e}chet derivative of
\be\label{4.23}\xi\mapsto\widetilde{V}_3(t,x,\xi,y):=V_3(t,x,
P_\xi,y),\ (t,x,\xi)\in[0,T]\times {\mathbb R}\times L^2({\cal F}_t)\ee

\noindent that, for all $\eta\in L^2({\cal F}_t)$,
\be\label{4.24}\begin{array}{lll}
& D\widetilde{V}_3(t,x,\xi,y)(\eta)\\
&=E\big[\widetilde{E}\big[(\partial_\mu
\Phi)(X^{t,x,P_\xi}_T,P_{X^{t,\xi}_T},\widetilde{X}_T^{t,
\widetilde{\xi}})\cdot D_{\widetilde{\xi}}[(\partial_\mu
\widetilde{X}^{t,\widetilde{\xi},P_\xi}_T)(y)](\widetilde{\eta})\\
&\ \ \ +\partial_x(\partial_\mu\Phi)(X^{t,x,P_\xi}_T,
P_{X^{t,\xi}_T},\widetilde{X}_T^{t,\widetilde{\xi}})
\cdot(\partial_\mu \widetilde{X}^{t,\widetilde{\xi},
P_\xi}_T)(y)\cdot D_\xi[X^{t,x,P_\xi}_T](\eta)\\
&\ \ \ +\overline{E}[\partial_\mu
(\partial_\mu\Phi)(X^{t,x,P_\xi}_T,P_{X^{t,\xi}_T},
\widetilde{X}_T^{t,\widetilde{\xi}},\overline{X}^{t,
\overline{\xi}}_T)
\cdot D_{\overline{\xi}}[\overline{X}^{t,
\overline{\xi}}_T](\overline{\eta})]\cdot(\partial_\mu
\widetilde{X}^{t,\widetilde{\xi},P_\xi}_T)(y)\\
&\ \ \ +\partial_y(\partial_\mu\Phi)(X^{t,x,
P_\xi}_T,P_{X^{t,\xi}_T},\widetilde{X}_T^{t,\widetilde{\xi}})
\cdot(\partial_\mu \widetilde{X}^{t,\widetilde{\xi},
P_\xi}_T)(y)\cdot D_{\widetilde{\xi}}[\widetilde{X}_T^{t,
\widetilde{\xi}}](\widetilde{\eta})\big]\big]\end{array}\ee

\noindent (recall the notations introduced in the proof of Theorem 3.1 and Proposition 3.1). On the other hand, we know already that
\be\label{4.25}\begin{array}{lll}
&{\rm i)}\ D_\xi[X^{t,x,P_\xi}_T](\eta)=
\widehat{E}[\partial_\mu X^{t,x,P_\xi}_T(\widehat{\xi})\cdot
\widetilde{\eta}],\\
&{\rm ii)}\ D_{\overline{\xi}}
[\overline{X}^{t,\overline{\xi}}_T](\overline{\eta})=
\partial_x\overline{X}_T^{t,
\overline{\xi},P_{\xi}}\cdot\overline{\eta}+
\widehat{E}[\partial_\mu \overline{X}^{t,\overline{\xi},
P_\xi}_T(\widehat{\xi})\cdot\widehat{\eta}],\\
&{\rm iii)}\ D_{\widetilde{\xi}}
[\widetilde{X}_T^{t,\widetilde{\xi}}](\widetilde{\eta})=
\partial_x\widetilde{X}_T^{t,
\widetilde{\xi},P_{\xi}}\cdot\widetilde{\eta}+
\widehat{E}[\partial_\mu
\widetilde{X}_T^{t,\widetilde{\xi},P_\xi}(\widehat{\xi})
\cdot\widehat{\eta}],\\
&{\rm iv)}\ D_{\widetilde{\xi}}[(\partial_\mu
\widetilde{X}^{t,\widetilde{\xi},P_\xi}_T)(y)]
(\widetilde{\eta})=
\partial_x(\partial_\mu \widetilde{X}^{t,
\widetilde{\xi},P_\xi}_T(y))\cdot\widetilde{\eta}+
\widehat{E}[(\partial^2_\mu
\widetilde{X}^{t,\widetilde{\xi},P_\xi}_T)(y,
\widehat{\xi})\cdot\widehat{\eta}].\end{array}\ee

\smallskip

\noindent Consequently, we have
\be\label{4.26}\begin{array}{lll}
& D\widetilde{V}_3(t,x,\xi,y)(\eta)\\
&=\widehat{E}\big[E\big[\widetilde{E}
\big[(\partial_\mu\Phi)(X^{t,x,P_\xi}_T,P_{X^{t,
\xi}_T},\widetilde{X}_T^{t,\widetilde{\xi}})\cdot
\big(\partial_x(\partial_\mu \widetilde{X}^{t,
\widehat{\xi},P_\xi}_T(y))+
(\partial^2_\mu\widetilde{X}^{t,\widetilde{\xi},
P_\xi}_T)(y,\widehat{\xi})\big)\\
&\ \  +\partial_x(\partial_\mu\Phi)
(X^{t,x,P_\xi}_T,P_{X^{t,\xi}_T},\widetilde{X}_T^{t,
\widetilde{\xi}})\cdot(\partial_\mu \widetilde{X}^{t,
\widetilde{\xi},P_\xi}_T)(y)\cdot\partial_\mu
X^{t,x,P_\xi}_T(\widehat{\xi})\\
&\ \  +\overline{E}[\partial_\mu
(\partial_\mu\Phi)(X^{t,x,P_\xi}_T,P_{X^{t,\xi}_T},
\widetilde{X}_T^{t,\widetilde{\xi}},\overline{X}^{t,
\overline{\xi}}_T)
\cdot(\partial_\mu\widetilde{X}^{t,\widetilde{\xi},
P_\xi}_T)(y)\cdot \big(\partial_x\overline{X}_T^{t,
\widehat{\xi},P_{\xi}}+\partial_\mu
\overline{X}^{t,\overline{\xi},P_\xi}_T
(\widehat{\xi}) \big)\\
&\ \  +\partial_y(\partial_\mu\Phi)(X^{t,x,
P_\xi}_T,P_{X^{t,\xi}_T},\widetilde{X}_T^{t,\widetilde{\xi}})
\cdot(\partial_\mu \widetilde{X}^{t,\widetilde{\xi},
P_\xi}_T)(y)\cdot\big(\partial_x\widetilde{X}_T^{t,
\widehat{\xi},P_{\xi}}+\partial_\mu
\widetilde{X}_T^{t,\widetilde{\xi},P_\xi}
(\widehat{\xi}))\big]\big]\cdot\widehat{\eta}\big].\end{array}\ee

\noindent Therefore,
\be\label{4.27}\begin{array}{lll}
&\partial_\mu{V}_3(t,x,P_\xi,y,z)\\
&=E\big[\widetilde{E}
\big[(\partial_\mu\Phi)(X^{t,x,P_\xi}_T,P_{X^{t,
\xi}_T},\widetilde{X}_T^{t,\widetilde{\xi}})\cdot
\big(\partial_x(\partial_\mu \widetilde{X}^{t,
z,P_\xi}_T(y))+
(\partial^2_\mu\widetilde{X}^{t,\widetilde{\xi},
P_\xi}_T)(y,z)\big)\\
&\ \ +\partial_x(\partial_\mu\Phi)
(X^{t,x,P_\xi}_T,P_{X^{t,\xi}_T},\widetilde{X}_T^{t,
\widetilde{\xi}})\cdot\partial_\mu \widetilde{X}^{t,
\widetilde{\xi},P_\xi}_T(y)\cdot\partial_\mu
X^{t,x,P_\xi}_T(z)\\
&\ \ +\overline{E}[\partial_\mu
(\partial_\mu\Phi)(X^{t,x,P_\xi}_T,P_{X^{t,\xi}_T},
\widetilde{X}_T^{t,\widetilde{\xi}},\overline{X}^{t,
\overline{\xi}}_T)
\cdot\partial_\mu\widetilde{X}^{t,\widetilde{\xi},
P_\xi}_T(y)\cdot \big(\partial_x\overline{X}_T^{t,
z,P_{\xi}}+\partial_\mu
\overline{X}^{t,\overline{\xi},P_\xi}_T
(z) \big)\\
&\ \ +\partial_y(\partial_\mu\Phi)(X^{t,x,
P_\xi}_T,P_{X^{t,\xi}_T},\widetilde{X}_T^{t,\widetilde{\xi}})
\cdot(\partial_\mu \widetilde{X}^{t,\widetilde{\xi},
P_\xi}_T)(y)\cdot\big(\partial_x\widetilde{X}_T^{t,
z,P_{\xi}}+\partial_\mu
\widetilde{X}_T^{t,\widetilde{\xi},P_\xi}
(z))\big]\big],\end{array}\ee

\noindent $t\in[0,T],\ x,\ y,\ z\in {\mathbb R},\ \xi\in L^2({\cal F}_t)$,
satisfies the relation
\be\label{4.28} D\widetilde{V}_3(t,x,\xi,y)
(\eta)=\widehat{E}\left[\partial_\mu{V}_3(t,x,P_\xi,y,
\widehat{\xi})\cdot\widehat{\eta}\right],\ee

\noindent i.e., the function $\partial_\mu{V}_3(t,x,
P_\xi,y,z)$ is the derivative of $ V_3(t,x,P_\xi,y)$
with respect to the measure at $P_\xi$. Moreover, the above
expression for  $\partial_\mu{V}_3(t,x,
P_\xi,y,z)$ combined with the estimates for the process
$X^{t,x,P_\xi}$ and those of its first and second order
derivatives studied in the preceding sections allows to
obtain after a direct computation
\be\label{4.29}|\partial_\mu{V}_3(t,x,
P_\xi,y,z)-\partial_\mu{V}_3(t,x',
P_\xi',y',z')|\le C(|x-x'|+|y-y'|+|z-z'|+W_2(P_\xi,
P_{\xi'})),\ee

\noindent for all $t\in[0,T],\ x,\ x',\ y,\ y',\ z,\ z'\in {\mathbb R}$ and
$\xi,\ \xi'\in L^2({\cal F}_t).$ Furthermore, extending
in a direct way the corresponding argument for the estimate
of the difference for the first order derivatives of
$V(t,x,P_\xi)$ at different time points (see (\ref{4.18})), we deduce from the
explicit expression for $\partial_\mu{V}_3(t,x,P_\xi,y,z)$
that, for some real $C\in {\mathbb R},$
\be\label{4.30}|\partial_\mu{V}_3(t,x,P_\xi,y,z)-
\partial_\mu{V}_3(t',x,P_\xi,y,z)|\le C|t-t'|^{1/2},\ee

\noindent  for all $t,\ t'\in[0,T],\ x,\ y,\ z\in {\mathbb R}$ and $\xi
\in L^2({\cal F}_t).$

In the same manner as we obtained the wished results for
$\partial_\mu V_3(t,x,P_\xi,y,z)$, we can investigate
$\partial_\mu V_1(t,x,P_\xi,y,z)$ and $\partial_\mu V_2(t,x,
P_\xi,y,z)$. This yields the wished results for $\partial_\mu^2
V(t,x,P_\xi,y,z).$ The proof is complete.\end{proof}
\section{It\^{o} formula and PDE associated with
mean-field SDE}
\bp Let $\Phi\in C_b^{2,1}({\mathbb R}^d
\times{\cal P}_2({\mathbb R}^d))$. Then, under Hypothesis (H.2), for
all $0\le t\le s\le T,\ x\in {\mathbb R}^d,\ \xi\in L^2({\cal F}_t;{\mathbb R}^d)$ the
following It\^{o} formula is satisfied:
\be\label{5.1}\begin{array}{lll}
& \Phi(X^{t,x,P_\xi}_s,P_{X^{t,\xi}_s})-
\Phi(x,P_\xi)\\
& =\displaystyle\int_t^s\bigg(\sum_{i=1}^d\partial_{x_i}
\Phi(X^{t,x,P_\xi}_r,P_{X^{t,\xi}_r})b_{i}
(X^{t,x,P_\xi}_r,P_{X^{t,\xi}_r})\\
&\ \ \ \displaystyle+\frac12\sum_{i,j,k=1}^d
\partial_{x_i x_j}^2\Phi(X^{t,x,P_\xi}_r,P_{X^{t,\xi}_r})(\sigma_{i,k}
\sigma_{j,k})(X^{t,x,P_\xi}_r,P_{X^{t,\xi}_r})\\
&\ \ \ \displaystyle+\widetilde{E}\big[\sum_{i=1}^d(\partial_\mu \Phi)_i
(X^{t,x,P_\xi}_r,P_{X^{t,\xi}_r},\widetilde{X}^{t,
\widetilde{\xi}}_r)b_{i}(\widetilde{X}^{t,\widetilde{\xi}}_r,P_{X^{t,\xi}_r})\\
&\ \ \ \displaystyle+\frac12\sum_{i,j,k=1}^d\partial_{y_i}\left((\partial_\mu \Phi)_j
(X^{t,x,P_\xi}_r,P_{X^{t,\xi}_r},\widetilde{X}^{t,\widetilde{\xi}}_r)
\right)(\sigma_{i,k}\sigma_{j,k})(\widetilde{X}^{t,\widetilde{\xi}}_s,
P_{X^{t,\xi}_s})\big]\bigg)dr\\
&\ \ \ \displaystyle+\int_t^s\sum_{i,j=1}^d\partial_{x_i}\Phi(X^{t,x,
P_\xi}_r,P_{X^{t,\xi}_r})\sigma_{i,j}(X^{t,x,P_\xi}_r,
P_{X^{t,\xi}_r})dB^j_r,\  s\in[t,T].\end{array}\ee\ep

\begin{proof} As already before in other proofs
let us again restrict ourselves to dimension $d=1$. The
general case is gotten by a straight-forward extension.

Let $0\le t<s\le T,\ u\in {\mathbb R},\ \xi\in L^2(
{\cal F}_t)$ and put $t_i^n:=t+i(s-t)2^{-n},\, 0\le i\le 2^n,\
n\ge 1.$ Then, since $\Phi(u,.)\in C^{2,1}
({\cal P}_2({\mathbb R}))$, we have due to Lemma 2.1
\be\label{5.2}\begin{array}{lll}
&\Phi(u,P_{X^{t,\xi}_s})-\Phi(u,P_\xi)=\displaystyle
\sum_{i=0}^{2^n-1}\left(\Phi(u,P_{X^{t,\xi}_{t_{i+1}^n}})-\Phi(u,
P_{X^{t,\xi}_{t_{i}^n}})\right)\\
&\displaystyle=\sum_{i=0}^{2^n-1}\bigg(\widetilde{E}[\partial_\mu
\Phi(u,P_{X^{t,\xi}_{t_{i}^n}},
\widetilde{X}^{t,\widetilde{\xi}}_{t_{i}^n})(\widetilde{X}^{t,
\widetilde{\xi}}_{t_{i+1}^n}-\widetilde{X}^{t,
\widetilde{\xi}}_{t_{i}^n})]\\
&\displaystyle\ \ \ +\frac12\widetilde{E}[\overline{E}[\partial^2_\mu\Phi(u,P_{X^{t,\xi}_{t_{i}^n}},\widetilde{X}^{t,
\widetilde{\xi}}_{t_{i}^n},\overline{X}^{t,\overline{\xi}}_{t_{i}^n})
(\widetilde{X}^{t,\widetilde{\xi}}_{t_{i+1}^n}-\widetilde{X}^{t,
\widetilde{\xi}}_{t_{i}^n})(\overline{X}^{t,
\overline{\xi}}_{t_{i+1}^n}-\overline{X}^{t,
\overline{\xi}}_{t_{i}^n})]]\\
&\displaystyle\ \ \ +\frac12\widetilde{E}
[\partial_y\partial_\mu\Phi(u,P_{X^{t,\xi}_{t_{i}^n}},\widetilde{X}^{t,
\widetilde{\xi}}_{t_{i}^n})
(\widetilde{X}^{t,\widetilde{\xi}}_{t_{i+1}^n}-\widetilde{X}^{t,
\widetilde{\xi}}_{t_{i}^n})^2]+R_{t_i^n}(P_{X^{t,\xi}_{t_{i+1}^n}},
P_{X^{t,\xi}_{t_{i}^n}})\bigg)\end{array}\ee

\smallskip

\noindent (For the notations we refer to the preceding sections),
where, for some $C\in {\mathbb R}_+$ depending only on the Lipschitz constants of $\partial_\mu^2\Phi$ and $\partial_y\partial_\mu\Phi$,
\be\label{5.3}\begin{array}{lll}
& |R_{t_i^n}(P_{X^{t,\xi}_{t_{i+1}^n}},
P_{X^{t,\xi}_{t_{i}^n}})|\le CE[|X^{t,\xi}_{t_{i+1}^n}
-X^{t,\xi}_{t_{i}^n}|^3]\\
& \le C\left(E[(\int_{t_i^n}^{t_{i+1}^n}|b(X^{t,x,P_\xi}_r,P_{X^{t,\xi}_r})|dr)^3]+E[(\int_{t_i^n}^{t_{i+1}^n}
|\sigma(X^{t,x,P_\xi}_r,P_{X^{t,\xi}_r})|^2dr)^{3/2}]\right)\\
&\le C(t_{i+1}^n-t_i^n)^{3/2},\ 0\le i\le 2^n-1.\end{array}\ee

\smallskip

\noindent Thus, taking into account the relations

\smallskip

i)\, \ $\displaystyle\widetilde{E}[\partial_\mu
\Phi(u,P_{X^{t,\xi}_{t_{i}^n}},
\widetilde{X}^{t,\widetilde{\xi}}_{t_{i}^n})(\widetilde{X}^{t,
\widetilde{\xi}}_{t_{i+1}^n}-\widetilde{X}^{t,
\widetilde{\xi}}_{t_{i}^n})]$

\quad \, \  $\displaystyle =\widetilde{E}[\partial_\mu
\Phi(u,P_{X^{t,\xi}_{t_{i}^n}},\widetilde{X}^{t,
\widetilde{\xi}}_{t_{i}^n})\cdot (\int_{t_i^n}^{t_{i+1}^n}
\sigma(\widetilde{X}^{t,\widetilde{\xi}}_r,P_{X^{t,
\xi}_r})d\widetilde{B}_r+
\int_{t_i^n}^{t_{i+1}^n}
b(\widetilde{X}^{t,\widetilde{\xi}}_r,P_{X^{t,
\xi}_r})dr)]$

\quad \, \  $\displaystyle =\widetilde{E}[\partial_\mu
\Phi(u,P_{X^{t,\xi}_{t_{i}^n}},\widetilde{X}^{t,
\widetilde{\xi}}_{t_{i}^n})\cdot
\int_{t_i^n}^{t_{i+1}^n}
b(\widetilde{X}^{t,\widetilde{\xi}}_r,P_{X^{t,
\xi}_r})dr],$

ii)\, $\displaystyle\widetilde{E} [\overline{E}[\partial^2_\mu\Phi(u,P_{X^{t,\xi}_{t_{i}^n}},\widetilde{X}^{t,
\widetilde{\xi}}_{t_{i}^n},\overline{X}^{t,\overline{\xi}}_{t_{i}^n})
(\widetilde{X}^{t,\widetilde{\xi}}_{t_{i+1}^n}-\widetilde{X}^{t,
\widetilde{\xi}}_{t_{i}^n})(\overline{X}^{t,
\overline{\xi}}_{t_{i+1}^n}-\overline{X}^{t,
\overline{\xi}}_{t_{i}^n}) ] ]$

\quad \, $\displaystyle =\widetilde{E} [\overline{E}[\partial^2_\mu\Phi(u,P_{X^{t,\xi}_{t_{i}^n}},\widetilde{X}^{t,
\widetilde{\xi}}_{t_{i}^n},\overline{X}^{t,\overline{\xi}}_{t_{i}^n})
\cdot (\int_{t_i^n}^{t_{i+1}^n}\sigma(\widetilde{X}^{t,\widetilde{\xi}}_r,
P_{X^{t,\xi}_r})d\widetilde{B}_r
+\int_{t_i^n}^{t_{i+1}^n}b(\widetilde{X}^{t,\widetilde{\xi}}_r,P_{X^{t,
\xi}_r})dr)$

\qquad $\displaystyle \times(\int_{t_i^n}^{t_{i+1}^n}
\sigma(\overline{X}^{t,\overline{\xi}}_r,P_{X^{t,\xi}_r})
d\overline{B}_r+\int_{t_i^n}^{t_{i+1}^n}
b(\overline{X}^{t,\overline{\xi}}_r,P_{X^{t,
\xi}_r})dr) ] ]$

 \quad \, $\displaystyle =\widetilde{E} [\overline{E}[\partial^2_\mu\Phi(u,P_{X^{t,\xi}_{t_{i}^n}},\widetilde{X}^{t,
\widetilde{\xi}}_{t_{i}^n},\overline{X}^{t,\overline{\xi}}_{t_{i}^n})
\cdot (\int_{t_i^n}^{t_{i+1}^n}b(\widetilde{X}^{t,\widetilde{\xi}}_r,
P_{X^{t,{\xi}}_r})dr)(\int_{t_i^n}^{t_{i+1}^n}b(\overline{X}^{t,
\overline{\xi}}_r,P_{X^{t,{\xi}}_r})dr) ] ],$

iii) $\displaystyle\widetilde{E} [
\partial_y\partial_\mu\Phi(u,P_{X^{t,\xi}_{t_{i}^n}},\widetilde{X}^{t,
\widetilde{\xi}}_{t_{i}^n})(\widetilde{X}^{t,
\widetilde{\xi}}_{t_{i+1}^n}-\widetilde{X}^{t,
\widetilde{\xi}}_{t_{i}^n})^2 ]$

\quad $\displaystyle =\widetilde{E} [
\partial_y\partial_\mu\Phi(u,P_{X^{t,\xi}_{t_{i}^n}},\widetilde{X}^{t,
\widetilde{\xi}}_{t_{i}^n}) (\int_{t_i^n}^{t_{i+1}^n}
\sigma(\widetilde{X}^{t,\widetilde{\xi}}_r,P_{X^{t,{\xi}}_r})
d\widetilde{B}_r+\int_{t_i^n}^{t_{i+1}^n}
b(\widetilde{X}^{t,\widetilde{\xi}}_r,P_{X^{t,
{\xi}}_r})dr)^2 ]$

\quad $\displaystyle =\widetilde{E} [
\partial_y\partial_\mu\Phi(u,P_{X^{t,\xi}_{t_{i}^n}},\widetilde{X}^{t,
\widetilde{\xi}}_{t_{i}^n})\cdot\int_{t_i^n}^{t_{i+1}^n}
|\sigma(\widetilde{X}^{t,\widetilde{\xi}}_r,P_{X^{t,
{\xi}}_r})|^2dr ]+Q_{t_i^n}$,

\noindent with $|Q_{t_i^n}|\le C(t_{i+1}^n-t_i^n)^{3/2},$ $0\le i\le 2^n-1,$\ as well as the continuity of
$r\mapsto (X^{t,x,P_\xi}_r,P_{X^{t,\xi}_r})\in L^2({\cal F};{\mathbb R}
\times{\cal P}_2({\mathbb R})),$ we get from the above sum over the second order
Taylor expansions, as $n\rightarrow +\infty$,
\be\label{5.4}\begin{array}{lll}
&\Phi(u,P_{X^{t,\xi}_s})-\Phi(u,P_\xi)
=\displaystyle\int_t^s\widetilde{E}\left[b(\widetilde{X}^{t,\widetilde{\xi}}_r,
P_{X^{t,{\xi}}_r})\partial_\mu\Phi(u,P_{X^{t,\xi}_r},
\widetilde{X}^{t,\widetilde{\xi}}_r)\right]dr\\
&\ \ \ \displaystyle+\frac12\int_t^s\widetilde{E}\left[\sigma^2(
\widetilde{X}^{t,\widetilde{\xi}}_r,P_{X^{t,{\xi}}_r})(\partial_y
\partial_\mu\Phi)(u,P_{X^{t,\xi}_r},
\widetilde{X}^{t,\widetilde{\xi}}_r)\right]dr,\  s\in[t,T].\end{array}\ee

\noindent From this latter relation we see that, for fixed $(t,\xi)$,
the function $\Psi(s,u):=\Phi(u,P_{X^{t,\xi}_s}),\, (s,u)\in[t,T]\times
{\mathbb R},$ is continuously differentiable in $s$ and twice continuously
differentiable in $u$, and all the corresponding derivatives
are bounded. In particular,
\be\label{5.5} \partial_s\Psi(s,u)=\widetilde{E}
\left[b(\widetilde{X}^{t,\widetilde{\xi}}_s,
P_{X^{t,{\xi}}_s})\partial_\mu\Phi(u,P_{X^{t,\xi}_s},
\widetilde{X}^{t,\widetilde{\xi}}_s)+\frac12\sigma^2(\widetilde{X}^{t,
\widetilde{\xi}}_s,P_{X^{t,\xi}_s})(\partial_y\partial_\mu\Phi)(u,P_{X^{t,\xi}_s},
\widetilde{X}^{t,\widetilde{\xi}}_s)\right],\ee

\smallskip

\noindent $(s,u)\in[t,T]\times {\mathbb R}.$ Consequently, we can apply to
$\Psi(s,X^{t,x,P_\xi}_s)
\ (=\Phi(X^{t,x,P_\xi}_s,P_{X^{t,\xi}_s}))$ the classical It\^{o}
formula. This yields
\be\label{5.6}\begin{array}{lll}
&\Phi(X^{t,x,P_\xi}_s,P_{X^{t,\xi}_s})-\Phi(x,P_\xi)=\displaystyle
\Psi(s,X^{t,x,P_\xi}_s)-\Psi(t,x)\\
&\displaystyle =\int_t^s\left(\partial_r\Psi(r,X^{t,x,P_\xi}_r)
+b(X^{t,x,P_\xi}_r,P_{X^{t,\xi}_r})\partial_x\Psi(r,X^{t,x,P_\xi}_r)
+\frac12\sigma^2(X^{t,x,P_\xi}_r,P_{X^{t,\xi}_r})\partial_x^2
\Psi(r,X^{t,x,P_\xi}_r)\right)dr\\
&\displaystyle\ \ \ \ \ +\int_t^s\sigma(X^{t,x,P_\xi}_r,
P_{X^{t,\xi}_r})\partial_x\Psi(r,X^{t,x,P_\xi}_r)dB_r\\
&\displaystyle =\int_t^s\bigg(\widetilde{E}\left[b
(\widetilde{X}^{t,\widetilde{\xi}}_r,
P_{X^{t,{\xi}}_r})\partial_\mu\Phi(u,P_{X^{t,\xi}_r},
\widetilde{X}^{t,\widetilde{\xi}}_r)+\frac12\sigma^2(\widetilde{X}^{t,
\widetilde{\xi}}_s,P_{X^{t,\xi}_s})(\partial_y\partial_\mu\Phi)
(X^{t,x,P_\xi}_r,
P_{X^{t,\xi}_s},\widetilde{X}^{t,\widetilde{\xi}}_s)\right]\\
&\displaystyle\ \ \ \ \  +b(X^{t,x,P_\xi}_r,P_{X^{t,\xi}_r})
\partial_x\Phi(X^{t,x,P_\xi}_r,P_{X^{t,\xi}_r})+\frac12\sigma^2(X^{t,x,
P_\xi}_r,P_{X^{t,\xi}_r})\partial_x^2\Phi(X^{t,x,P_\xi}_r,
P_{X^{t,\xi}_r})\bigg)dr\\
&\displaystyle\ \ \ \ \  +\int_t^s\sigma(X^{t,x,P_\xi}_r,
P_{X^{t,\xi}_r})\partial_x\Phi(X^{t,x,P_\xi}_r,P_{X^{t,\xi}_r})dB_r,\  s\in[t,T].\end{array}\ee

\noindent The proof is complete.\end{proof}

The preceding proposition can be extended without difficulties in a
straight-forward computation to the following case:
\bt Let $F:[0,T]\times {\mathbb R}^d\times{\cal P}({\mathbb R}^d)
\rightarrow {\mathbb R}$ be such that $F(t,.,.)\in C_b^{2,1}({\mathbb R}^d\times{\cal P}({\mathbb R}^d))$, for
all $t\in [0,T],$ $F(.,x,\mu)\in C^1([0,T]),$ for all $(x,\mu)\in {\mathbb R}^d
\times {\cal P}_2({\mathbb R}^d)$, and all derivatives, with respect to $t$ of first order, and
with respect to $(x,\mu)$ of first and of second order, are uniformly bounded over
$[0,T]\times {\mathbb R}^d\times{\cal P}({\mathbb R}^d)$ (for short: $F\in C_b^{1,(2,1)}([0,T]\times
{\mathbb R}^d\times{\cal P}_2({\mathbb R}^d))$). Then, under Hypothesis (H.2), for
all $0\le t\le s\le T,\ x\in {\mathbb R}^d,\ \xi\in L^2({\cal F}_t;{\mathbb R}^d)$, the
following It\^{o} formula is satisfied:
\be\label{5.7}\begin{array}{lll}
&F(s,X^{t,x,P_\xi}_s,P_{X^{t,\xi}_s})-F(t,x,P_\xi)\\
&=\displaystyle\int_t^s\bigg(\partial_r
F(r,X^{t,x,P_\xi}_r,P_{X^{t,\xi}_r})+\sum_{i=1}^d\partial_{x_i}
F(r,X^{t,x,P_\xi}_r,P_{X^{t,\xi}_r})b_{i}
(X^{t,x,P_\xi}_r,P_{X^{t,\xi}_r})\\
&\displaystyle\ \ \ +\frac12\sum_{i,j,k=1}^d
\partial_{x_i x_j}^2F(r,X^{t,x,P_\xi}_r,P_{X^{t,\xi}_r})(\sigma_{i,k}
\sigma_{j,k})(X^{t,x,P_\xi}_r,P_{X^{t,\xi}_r})\\
&\displaystyle\ \ \ +\widetilde{E}\big[\sum_{i=1}^d(\partial_\mu F)_i
(r,X^{t,x,P_\xi}_r,P_{X^{t,\xi}_r},\widetilde{X}^{t,
\widetilde{\xi}}_r)b_{i}(\widetilde{X}^{t,\widetilde{\xi}}_r,P_{X^{t,\xi}_r})\\
&\displaystyle\ \ \ +\frac12\sum_{i,j,k=1}^d\partial_{y_i}\left((\partial_\mu F)_j
(r,X^{t,x,P_\xi}_r,P_{X^{t,\xi}_r},\widetilde{X}^{t,\widetilde{\xi}}_r)
\right)(\sigma_{i,k}\sigma_{j,k})(\widetilde{X}^{t,\widetilde{\xi}}_r,
P_{X^{t,\xi}_r})\big]\bigg)dr\\
&\displaystyle\ \ \  +\int_t^s\sum_{i,j=1}^d\partial_{x_i}F(r,X^{t,x,
P_\xi}_r,P_{X^{t,\xi}_r})\sigma_{i,j}(X^{t,x,P_\xi}_r,
P_{X^{t,\xi}_r})dB^j_r,\  s\in[t,T].\end{array}\ee\et

 The above It\^{o} formula applied to
$\Phi(X^{t,x,P_\xi}_s,P_{X^{t,\xi}_s})$ allows now to show that
our value function $V(t,x,P_\xi)$ is continuously differentiable with
respect to $t$, with a derivative $\partial_tV$ bounded over $[0,T]
\times {\mathbb R}^d\times{\cal P}_2({\mathbb R}^d).$

\bl Assume that $\Phi\in C^{2,1}({\mathbb R}^d\times
{\cal P}_2({\mathbb R}^d))$. Then, under Hypothesis (H.2), $V\in C^{1,(2,1)}([0,T]\times
{\mathbb R}^d\times{\cal P}_2({\mathbb R}^d))$ and its derivative $\partial_tV(t,x,P_\xi)$ with
respect to $t$ verifies, for some constant $C\in {\mathbb R}$,
\be\label{5.8}\begin{array}{lll}
&{\rm i)}\ |\partial_tV(t,x,P_\xi)|\le C,\\
&{\rm ii)}\ |\partial_tV(t,x,P_\xi)-\partial_tV(t,x',P_{\xi'})|\le C(|x-x'|+W_2(P_\xi,P_{\xi'})),\\
&{\rm iii)}\ |\partial_tV(t,x,P_\xi)-\partial_tV(t',x,P_\xi)|\le C|t-t'|^{1/2},\\ \end{array}\ee

\noindent for all $t,\ t'\in[0,T],\ x,\ x'\in {\mathbb R}^d,\ \xi,\ \xi'\in L^2({\cal F}_t).$
\el
\begin{proof} Recall that, for $t\in[0,T],\ x\in {\mathbb R}^d,$ and $\xi$
(which can be supposed without loss of generality to belong to $L^2({\cal F}_0)$;
see our previous discussion in the proof of Lemma 4.1), we have
\be\label{5.9}V(t,x,P_\xi)=E[\Phi(X^{t,x,P_\xi}_T,P_{X^{t,\xi}_T})]
=E[\Phi(X^{0,x,P_\xi}_{T-t},P_{X^{0,\xi}_{T-t}})].\ee

\noindent Hence, taking the expectation over the It\^{o} formula in the last
but one  proposition for $s=T-t$ and initial time $0$, we get
\be\label{5.10}\begin{array}{lll}
&V(t,x,P_\xi)-V(T,x,P_\xi)=\displaystyle\int_0^{T-t}E\bigg[\bigg(\sum_{i=1}^d\partial_{x_i}
\Phi(X^{0,x,P_\xi}_r,P_{X^{0,\xi}_r})b_{i}
(X^{0,x,P_\xi}_r,P_{X^{0,\xi}_r})\\
&\ \ \ \displaystyle+\frac12\sum_{i,j,k=1}^d
\partial_{x_i x_j}^2\Phi(X^{0,x,P_\xi}_r,P_{X^{0,\xi}_r})(\sigma_{i,k}
\sigma_{j,k})(X^{0,x,P_\xi}_r,P_{X^{0,\xi}_r})\\
&\ \ \ \displaystyle+\widetilde{E}\big[\sum_{i=1}^d(\partial_\mu \Phi)_i
(X^{0,x,P_\xi}_r,P_{X^{0,\xi}_s},\widetilde{X}^{0,
\widetilde{\xi}}_r)b_{i}(\widetilde{X}^{0,\widetilde{\xi}}_r,P_{X^{0,\xi}_r})\\
&\ \ \ \displaystyle+\frac12\sum_{i,j,k=1}^d\partial_{y_i}\left((\partial_\mu \Phi)_j
(X^{0,x,P_\xi}_r,P_{X^{0,\xi}_r},\widetilde{X}^{0,\widetilde{\xi}}_r)
\right)(\sigma_{i,k}\sigma_{j,k})(\widetilde{X}^{0,\widetilde{\xi}}_r,
P_{X^{0,\xi}_r})\big]\bigg)\bigg]dr.\\ \end{array}\ee

\noindent Then it is evident that $V(t,x,P_\xi)$ is continuously
differentiable with respect to $t$,
\be\label{5.11}\begin{array}{lll}
&\partial_t V(t,x,P_\xi)=\displaystyle-E\bigg[\bigg(\sum_{i=1}^d\partial_{x_i}
\Phi(X^{0,x,P_\xi}_{T-t},P_{X^{0,\xi}_{T-t}})b_{i}
(X^{0,x,P_\xi}_{T-t},P_{X^{0,\xi}_{T-t}})\\
&\ \ \ \displaystyle+\frac12\sum_{i,j,k=1}^d
\partial_{x_i x_j}^2\Phi(X^{0,x,P_\xi}_{T-t},P_{X^{0,\xi}_{T-t}})(\sigma_{i,k}
\sigma_{j,k})(X^{0,x,P_\xi}_{T-t},P_{X^{0,\xi}_{T-t}})\\
&\ \ \ \displaystyle+\widetilde{E}\big[\sum_{i=1}^d(\partial_\mu \Phi)_i
(X^{0,x,P_\xi}_{T-t},P_{X^{0,\xi}_{T-t}},\widetilde{X}^{0,
\widetilde{\xi}}_{T-t})b_{i}(\widetilde{X}^{0,\widetilde{\xi}}_{T-t},P_{X^{0,\xi}_{T-t}})\\
&\ \ \ \displaystyle+\frac12\sum_{i,j,k=1}^d\partial_{y_i}\left((\partial_\mu \Phi)_j
(X^{0,x,P_\xi}_{T-t},P_{X^{0,\xi}_{T-t}},\widetilde{X}^{0,\widetilde{\xi}}_{T-t})
\right)(\sigma_{i,k}\sigma_{j,k})(\widetilde{X}^{0,\widetilde{\xi}}_{T-t},P_{X^{0,
\xi}_{T-t}})\big]\bigg)\bigg].\end{array}\ee

\noindent Moreover, using this latter formula, we can now prove in analogy
to the other derivatives of $V$ that $\partial_t V$ satisfies the estimates
stated in this lemma. The proof is complete.\end{proof}

Now we are able to establish and to prove our main result.
\bt We suppose that $\Phi\in C^{2,1}({\mathbb R}^d\times
{\cal P}_2({\mathbb R}^d))$. Then, under Hypothesis (H.2), the function $V(t,x,P_\xi)
=E[\Phi(X^{t,x,P_\xi}_T,P_{X^{t,\xi}_T})],\, (t,x,\xi)\in[0,T]\times {\mathbb R}^d
\times L^2({\cal F}_t)$, is the unique solution in $C^{1,(2,1)}([0,T]\times
{\mathbb R}^d\times {\cal P}_2({\mathbb R}^d))$ of the PDE
\be\label{5.12}\begin{array}{lll}
&\displaystyle0=\partial_tV(t,x,P_\xi)+\sum_{i=1}^d\partial_{x_i}
V(t,x,P_\xi)b_{i}(x,P_\xi)\displaystyle+\frac12\sum_{i,j,k=1}^d
\partial_{x_i x_j}^2V(t,x,P_\xi)(\sigma_{i,k}
\sigma_{j,k})(x,P_\xi)\\
&\displaystyle\ \ \ \ \ +\widetilde{E}\big[\sum_{i=1}^d(\partial_\mu V)_i
(t,x,P_\xi,\widetilde{\xi})b_{i}(\widetilde{\xi},P_\xi)\displaystyle+
\frac12\sum_{i,j,k=1}^d\partial_{y_i}(\partial_\mu V)_j
(t,x,P_\xi,\widetilde{\xi})(\sigma_{i,k}\sigma_{j,k})(\widetilde{\xi},
P_\xi)\big],\\
&\displaystyle\ \ \  \ \ \  \hfill {(t,x,\xi)\in[0,T]\times {\mathbb R}^d\times L^2({\cal F})},\\
&V(T,x,P_\xi)=\Phi(x,P_\xi),\ (x,\xi)\in {\mathbb R}^d\times L^2({\cal F}).\\\end{array}\ee\et

\begin{proof} As before we restrict ourselves in this proof
to the one-dimensional case $d=1.$ Recalling the flow property
\be\label{5.13}\left(X^{s,X^{t,x,P_\xi}_s,P_{X^{t,\xi}_s}}_r,
X^{s,X^{t,\xi}_s}_r\right)=\left(X^{t,x,P_\xi}_r,X^{t,\xi}_r\right),\ 0\le t
\le s\le r\le T,\ x\in {\mathbb R},\ \xi\in L^2({\cal F}_t),\ee

\noindent of our dynamics as well as
\be\label{5.14}V(s,y,P_\vartheta)=E[\Phi(X^{s,y,P_\vartheta}_T, P_{X^{s,\vartheta}_T})]
=E[\Phi(X^{s,y,P_\vartheta}_T, P_{X^{s,\vartheta}_T})|{\cal F}_s],\ s\in[0,T],\ y\in {\mathbb R},\ \vartheta\in L^2({\cal F}_s),\ee

\smallskip

\noindent  we deduce that
\be\label{5.15}\begin{array}{lll}
&V(s,X^{t,x,P_\xi}_s,P_{X^{t,\xi}_s})
=E[\Phi(X^{s,y,P_\vartheta}_T,P_{X^{s,\vartheta}_T})
|{\cal F}_s]_{\big|(y,\vartheta)=(X^{t,x,P_\xi}_s,X^{t,\xi}_s)}\\
&=E[\Phi(X^{s,X^{t,x,P_\xi}_s,P_{X^{t,\xi}_s}}_T,
P_{X^{s,X^{t,\xi}_s}_T})|{\cal F}_s]=E[\Phi(X^{t,x,P_\xi}_T,
P_{X^{t,\xi}_T})|{\cal F}_s],\  s\in[t,T],\end{array}\ee

\smallskip

\noindent i.e., $V(s,X^{t,x,P_\xi}_s,P_{X^{t,\xi}_s}),\,
s\in[t,T],$ is a martingale. On the other hand, since due to Lemma 5.1 the function $V\in C^{1,(2,1)}([0,T]\times
{\mathbb R}^d\times {\cal P}_2({\mathbb R}^d))$ satisfies the regularity assumptions for the It\^{o} formula,
we know that
\be\label{5.16}\begin{array}{lll}
& V(s,X^{t,x,P_\xi}_s,P_{X^{t,\xi}_s})-V(t,x,P_\xi)\\
& =\displaystyle\int_t^s\bigg(\partial_r
V(r,X^{t,x,P_\xi}_r,P_{X^{t,\xi}_r})+\partial_{x}
V(r,X^{t,x,P_\xi}_r,P_{X^{t,\xi}_r})b(X^{t,x,P_\xi}_r,P_{X^{t,\xi}_r})\\
&\displaystyle\ \ \ \ +\frac12
\partial_{x}^2V(r,X^{t,x,P_\xi}_r,P_{X^{t,\xi}_r})\sigma^2
(X^{t,x,P_\xi}_r,P_{X^{t,\xi}_r})\\
&\displaystyle\ \ \ \ +\widetilde{E}\big[\partial_\mu V
(r,X^{t,x,P_\xi}_r,P_{X^{t,\xi}_r},\widetilde{X}^{t,
\widetilde{\xi}}_r)b(\widetilde{X}^{t,\widetilde{\xi}}_r,P_{X^{t,\xi}_r})\\
&\displaystyle\ \ \ \ +\frac12\partial_{y}((\partial_\mu V)
(r,X^{t,x,P_\xi}_r,P_{X^{t,\xi}_r},\widetilde{X}^{t,\widetilde{\xi}}_r)
)\sigma^2(\widetilde{X}^{t,\widetilde{\xi}}_r,P_{X^{t,\xi}_r})\big]
\bigg)dr\\
&\displaystyle\ \ \ \  +\int_t^s\partial_{x}V(r,X^{t,x,
P_\xi}_r,P_{X^{t,\xi}_r})\sigma(X^{t,x,P_\xi}_r,
P_{X^{t,\xi}_r})dB_r,\ s\in[t,T].\end{array}\ee

\noindent Consequently,
\be\label{5.17}\begin{array}{lll} &V(s,X^{t,x,P_\xi}_s,P_{X^{t,\xi}_s})-V(t,x,P_\xi)\\
&=\displaystyle\int_t^s\partial_{x}V(r,X^{t,x,
P_\xi}_r,P_{X^{t,\xi}_r})\sigma(X^{t,x,P_\xi}_r,
P_{X^{t,\xi}_r})dB_r,\ s\in[t,T],\end{array}\ee

\noindent and
\be\label{5.18}\begin{array}{lll} &0=\displaystyle\int_t^s\bigg(\partial_r
V(r,X^{t,x,P_\xi}_r,P_{X^{t,\xi}_r})+\partial_{x}
V(r,X^{t,x,P_\xi}_r,P_{X^{t,\xi}_r})b(X^{t,x,P_\xi}_r,P_{X^{t,\xi}_r})\\
&\displaystyle\ \ \ \ \ +\frac12
\partial_{x}^2V(r,X^{t,x,P_\xi}_r,P_{X^{t,\xi}_r})\sigma^2
(X^{t,x,P_\xi}_r,P_{X^{t,\xi}_r})\\
&\displaystyle\ \ \ \ \ +\widetilde{E}\big[\partial_\mu V
(r,X^{t,x,P_\xi}_r,P_{X^{t,\xi}_r},\widetilde{X}^{t,
\widetilde{\xi}}_r)b(\widetilde{X}^{t,\widetilde{\xi}}_r,P_{X^{t,\xi}_r})\\
&\displaystyle\ \ \ \ \ +\frac12\partial_{y}\left((\partial_\mu V)
(r,X^{t,x,P_\xi}_r,P_{X^{t,\xi}_r},\widetilde{X}^{t,\widetilde{\xi}}_r)
\right)\sigma^2(\widetilde{X}^{t,\widetilde{\xi}}_r,P_{X^{t,\xi}_r})\big]
\bigg)dr,\ s\in[t,T],\end{array}\ee

\noindent from where we obtain easily the wished PDE.

Thus, it only still remains to prove the uniqueness of the solution of
the PDE in the class $C^{1,(2,1)}([0,T]\times {\mathbb R}^d\times{\cal P}_2({\mathbb R}^d))$.
Let $U\in C^{1,(2,1)}([0,T]\times {\mathbb R}^d\times{\cal P}_2({\mathbb R}^d))$ be a solution
of PDE (\ref{5.12}). Then, from the It\^{o} formula we have that
\be\label{5.19}\begin{array}{lll}
&U(s,X^{t,x,P_\xi}_s,P_{X^{t,\xi}_s})-U(t,x,P_\xi) =\displaystyle\int_t^s\partial_{x}U(r,X^{t,x,
P_\xi}_r,P_{X^{t,\xi}_r})\sigma(X^{t,x,P_\xi}_r,
P_{X^{t,\xi}_r})dB_r,\ s\in[t,T],\end{array}\ee

\noindent  is a martingale. Thus, for all $t\in[0,T],\ x\in {\mathbb R}$ and
$\xi\in L^2({\cal F}_t)$,
\be\label{5.20} U(t,x,P_\xi)=E[U(T,X^{t,x,P_\xi}_T,P_{X^{t,\xi}_T})
|{\cal F}_t]=E[\Phi(X^{t,x,P_\xi}_T,P_{X^{t,\xi}_T})]=V(t,x,P_\xi).\ee

\noindent This proves that the functions $U$ and $V$ coincide, i.e.,
the solution is unique in $C^{1,(2,1)}([0,T]\times {\mathbb R}^d\times{\cal
P}_2({\mathbb R}^d))$. The proof is complete.\end{proof}


\begin{thebibliography}{99}

\bibitem{BRTV} {\sc Benachour, S., Roynette, B., Talay, D., Vallois, P.} (1998) {\it Nonlinear selfstabilizing
processes. I: Existence, invariant probability, propagation of chaos.} Stochastic Processes Appl. 75, No.2, 173-201.

\bibitem{BT}    {\sc Bossy, M., Talay, D.} (1997) {\it A stochastic particle method for the McKean-Vlasov and the Burgers equation.} Math. Comput. 66, No.217, 157-192.

\bibitem{BLP1}{\sc Buckdahn, R., Djehiche, B., Li, J., Peng, S.} (2009) {\it Mean-Field Backward
Stochastic Differential Equations. A limit Approach}; Annals of Probability. 37(4), 1524-1565.


\bibitem{BLP} {\sc Buckdahn, R., Li, J., Peng, S.} (2009) {\it Mean-field backward stochastic differential equations and
 related patial differential equations.}  Stochastic Processes and their Applications, 119, 3133-3154.

\bibitem{DG} {\sc Dawson, D. A.; G\"{a}rtner, J.} (1987) {\it Large deviations from the McKean-Vlasov
limit for weakly interacting diffusions.} Stochastics 20, 247-308.

\bibitem{PL}{\sc Cardaliaguet, P.} (2013) {\it Notes on Mean Field Games
(from P.-L. Lions' lectures at Coll\`{e}ge de France)}. Available on the website of Coll\`{e}ge de
France (http://www.college-de-france.fr).

\bibitem{PC}{\sc Cardaliaguet, P.} (2013) {\it\sc Weak solutions for first order mean field games with local coupling}. http://hal.archives-ouvertes.fr/hal-00827957.

\bibitem{CD}{\sc Carmona, R., Delarue, F.} (2013) {\it Forward-backward stochastic differential equations and controlled Mckean-Vlasov dynamics.} http://arXiv:1303.5835v1.

\bibitem{CD1} {\sc Carmona, R., Delarue, F.} (2014) {\it The Master Equation for Large Population Equilibriums.} arXiv:1404.4694v2.

\bibitem{C}{\sc Chan, T.} (1994) {\it Dynamics of the McKean-Vlasov equation.} Annals of Probability. 22(1), 431-441.

\bibitem{Kac1}{\sc Kac, M.} (1956) {\it  Foundations of kinetic theory.} In Proceedings of the 3rd Berkeley Symposium on Mathematical Statistics and Probability, volume 3, 171-197.

\bibitem{Kac2}{\sc Kac, M.} (1958) {\it Probability and Related Topics in the Physical Sciences.} Interscience Publishers, New York.

\bibitem{K}{\sc Kotelenez, P.} (1995) {\it A class of quasilinear stchastic partial differential eqautions of McKean-Vlasov type with mass conservation.} Probab. Theory Relat. Fields. 102, 159-188.


\bibitem{LL}{\sc Lasry, J.M., Lions, P.L.} (2007) {\it Mean field games.} Japan. J.
Math. 2, 229-260. Available online: DOI: 10.1007/s11537-007-0657-8.

\bibitem{M}{\sc M\'{e}l\'{e}ard, S.} (1996) {\it Asymptotic behaviour of some interacting particle systems; Mckean-Vlasov and Blotzmann
models.} In D. Talay and L. Tubaro, editors, Probabilistic Models For Nonlinear PDE's, 42-95, Lectures Notes in Math., 1627, Berlin,
Heidelberg, New York, Springer Verlag.

\bibitem{O}{\sc Overbeck, L.} (1995) {\it Superprocesses and McKean-Vlasov equations with creation of
mass.}

\bibitem{PH}{\sc Pra, P.D., Hollander, F.D.} (1995) {\it McKean-Vlasov limit for interacting random processes in random media.} Journal of Statistical Physics. 84, No.314, 735-772.

\bibitem{S1}{\sc Sznitman, A.S.} (1984) {\it Nonlinear reflecting diffusion processes, and the propagation of chaos and fluctuations associated.}
Journal of functional analysis. 56, 311 - 336.

\bibitem{S2}{\sc Sznitman, A.S.} (1991) {\it Topics in propagation of
chaos.} Lect. Notes in Math, 1464, 165-252, Springer-Verlag, Berlin.

\end{thebibliography}
\end{document}